\documentclass[12pt]{amsart}   
\usepackage{amssymb,hyperref}
\usepackage{tikz}

\theoremstyle{plain}
\newtheorem{theorem}{Theorem}[section]
\newtheorem{Theorem}{Theorem}

\newtheorem{lemma}[theorem]{Lemma}
\newtheorem{proposition}[theorem]{Proposition}
\newtheorem{corollary}[theorem]{Corollary}

\theoremstyle{definition}
\newtheorem{definition}[theorem]{Definition}
\newtheorem{example}[theorem]{Example}
\newtheorem{remark}[theorem]{Remark}
\newtheorem{problem}[theorem]{Problem}

\DeclareMathOperator{\BSub}{BSub}
\DeclareMathOperator{\Sub}{Sub}
\DeclareMathOperator{\Dir}{Pp}
\DeclareMathOperator{\oDir}{Dir}

\DeclareMathOperator{\BShad}{BShad}
\newcommand{\cat}[1]{\ensuremath{\mathbf{#1}}}

\newcommand{\up}{\ensuremath{\mathop\uparrow\nolimits}}
\newcommand{\down}{\ensuremath{\mathop\downarrow\nolimits}}
\newcommand{\mc}{\mathcal}

\newcommand{\tA}{\tilde{A}}

\pgfdeclarelayer{dotlayer}
\pgfdeclarelayer{linelayer}
\pgfsetlayers{main,linelayer,dotlayer}
\makeatletter
\pgfkeys{%
  /tikz/on layer/.code={
    \pgfonlayer{#1}\begingroup
    \aftergroup\endpgfonlayer
    \aftergroup\endgroup
  },
  /tikz/node on layer/.code={
    \gdef\node@@on@layer{%
      \setbox\tikz@tempbox=\hbox\bgroup\pgfonlayer{#1}\unhbox\tikz@tempbox\endpgfonlayer\pgfsetlinewidth{0.7pt}\egroup}
    \aftergroup\node@on@layer
  },
  /tikz/end node on layer/.code={
    \endpgfonlayer\endgroup\endgroup
  }
}
\def\node@on@layer{\aftergroup\node@@on@layer}
\makeatother
\tikzstyle{dot}=[circle, draw=black!70, fill=black!20, inner sep=.25ex, node on layer=dotlayer]
\tikzstyle{Dot}=[circle, draw=black, fill=black!70, inner sep=.25ex, node on layer=dotlayer]
\tikzstyle{thinline}=[-, draw=black!20, line width=.5pt, node on layer=linelayer]
\tikzstyle{line}=[-, draw=black!50, line width=.5pt, node on layer=linelayer]
\tikzstyle{Line}=[-, draw=black, line width=1pt, node on layer=linelayer]
\newcommand{\powerone}[1]{\begin{tikzpicture}[scale=.5,draw=gray!50,inner sep=0pt,outer sep=0pt]
    \node (a0) {$\;\;0\;\;$};
    \node (a1) at ([yshift=1cm]a0) {$\;\;#1\;\;$};
    \draw ([yshift=-2mm]a0.north) to ([yshift=2mm]a1.south);
  \end{tikzpicture}}
\newcommand{\powertwo}[2]{\begin{tikzpicture}[scale=.33,draw=gray!50,inner sep=0pt,outer sep=0pt]
    \node (b0) {$0$};
    \node (b1) at ([yshift=2cm]0) {$1$};
    \node (bx) at ([xshift=-1cm,yshift=1cm]0) {$#1$};
    \node (bx') at ([xshift=1cm,yshift=1cm]0) {$#2$};
    \draw (b0) to (bx) to (b1);
    \draw (b0) to (bx') to (b1);
  \end{tikzpicture}}
\newcommand{\powerthree}[6]{\begin{tikzpicture}[scale=.5,draw=gray!50,inner sep=0pt,outer sep=0pt]
    \node (c0) {$0$};
    \node (c1) at ([yshift=26mm]0) {$1$};
    \node (ca) at ([xshift=-1cm,yshift=8mm]0) {$#1$};
    \node (cb) at ([yshift=8mm]0) {$#2$};
    \node (cc) at ([xshift=1cm,yshift=8mm]0) {$#3$};
    \node (cd) at ([xshift=-1cm,yshift=18mm]0) {$#4$};
    \node (ce) at ([yshift=18mm]0) {$#5$};
    \node (cf) at ([xshift=1cm,yshift=18mm]0) {$#6$};
    \draw (c0) to (ca); \draw (c0) to (cb); \draw (c0) to (cc);
    \draw (c1) to (cd); \draw (c1) to (ce); \draw (c1) to (cf);
    \draw (ca) to (ce); \draw (ca) to (cf);
    \draw (cb) to (cd); \draw (cb) to (cf);
    \draw (cc) to (cd); \draw (cc) to (ce);
  \end{tikzpicture}}

\title{Boolean subalgebras of orthoalgebras}

\author{John Harding}
\address{New Mexico State University}
\email{hardingj@nmsu.edu}

\author{Chris Heunen}
\address{University of Edinburgh}
\email{chris.heunen@ed.ac.uk}

\author{Bert Lindenhovius}
\address{Tulane University}
\email{alindenh@tulane.edu} 

\author{Mirko Navara}
\thanks{C.H. is supported by EPSRC Fellowship EP/L002388/1. B.L. is supported by the
	MURI project `Semantics, Formal Reasoning, and Tool Support for Quantum Programming' sponsored by the U.S. Department of Defense and the U.S. Air Force Office of Scientific Research. M.N. was supported by the Ministry of Education of the Czech Republic under Project RVO13000.}
\address{Department of Cybernetics,
Faculty of Electrical Engineering,
Czech Technical University in Prague}
\email{navara@cmp.felk.cvut.cz}

\begin{document}

\begin{abstract}
  We develop a direct method to recover an orthoalgebra from its poset of Boolean subalgebras. For this a new notion of direction is introduced. Directions are also used to characterize in purely order-theoretic terms those posets that are isomorphic to the poset of Boolean subalgebras of an orthoalgebra. These posets are characterized by simple conditions defining orthodomains and the additional requirement of having enough directions. Excepting pathologies involving maximal Boolean subalgebras of four elements, it is shown that there is an equivalence between the category of orthoalgebras and the category of orthodomains with enough directions with morphisms suitably defined. Furthermore, we develop a representation of orthodomains with enough directions, and hence of orthoalgebras, as certain hypergraphs. This hypergraph approach extends the technique of Greechie diagrams and resembles projective geometry. Using such hypergraphs, every orthomodular poset can be represented by a set of points and lines where each line contains exactly three points. 
\end{abstract}

\maketitle

\section{Introduction}

\emph{Contextuality} is the phenomenon in quantum theory that the outcome of a measurement may depend on the context in which that measurement is made, that is, on the experimental implementation of that measurement. This principle prevents hidden-variable explanations and clarifies why deterministic explanations of quantum theory are impossible. Contextuality is concerned with how various parts, that are locally consistent, fit together globally. This work is centered around the idea that ``the shape of how the parts fit together'' is enough to determine the whole. The (contents of the) parts themselves are not necessary.

Here, we treat contextuality algebraically: the quantum system is modeled as an algebra, and measurement contexts are modeled as certain subalgebras. We consider \emph{orthoalgebras}~\cite{Greechie}, which are certain structures with a partially defined binary operation $\oplus$ called orthogonal sum, a unary operation $'$ called orthocomplementation, and constants $0,1$. This includes Boolean algebras, orthomodular lattices, and orthomodular posets. For an example, see Figure~\ref{fig:exampleorthodomain}. The appropriate notion of a measurement context then is a \emph{Boolean subalgebra}: a subset that is closed under the operations and is induced by restricting the join of a Boolean algebra to orthogonal elements. See for example Figures~\ref{fig:exampleorthodomain} and~\ref{fig:booleanorthodomains}.
Thus our main object of study is the partially ordered set $\BSub(A)$ of Boolean subalgebras $B \subseteq A$ of an orthoalgebra $A$; we call this its \emph{orthodomain}.

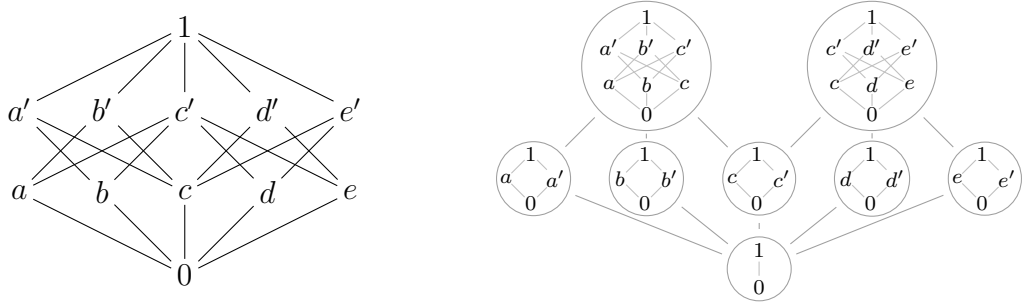
\begin{figure}[h!]
\hfill$\begin{aligned}\begin{tikzpicture}[inner sep=2pt, scale=1.1]
 \node (0) at (0,0) {$0$};
 \node (1) at (0,3) {$1$};
 \node (a) at (-2,1) {$a$};
 \node (b) at (-1,1) {$b$};
 \node (c) at (0,1) {$c$};
 \node (d) at (1,1) {$d$};
 \node (e) at (2,1) {$e$};
 \node (a') at (-2,2) {$a'$};
 \node (b') at (-1,2) {$b'$};
 \node (c') at (0,2) {$c'$};
 \node (d') at (1,2) {$d'$};
 \node (e') at (2,2) {$e'$};
 \draw (0) to (a); \draw (0) to (b); \draw (0) to (c); \draw (0) to (d); \draw (0) to (e);
 \draw (1) to (a'); \draw (1) to (b'); \draw (1) to (c'); \draw (1) to (d'); \draw (1) to (e');
 \draw (a) to (b'); \draw (a) to (c');
 \draw (b) to (a'); \draw (b) to (c');
 \draw (c) to (a'); \draw (c) to (b'); \draw (c) to (d'); \draw (c) to (e');
 \draw (d) to (c'); \draw (d) to (e'); 
 \draw (e) to (c'); \draw (e) to (d');
\end{tikzpicture}\end{aligned}
\qquad\qquad
\begin{aligned}\begin{tikzpicture}[font=\tiny,scale=1.5,inner sep=-3pt,outer sep=3pt, draw=gray!75]
 \node[circle,draw] (0) at (0,0) {$\powerone{1}$};
 \node[circle,draw] (a) at (-2,.8) {$\powertwo{a}{a'}$};
 \node[circle,draw] (b) at (-1,.8) {$\powertwo{b}{b'}$};
 \node[circle,draw] (c) at (0,.8) {$\powertwo{c}{c'}$};
 \node[circle,draw] (d) at (1,.8) {$\powertwo{d}{d'}$};
 \node[circle,draw] (e) at (2,.8) {$\powertwo{e}{e'}$};
 \node[circle,draw] (l) at (-1,1.8) {$\powerthree{a}{b}{c}{a'}{b'}{c'}$};
 \node[circle,draw] (r) at (1,1.8) {$\powerthree{c}{d}{e}{c'}{d'}{e'}$};
 \draw (0) to (a); \draw (0) to (b); \draw (0) to (c); \draw (0) to (d); \draw (0) to (e);
 \draw (a) to (l); \draw (b) to (l); \draw (c) to (l);
 \draw (c) to (r); \draw (d) to (r); \draw (e) to (r);
\end{tikzpicture}\end{aligned}$\hfill\,
\caption{The orthoalgebra on the left is constructed from gluing together two Boolean algebras $\{0,a,b,c,a',b',c',1\}$ and $\{0,c,d,e,c',d',e',1\}$. The orthogonal sum $\oplus$ is the union of the orthogonal join operations of these Boolean algebras. 
On the right is the Hasse diagram of its poset of Boolean subalgebras.}
\label{fig:exampleorthodomain}
\end{figure}

This fits in the established mathematical pattern where some collection of substructures of a structure plays a key role: in classical logic, the collection of subsets of a set; in probability theory, the measurable subsets of a measurable space; in intuitionistic logic, the open subsets of a topological space; in projective geometry, the subspaces of a vector space; and in quantum theory, the collection of closed subspaces of a Hilbert space. In the recent topos-theoretic approach to quantum mechanics \cite{Isham}, the poset of abelian subalgebras of a von Neumann algebra are the central ingredient used to treat contextuality. This latter example is the origin of this work. 

Let us emphasize again that, if an orthoalgebra $A$ is the `whole', we merely consider the `shape' $\BSub(A)$ of how the `parts' $B \subseteq A$ fit together, and not the internal structure of the `parts' $B$ as Boolean algebras. This is like a jigsaw puzzle that can be solved by finding out how the pieces fit without relying on the pictures on the pieces. Our first main result is the following reconstruction.

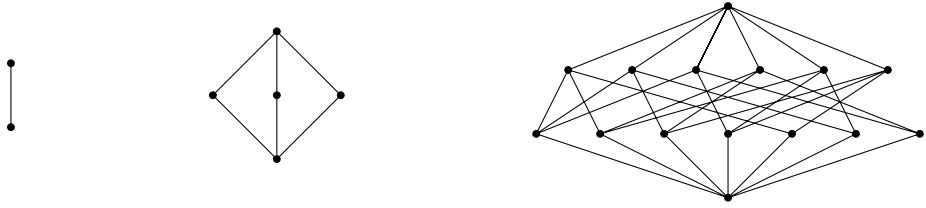
\begin{figure}
\hfill
\begin{tikzpicture}[scale=.85]
\node at (0,0) { };
\draw[fill] (0,1) circle (1.5pt); \draw[fill] (0,2) circle (1.5pt); 
\draw[thin] (0,1)--(0,2);
\end{tikzpicture}
\hfill
\begin{tikzpicture}[scale=.85]
\node at (0,0) { };
\draw[fill] (0,0.5) circle (1.5pt); \draw[fill] (0,2.5) circle (1.5pt); 
\draw[fill] (-1,1.5) circle (1.5pt); \draw[fill] (0,1.5) circle (1.5pt); \draw[fill] (1,1.5) circle (1.5pt);
\draw[thin] (0,0.5)--(-1,1.5)--(0,2.5)--(0,1.5)--(0,0.5)--(1,1.5)--(0,2.5);
\end{tikzpicture}
\hfill
\begin{tikzpicture}[scale=.85]
\draw[fill] (0,0) circle (1.5pt); \draw[fill] (0,3) circle (1.5pt); 
\draw[fill] (-3,1) circle (1.5pt); \draw[fill] (-2,1) circle (1.5pt); \draw[fill] (-1,1) circle (1.5pt);\draw[fill] (0,1) circle (1.5pt);\draw[fill] (1,1) circle (1.5pt);\draw[fill] (2,1) circle (1.5pt);\draw[fill] (3,1) circle (1.5pt);
\draw[fill] (-2.5,2) circle (1.5pt);\draw[fill] (-1.5,2) circle (1.5pt);\draw[fill] (-.5,2) circle (1.5pt);\draw[fill] (.5,2) circle (1.5pt);\draw[fill] (1.5,2) circle (1.5pt);\draw[fill] (2.5,2) circle (1.5pt);
\draw[thin] (0,0)--(-3,1)--(-2.5,2)--(-2,1)--(0,0)--(-1,1)--(-1.5,2)--(-3,1)--(-.5,2)--(0,1)--(1.5,2)--(-2,1)--(.5,2)--(-1,1)--(2.5,2)--(0,1)--(0,0)--(1,1)--(-2.5,2)--(0,3)--(-1.5,2)--(2,1)--(0,0)--(3,1)--(-.5,2)--(0,3)--(-.5,2)--(0,3)--(.5,2)--(3,1);
\draw[thin] (1.5,2)--(0,3)--(2.5,2)--(1,1);\draw[thin] (2,1)--(1.5,2);
\end{tikzpicture}
\hfill\,
\caption{Hasse diagrams of the posets $\BSub(A)$ for Boolean algebras $A$ with 4, 8, and 16 elements. 
A Boolean algebra with 4 elements has two subalgebras: $\{0,1\}$, and itself. 
A Boolean algebra $A$ with 8 elements has five subalgebras: $\{0,1\}$, three subalgebras $\{0,a,a',1\}$ for $a \in A \setminus\{0,1\}$, and $A$ itself. 
A Boolean algebra with 16 elements has a more complicated structure of subalgebras, containing subalgebras with 2, 4, 8, and 16  elements.}
\label{fig:booleanorthodomains}
\end{figure}

\begin{Theorem}\label{firstmainresult}
  If $A$ is a proper orthoalgebra, $\BSub(A)$ has enough directions and $\oDir(\BSub(A))$ is an orthoalgebra isomorphic to $A$.
\end{Theorem}

The main ingredient in this reconstruction is the new notion of a \emph{direction}. Call an element of a poset \emph{basic} if it is of height 0 or 1. Then the basic elements in $\BSub(A)$ are the subalgebras $\{0,a,a',1\}$ for $a \in A$. In each Boolean subalgebra $B$ covering $\{0,a,a',1\}$ in $\BSub(A)$, we can consider the subalgebra $\down a\cup\up a'$. This will be equal to either $\{0,a,a',1\}$ or to $B$ depending on which of $a$ or $a'$ is basic in $B$. A direction assigns a consistent choice of this to each cover of $\{0,a,a',1\}$. 
If $A$ is \emph{proper}, in that its maximal Boolean subalgebras have more than 4 elements, each basic element $\{0,a,a',1\}$ in $\BSub(A)$ will have exactly 2 directions, and these serve the role of $a$ and $a'$ in an isomorphic copy of the given orthoalgebra built from the directions. This is what is meant by having \emph{enough directions}. Theorem~\ref{firstmainresult} somewhat resembles
the reconstruction of a sober topological space from the points of its lattice of open sets.

Compare this to related results. There has been considerable work showing that $\BSub(A)$ determines $A$ in the setting of Boolean algebras~\cite{Sachs}, and orthomodular posets~\cite{HardingNavara}. Similar to $\BSub(A)$ is the poset $\mathrm{CSub}(A)$ of commutative subalgebras of some operator algebra $A$. This poset determines the Jordan structure of $A$ for von Neumann algebras~\cite{Doring}, or for various classes of C*-algebras~\cite{Firby,Hamhalter,Hamhalter2,Lindenhovius}. However, these results are all of the following nature: if $\BSub(A)$ and $\BSub(A')$ are isomorphic (or in the analytic case, if $\mathrm{CSub}(A)$ and $\mathrm{CSub}(A')$ are isomorphic), then there exists a (Jordan) isomorphism between $A$ and $A'$. Even when $A$ is a Boolean algebra, the only known method to reconstruct $A$ from $\BSub(A)$ is indirect, via a family of colimits in the category of Boolean algebras~\cite{Gratzer,Heunen}. Theorem~\ref{firstmainresult} is a concrete, direct, reconstruction of $A$ from the poset $\BSub(A)$.

The second main result of this paper is a characterization of the partially ordered sets of the form $\BSub(A)$ for an orthoalgebra $A$. For Boolean algebras $A$ such a characterization is known~\cite{Gratzer}: the posets $\BSub(A)$ are those algebraic lattices where the principal downset of each compact element is a partition lattice.
Such lattices are called \emph{Boolean domains}. We extend this to the quantum setting of orthoalgebras. For orthoalgebras $A$ we identify several basic properties of $\BSub(A)$, such as having Boolean domains as principal downsets. We call such posets \emph{orthodomains}. It would take complex combinatorics to characterize in elementary terms the orthodomains of the form $\BSub(A)$ from some orthoalgebra $A$. We sidestep this issue by characterizing them as the orthodomains with enough directions, just like lattices of open sets of topological spaces are characterized as frames with enough points.

\begin{Theorem}\label{secondmainresult}
  An orthodomain $X$ is of the form $\BSub(A)$ for a proper orthoalgebra $A$ if and only if it is tall and has enough directions, and in that case $X \simeq \BSub(\oDir(X))$.  
\end{Theorem}

Tallness is a condition requiring the existence of certain joins. The structure of a tall orthodomain with enough directions is fundamentally determined by its elements of height at most 3. We call an orthodomain short if all of its elements have height at most 3. Each tall orthodomain with enough directions can be truncated to a short orthodomain with enough directions, and each short orthodomain with enough directions can be uniquely extended to a tall orthodomain with enough directions. 

Short orthodomains can be efficiently described via certain \emph{hypergraphs}. A hypergraph $\mathcal{H}$ consists of a set $P$ of points; a set $L$ of lines where each line consists of three distinct points; and a set $T$ of planes where each plane consists of 6 lines in a particular configuration similar to a Fano plane. These correspond to posets of subalgebras of 4, $8$, and $16$-element Boolean algebras, see Figure~\ref{fig:hypergraphs}. Those hypergraphs arising as the hypergraphs of orthoalgebras are called \emph{orthohypergraphs} and can be characterized in terms of their having enough directions. We also show that for orthomodular posets, the orthodomain is in fact already fundamentally determined at height 2, so we may forget about planes and fruitfully think in terms of \emph{projective geometry}, retaining only points and lines, where each line consists of exactly three points. In effect, in the hypergraph of an orthomodular poset, any configuration of points and lines that looks like a plane is a plane. 

\begin{figure}[h]
\begin{center}
\begin{tikzpicture}
\draw[fill] (-6,.6) circle(.75pt); \draw[fill] (-4,.6) circle(.75pt);\draw[fill] (-3,.6) circle(.75pt);\draw[fill] (-2,.6) circle(.75pt);
\draw[thin] (-4,.6)--(-2,.6);
\node at (-6,-.5) {point}; \node at (-3,-.5) {line}; \node at (1,-.5) {plane};
\draw[fill] (0,0) circle(.75pt); \draw[fill] (1,0) circle(.75pt); \draw[fill] (2,0) circle(.75pt); \draw[fill] (1,1.414) circle(.75pt); \draw[fill] (1,.48) circle(.75pt); \draw[fill] (.5,.707) circle(.75pt); \draw[fill] (1.5,.707) circle(.75pt); 
\draw[thin] (0,0)--(2,0)--(1,1.414)--(0,0)--(1.5,.707); \draw[thin] (1,1.414)--(1,0); \draw[thin] (.5,.707)--(2,0);
\end{tikzpicture}
\end{center}
\caption{Hypergraphs of the examples from Figure~\ref{fig:booleanorthodomains}.}
\label{fig:hypergraphs}
\end{figure}
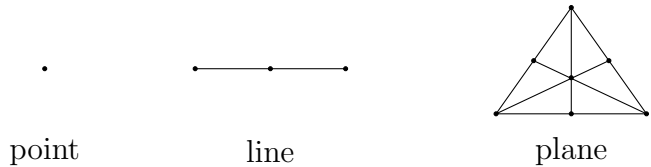

The standard method to represent orthomodular posets and orthoalgebras is via \emph{Greechie diagrams}~\cite{GreechieD,N:handQL}, which are also hypergraphs but of a different sort. Both methods are preferable to Hasse diagrams, which have no convenient way to indicate orthocomplements or orthogonal joins, and are completely intractable in all but the simplest cases. Greechie diagrams will generally be much smaller than our hypergraphs, but the price to pay is that they hide much of the structure, often in a way that is very difficult to understand. Furthermore, Greechie diagrams apply only to chain-finite orthomodular posets and orthoalgebras, while hypergraphs apply to arbitrary ones, even ones without atoms. Finally, as we discuss next, hypergraphs allow one to deal also with morphisms. For Greechie diagrams there are no such results about morphisms, and it seems highly problematic. Thus our new hypergraph representation can be an effective addition to the toolbox of working with orthoalgebras, and seems to capture the essence of their contextuality.

By Theorems~\ref{firstmainresult} and~\ref{secondmainresult} we may work with $\BSub(A)$ or even its hypergraph directly, instead of with $A$ itself, without losing information. The third main result of this paper is an investigation of functorial aspects of the reconstruction. We define morphisms between orthohypergraphs as certain partial functions that map points to points. We do not obtain a full categorical equivalence, and have to make some exceptions because the Boolean algebra with 1 element and the Boolean algebra with 2 elements have the same orthodomain. Similarly, the Boolean algebra with 4 elements has 2 automorphisms, whereas its orthodomain has only 1. However, the fundamental problems lie only with such small pathologies. Using the term {\em proper} to restrict to cases where 4-element maximal Boolean subalgebras do not play a role, we prove the following. 

\begin{Theorem}\label{thirdmainresult}
  The functor that assigns to each orthoalgebra its orthohypergraph is essentially surjective and injective on non-trivial objects, and full and faithful with respect to proper maps. 
\end{Theorem}

We proceed as follows. Section~\ref{sec:Booleandirections} starts by describing directions in the Boolean setting, and Section~\ref{sec:orthoalgebras} introduces orthoalgebras and their Boolean subalgebras. Section~\ref{sec:directions} generalizes directions to orthoalgebras and proves Theorem~\ref{firstmainresult}. Directions are used again in Section~\ref{sec:orthodomains} to prove Theorem~\ref{secondmainresult}. Section~\ref{sec:hypergraphs} introduces the hypergraph representation of an orthoalgebra, and Section~\ref{sec:morphisms} considers morphisms to prove Theorem~\ref{thirdmainresult}. Section~\ref{sec:conclusion} concludes.

\section{Subalgebras of Boolean algebras and their directions}\label{sec:Booleandirections}

We use standard terminology for partially ordered sets, as in \textit{e.g.}~\cite{Burris}. In particular, for an element $x$ of a partially ordered set $X$, denote its principal ideal and principal filter by
\[
\down x=\{w\in X \mid w\leq x\}
\quad\text{ and }\quad 
\up  x=\{y\in X \mid x\leq y\}\text.
\]

\begin{definition}
Write $\Sub(B)$ for the set of Boolean subalgebras of a Boolean algebra $B$ partially ordered by inclusion with $\bot$ its least element and $\top$ its largest element.
\end{definition}

Since the intersection of Boolean subalgebras is a subalgebra, $\Sub(B)$ is a complete lattice, and since finitely generated Boolean algebras are finite, the compact elements of this lattice are the finite Boolean subalgebras of $B$. Here we recall that an element $x$ of a partially ordered set $X$ with directed joins is \emph{compact} if $x \leq \bigvee Y$ for a directed subset $Y \subseteq X$ implies that $x \leq y$ for some $y \in Y$. Since every Boolean algebra is the union of its finite subalgebras, $\Sub(B)$ is an algebraic lattice.  The algebraic lattices of the form $\Sub(B)$ were characterized by Gr\"atzer \textit{et.\ al.}~\cite{Gratzer} as follows. Here we recall that a partition lattice is a lattice that is isomorphic to the lattice of the partitions of a set. 

\begin{theorem}\label{hui}\cite{Gratzer}
A poset $X$ is isomorphic to $\Sub(B)$ for a Boolean algebra $B$ if and only if the following two conditions hold:
\begin{enumerate}
\item $X$ is an algebraic lattice;
\item $\down x$ is a finite partition lattice for each compact element $x$ of $X$.
\end{enumerate}
\end{theorem}
We call such lattices \emph{Boolean domains}. 

\begin{definition}
A subalgebra of a Boolean algebra is called a \emph{(principal) ideal subalgebra} when it is of the form $I \cup I'$ for a (principal) ideal $I$, where $I'=\{a' \mid a \in I\}$.
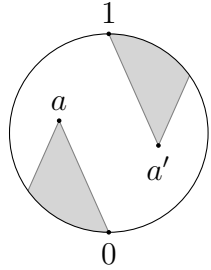
\[\begin{tikzpicture}[scale=.66]
\begin{scope}
\clip (0,0) circle (2cm);
\draw[fill=gray!33, draw=gray] (-1,.25) to (-2,-2) to (0,-2) to (-1,.25); 
\draw[fill=gray!33, draw=gray] (1,-.25) to (2,2) to (0,2) to (1,-.25); 
\end{scope}
\draw (0,0) circle (2cm);
\draw[fill] (0,-2) circle (1pt); \node at (0,-2) [below] {$0$};
\draw[fill] (0,2) circle (1pt); \node at (0,2) [above] {$1$};
\draw[fill] (-1,.25) circle (1pt); \node at (-1,.25) [above] {$a$};
\draw[fill] (1,-.25) circle (1pt); \node at (1,-.25) [below] {$a'$};
\end{tikzpicture}\]
\end{definition}

Principal ideal subalgebras are of the form $\down a \cup \up a'$ for $a\in B$. They will play a central role throughout the paper. To describe their use, we begin with the order-theoretic characterization of ideal subalgebras given by Sachs~\cite{Sachs}.

\begin{definition}
An element $x$ of a lattice is \emph{dual modular} if $(x \vee y) \wedge z = x \vee (y \wedge z)$ for each~$z$ with $x\leq z$ and $(w \vee x) \wedge y = w \vee (x \wedge y)$ for each $y$ with $w\leq y$.
\end{definition}

\begin{lemma}\cite[Theorem~1]{Sachs}\label{dualmodular}
The dual modular elements of $\Sub(B)$ are the ideal subalgebras.
\end{lemma}

The least element $\bot$ of the Boolean domain $\Sub(B)$ is $\{0,1\}$, the largest element $\top$ is $B$, and the atoms of $\Sub(B)$ are the elements $\{0,a,a',1\}$ for $a\neq 0,1$. Hence there is a bijection between complementary pairs $\{a,a'\}$ in $B$ and elements of $\Sub(B)$ that are either $\bot$ or an atom. 

\begin{definition}
Call an element of a poset with a least element \emph{basic} if it is either an atom or the least element. 
\end{definition}

\begin{definition}
For an element $a$ of a Boolean algebra, we denote the Boolean subalgebra
$$x_a=\{0,a,a',1\}\,.$$
\end{definition}

Later on, we shall use the same notation also when $a$ is an element of an orthoalgebra.

\begin{lemma}\label{sbi}
For a Boolean algebra $B$, the basic elements of $\Sub(B)$ that are dual modular are $x_a
$ where either $a,a'$ is basic. In fact, they are principal ideal subalgebras.
\end{lemma}
\begin{proof}
Follows immediately from Lemma~\ref{dualmodular}.
\end{proof}

Our key definition is the following:

\begin{definition} \label{phi}
For $B$ a Boolean algebra, we define the mapping $\varphi\colon B\to(\Sub(B))^2$ by
$$\varphi(a)=(\down a \cup \up a', \down a' \cup \up a)\,.$$ 
We call $\varphi(a)$ the \emph{principal pair} corresponding to $a$.
\end{definition}

We call a Boolean algebra \emph{small} if it has at most $4$ elements. Our aim is to show that if $B$ is not small, than $\varphi$ is one-to-one, and to characterize the range of $\varphi$ in purely order-theoretical terms. This will allow us to reconstruct an isomorphic copy of $B$ from the poset $\Sub(B)$.
We formulate this for a general Boolean domain rather than a special case of $\Sub(B)$ for a Boolean algebra~$B$, although all are isomorphic to such.

\begin{definition}
\label{pp}
Let $X$ be a Boolean domain.
A~\emph{principal pair} of $X$ is an ordered pair $(y,z)$ of dual modular elements of~$X$ that satisfies one of the following conditions:
\begin{enumerate}
\item $y=\top$, $z$ is a 
basic element;
\item $z=\top$, $y$ is a 
basic element;
\item $y\vee z=\top$ and $y\wedge z$ is a basic element which is not dual modular.
\end{enumerate}
We say that a principal pair $(y,z)$ (as well as $(z,y)$) is a principal pair for the basic element~$x$ if $y\wedge z=x$.
Write $\Dir(X)$ for the set of principal pairs of~$X$. 
\end{definition}

\begin{remark} \label{smallB}
Notice that in Definition~\ref{pp} the element $y\wedge z=x$ is always basic and it is dual modular iff case (1) or (2) applies. Therefore, if $x$ is a dual modular basic element, then $(x,\top)$ and $(\top,x)$ are the only principal pairs for~$x$. 

If $B=\{0,1\}$, then $\bot=\top$ and $\varphi[B]=\Dir(\Sub(B))=\{(\top,\top)\}$.
If $B$ has $4$ elements, then $B=x_a
=\top$, $a\notin\{0,1\}$, and $\varphi[B]=\Dir(\Sub(B))=\{(\bot,\top), (\top,\bot), (\top,\top)\}$, where $(\top,\top)=\varphi(a)=\varphi(a')$.
A principal pair $(y,z)$ satisfies $y=z$ only if $B$ is small and $y=z=\top$.
\end{remark}

\begin{proposition} \label{prop:two-directions}
Let $B$ be a Boolean algebra, $a\in B$, and $x_a
$ be the corresponding basic element of $\Sub(B)$. For the map $\varphi$ of Definition~\ref{phi}, $\varphi(a)$ and $\varphi(a')$ are the only principal pairs for $x_a$, and if $B$ is not small these are distinct. So the image $\varphi[B]$ is $\Dir(\Sub(B))$, and if $B$ is not small then $\varphi\colon B\to\Dir(\Sub(B))$ is a bijection. 
\end{proposition}

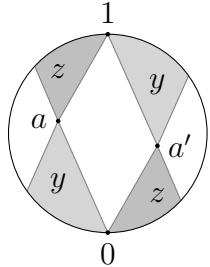
\[\begin{tikzpicture}[scale=.66]
\begin{scope}
\clip (0,0) circle (2cm);
\draw[fill=gray!33, draw=gray] (-1,.25) to (-2,-2) to (0,-2) to (-1,.25); \node at (-1,-1) {$y$};
\draw[fill=gray!33, draw=gray] (1,-.25) to (2,2) to (0,2) to (1,-.25); \node at (1,1) {$y$};
\draw[fill=gray!50, draw=gray] (-1,.25) to (-1.75,2) to (0,2) to (-1,.25); \node at (-1,1.25) {$z$};e
\draw[fill=gray!50, draw=gray] (1,-.25) to (1.75,-2) to (0,-2) to (1,-.25); \node at (1,-1.25) {$z$};
\end{scope}
\draw (0,0) circle (2cm);
\draw[fill] (0,-2) circle (1pt); \node at (0,-2) [below] {$0$};
\draw[fill] (0,2) circle (1pt); \node at (0,2) [above] {$1$};
\draw[fill] (-1,.25) circle (1pt); \node at (-1,.25) [left] {$a$};
\draw[fill] (1,-.25) circle (1pt); \node at (1,-.25) [right] {$a'$};
\end{tikzpicture}\]

\begin{proof}
If $a$ is basic then $\varphi(a)=(x_a,\top)$, and if $a'$ is basic then $\varphi(a)=(\top,x_a)$. In both cases, $x_a$~is a dual modular basic element. 
So $\varphi(a),\varphi(a')$ are the two possible principal pairs for~$x_a$. If $B$ is not small, they are distinct (Remark~\ref{smallB}).

Assume that $a$ is not $0$, $1$, an atom, or a coatom. 
So there are $b,c$ with $0<b<a<c<1$ (and $B$ is not small). Then $x_a$ is a basic element that is not dual modular. Let $y=\down a\cup\up a'$ and $z=\down a'\cup\up a$. It is clear that $y,z$ are ideal subalgebras, and hence are dual modular. Also $y\wedge z=x_a$. For any $e\in B$ we have $e=(e\wedge a)\vee(e\wedge a')$. Since $e\wedge a \in y$ and $e\wedge a'\in z$, then $e$ is in the subalgebra $y\vee z$ generated by $y,z$. So $y\vee z = \top$. Hence $(y,z)$ and $(z,y)$ are principal pairs for~$x_a$. Since $b\in y$ and $b\notin z$, these principal pairs are distinct. 

We now show these are the only principal pairs for~$x_a$. Suppose that $(v,w)$ is a principal pair for~$x_a$. Since $v,w$ are dual modular, they are ideal subalgebras. So $v=I\cup I'$ and $w=J\cup J'$ for ideals $I,J \subseteq B$. Now $x_a=v\wedge w$ gives 
\[ 
x_a=\{0,a,a',1\} = (I\cap J)\cup (I\cap J')\cup (I'\cap J)\cup (I'\cap J')
\]
It cannot be the case that $a\in I\cap J$ since then $b \in I\cap J$ because $I$ and $J$ are ideals, and similarly $a\notin I'\cap J'$. So one of $a,a'$ belongs to $I\cap J'$ and the other to $I'\cap J$. Say $a\in I\cap J'$. Since $J'$ is a filter, there cannot be an element of $I$ other than $1$ that is larger than $a$ since it would belong to $I\cap J'$, and since $I$ is an ideal and $a<c<1$ it cannot be that $1\in I$ since this would imply that $c\in I$. So $a$ is the largest element of $I$, and similarly it is the least element of $J'$. So $(v,w)=(y,z)$. If $a\in I'\cap J$, then by symmetry $(v,w)=(z,y)$. 

We have shown for any $a\in B$ that $\varphi(a),\varphi(a')$ are principal pairs for~$x_a
$, they are the only principal pairs for $x_a$, and that these are distinct if $B$ is not small. Since every principal pair of $\Sub(B)$ is a principal pair for some basic element and all basic elements arise as $x_a
$ for some $a\in B$, this shows that $\varphi$ is onto. If $a,b\in B$ and $\varphi(a)=\varphi(b)$, then $\varphi(a)$ and $\varphi(b)$ are principal pairs for the same basic element. Then if $B$ is not small, $a=b$.
\end{proof}

As every Boolean domain is isomorphic to $\Sub(B)$ for some Boolean algebra $B$, and a Boolean domain with more than two elements is isomorphic to $\Sub(B)$ for some $B$ that is not small, Proposition~\ref{prop:two-directions} and Remark~\ref{smallB} give the following:

\begin{corollary}\label{two}
If $X$ is a Boolean domain, then each basic element has at most two principal pairs, and if $X$ has more than two elements, each basic element has exactly two principal pairs. In the case where $X$ has two or fewer elements, each element is basic, all elements other than $\top$ have two directions, and $\top$ has a single direction. 
\end{corollary}

How do we incorporate the Boolean algebra structure in our considerations? If $\varphi(a)=(y,z)$, then $\varphi(a')=(z,y)$. The partial ordering of $\Dir(\Sub(B))$ is more subtle. If $\varphi(a)=(y_1,z_1)$ and $\varphi(b)=(y_2,z_2)$, then $a\le b$ implies $y_1 \le y_2$, and $z_1 \ge z_2$. But these relationships can hold without $a\leq b$. If $a$ is an atom of $B$, so $y_1=\{0,a,a',1\}$ is a dual modular atom of $X$, then $\varphi(a)=(y_1,\top)$ and  $\varphi(a')=(\top,y_1)$, but of course $a\not\leq a'$. The following definition excludes this situation.
Notice that this is the only situation requiring an exception.
Suppose that $y_1 \le y_2$, $z_1 \ge z_2$, and $a\not\le b$.
We can easily exclude the cases when $y_1\wedge z_1$ is not a dual modular atom or when $y_1=\top$.
In the remaining case, $z_1=\top$ and $a$ is an atom.
Thus $a\not\le b$ means that $a,b$ are orthogonal, i.e., $a\le b'$.
If $a<b'$, then $y_2=\down b \cup \up b'$ does not contain~$a$; a contradiction with $y_1 \le y_2$.
The case of $b=a'$ ($y_2=z_1$, $z_2=y_1$) is the only one which needs to be forbidden.

\begin{definition}\label{swop}
Let $X$ be a Boolean domain. Define a unary operation $\,'$ on $\Dir(X)$ by $$(y,z)' = (z,y)\,.$$ Define a binary relation $\le$ on $\Dir(X)$ by $(y_1,z_1) \leq (y_2,z_2)$ when $y_1 \le y_2$, $z_1 \ge z_2$, and, additionally, if $y_1\wedge z_1$ is a dual modular atom, then $(y_2,z_2)\neq (z_1,y_1)$. 
\end{definition}

\begin{proposition}\label{lpo}
For a Boolean algebra $B$ that is not small, $\varphi$ is an order isomorphism that preserves~$\,'$.
\end{proposition}
\begin{proof}
The complement $\,'$ of $\Dir(\Sub(B))$ commutes with~$\varphi$.

Suppose $a,b\in B$, $a \leq b$. 
By Proposition~\ref{prop:two-directions}, $\varphi(a)$ is a principal pair for $x=\{0,a,a',1\}$. 
Observe $\down a \cup \up a' \subseteq \down b \cup \up b'$ and $\down a' \cup \up a \supseteq\down b' \cup \up b$. 
This suffices for $\varphi(a)\leq\varphi(b)$ unless $x$ is a dual modular atom.
If $x$ is a dual modular atom, then, by Lemma~\ref{sbi}, $a$ or $a'$ is an atom. 
Then $b \neq a'$, so $\varphi(b) \neq \varphi(a')= \varphi(a)'$, and hence $\varphi(a)\leq\varphi(b)$.

Finally, suppose $\varphi(a)\leq\varphi(b)$. We will show $a \leq b$ by contradiction; suppose $b' \not\leq a'$. Again $\down a \cup \up a' \subseteq \down b \cup \up b'$ and $\down a' \cup \up a \supseteq\down b' \cup \up b$. It follows that $a'\in\down b$ and $b\in\down a'$. So $a'\leq b$ and $b\leq a'$, giving $a'=b$. Since $\varphi(a)\leq\varphi(b)=\varphi(a)'$, the definition of $\leq$ implies that $x=\{0,a,a',1\}$ cannot be a dual modular atom of $\Sub(B)$. Hence neither of $a,a'$ is an atom of $B$. Since $a\nleq b$ we have $a\neq 0$, so there is $c$ such that $0<c<a$. Then $c\in\down a\cup\up a'$, but $c\notin\down a'\cup\up a=\down b\cup\up b'$, a contradiction. 
\end{proof}

For a Boolean algebra $B$ that is small, $\varphi$ preserves $\leq$ and~$\,'$, but it is not a bijection, see the proof of Proposition~\ref{prop:two-directions}.

\begin{theorem}\label{thm:booleanobjects}
For a Boolean algebra $B$ with more than $4$ elements, and a Boolean domain $X$ with more than $2$ elements: 
\begin{enumerate}
\item $\Sub(B)$ is a Boolean domain;
\item $\Dir(X)$ is a Boolean algebra;
\item $B$ is isomorphic to $\Dir(\Sub(B))$;
\item $X$ is isomorphic to $\Sub(\Dir(X))$.
\end{enumerate}
\end{theorem}

\begin{proof}
Part (1) follows from Theorem~\ref{hui} (even without any limitation of the number of elements of~$B$). Part (3) follows from Proposition~\ref{lpo}. Since $X$ is a Boolean domain with more than $2$ elements, Theorem~\ref{hui} provides a Boolean algebra $A$ with more than $4$ elements with $X\simeq\Sub(A)$, so $\Dir(X)\simeq\Dir(\Sub(A))$, which is Boolean by (3), establishing part (2). To prove part (4), say $X\simeq\Sub(A)$ for a Boolean algebra $A$; then part (3) gives $\Sub(\Dir(X))\simeq\Sub(\Dir(\Sub(A)))\simeq\Sub(A)\simeq X$.
\end{proof}

For a Boolean algebra $B$ and $a\in B$, consider the Boolean subalgebras of $B$ that contain $a$, and in each of these take the principal pair in its subalgebra lattice corresponding to $a$. While this is a more complex object, it leads to an alternative view of how principal pairs encode elements, and is the tool we use to extend matters to the orthoalgebra setting. For the following, we note that for $x=\{0,a,a',1\}$, the upset $\up x$ is the set of Boolean subalgebras of $B$ that contain $a$. For $y\in\up x$ we use the following notation. 
\[ \down_y a=\{b\in y:b\leq a\}\quad\mbox{ and }\quad\up_y a=\{b\in y:a\leq b\}\]

\begin{definition}
\label{log}
For $B$ a Boolean algebra, $a\in B$, and $x=\{0,a,a',1\}$, let $d_a\colon \up x\to (\Sub(B))^2$ be given by 
\begin{equation*} \label{d_a}
d_a(y) = (\down_y a\cup\up_y a',\down_y a'\cup\up_y a)\text.
\end{equation*}
\end{definition}

Note that if $y$ is a subalgebra of $B$, then the lattice of subalgebras $\Sub(y)$ of $y$ is the interval $\down y$ of $\Sub(B)$. Note also that the definition of $d_a$ can be expanded to
\[ d_a(y)=\bigl(y\land(\down a\cup\up a'), y\land(\down a'\cup\up a)\bigr)\text.\]

The aim is as before --- to characterize the mappings $d_a$ order-theoretically and show that when $B$ is not small that these are in bijective correspondence with $B$. Then we define structure on the collection of such mappings, and show that with respect to this structure, this bijective correspondence is an isomorphism.

\begin{definition}
\label{logg}
For a Boolean domain $X$, a \emph{direction} of $X$ is a map $d\colon \up x\to X^2$ for some basic element $x\in X$ such that for each $y,z\in\up x$:
\begin{enumerate}
\item $d(y)$ is a principal pair for $x$ in the Boolean domain $\down y$;
\item if $y\leq z$ and $d(z)=(v,w)$, then $d(y)=(y\wedge v,y\wedge w)$.
\end{enumerate}
We say $d$ is a \emph{direction for} $x$, and write $\oDir(X)$ for the set of directions of $X$. 
\end{definition}

\begin{proposition}
\label{swa}
Let $X$ be a Boolean domain
. 
\begin{enumerate}
\item Each direction $d$ of $X$ determines a principal pair $d(\top)$ of $X$ and vice versa. 
\item If $d$ is a direction for $x$ and $x<y$, then $d(y)$ determines $d(\top)$ and hence $d$. 
\item For each principal pair $(u,v)$ of $X$ there is a unique direction $d$ with $d(\top)=(u,v)$.
\end{enumerate}
In particular, there is a bijection $\gamma\colon \oDir(X)\to \Dir(X)$ with $\gamma(d)=d(\top)$. 
\end{proposition}

\begin{proof}
(1) This is clear from the definition of a direction. (2) Assume first that $x=\bot$. Then the principal pairs for $x$ in $X$ are $(\top,\bot)$ and $(\bot,\top)$
. If $d(\top)=(\top,\bot)$, then the definition of a direction gives $d(y)=(y,\bot)$, and if $d(\top)=(\bot,\top)$, then $d(y)=(\bot,y)$. Since $y\neq\bot$, $d(y)$ determines $d(\top)$. Suppose that $x$ is an atom of $X$. Then, since $x<y$, $\down y$ has more than two elements. So by Theorem~\ref{thm:booleanobjects} the components of the principal pair $d(y)$ for $x$ in $\down y$ are different. If $d(\top)=(u,v)$, then $d(y)=(y\wedge u,y\wedge v)$, and if $d(\top)=(v,u)$, then $d(y)=(y\wedge v,y\wedge u)$. Thus $d(y)$ determines $d(\top)$. (3) Since every Boolean domain is isomorphic to $\Sub(B)$ for some Boolean algebra $B$, we may assume $X=\Sub(B)$. Suppose $(u,v)$ is a principal pair for a basic element $x$ of $X$. By Theorem~\ref{thm:booleanobjects} there is $a\in B$ with $x=\{0,a,a',1\}$ and $\varphi(a)=(u,v)$. So $u=\down a\cup\up a'$ and $v=\down a'\cup\up a$. Define $d\colon \up x\to X^2$ by $d(y)=(\down_y a\cup\up_y a',\down_y a'\cup\up_y a)$. It is easily seen that $d$ is a direction with $d(\top)=(u,v)$. Its uniqueness follows from (2). 
\end{proof}

Proposition~\ref{swa} allows us to count the number of directions for a basic element using the known number of principal pairs for it.

\begin{corollary}\label{Cor:no_of_Bool_directions}
Let $X$ be a Boolean domain. 
If $x\ne\top$ is a basic element of $X$, then there are exactly two directions for~$x$.
If $\top$ is a basic element of $X$ (so $X$ has at most two elements), then there is exactly one direction for~$\top$.
\end{corollary}

The bijection of Proposition~\ref{swa} can be used to define a unary operation $'$ and binary relation $\leq$ on $\oDir(X)$ so that $\Dir(X)$ is isomorphic to $\oDir(X)$. For a direction $d$, we have that $d'$ is the direction with the same domain and 
if $d(y)=(u,v)$ then $d'(y)=d(y)'=(v,u)$. 
For directions $d,e$, we have $d\leq e$ iff the principal pairs $d(\top)$ and $e(\top)$ satisfy $d(\top)\leq e(\top)$. The following corollary of Theorem~\ref{thm:booleanobjects} is then immediate. 

\begin{corollary}\label{Cor:booleanobjects}
For a Boolean algebra $B$ with more than $4$ elements, and a Boolean domain $X$ with more than $2$ elements: 
\begin{enumerate}
\item $\Sub(B)$ is a Boolean domain;
\item $\oDir(X)$ is a Boolean algebra;
\item $B$ is isomorphic to $\oDir(\Sub(B))$;
\item $X$ is isomorphic to $\Sub(\oDir(X))$.
\end{enumerate}
\end{corollary}

We conclude this section with an alternate view of the reconstruction of a Boolean algebra~$B$ from its Boolean domain $X=\Sub(B)$. Let $a\in B$
. We consider the case when $x_a\neq\bot$. For each cover $y$ of $x_a$ we have that the 4-element Boolean algebra $x_a$ is a subalgebra of the $8$-element Boolean algebra $y$. The element $a\in x_a$ can either embed as an atom in $y$, or as a coatom in $y$. In the first case $(\down_y a\cup\up_y a',\down_y a'\cup\up_y a)$ is $(x_a,y)$, and in the second, it is $(y,x_a)$. If we use $\down$ for $(x_a,y)$ and $\up$ for $(y,x_a)$, a direction $d$ of $X$ for the basic element $x_a$ assigns to each cover $y$ of $x_a$ the value $d(y)=\down $ or $d(y)=\up $ describing how  $x_a$ is embedded. This assignment of $\down $ and $\up $ to the covers of $x_a$ must be done in a way that is consistent with $d$ being a direction, and for each $x_a$ there are only two possibilities, one obtained from the other by interchanging $\down $ and $\up $ for each cover. Virtually identical remarks hold when $x_a=\bot$, except that we consider embedding a 2-element Boolean algebra into 4-element ones. See Figure \ref{fig1b}. 

\begin{example} 
Consider the power set $B=\mathcal{P}(\{1,2,3,4\})$ of $\{1,2,3,4\}$. Its poset $X$ of Boolean subalgebras is given in Figure~\ref{fig1}. 

\renewcommand{\powerone}[1]{\begin{tikzpicture}[scale=.5,draw=gray!50,inner sep=0pt,outer sep=0pt]
    \node (a0) {$\;\emptyset\;$};
    \node (a1) at ([yshift=1cm]a0) {$\;#1\;$};
      \draw (a0) to (a1);
  \end{tikzpicture}}
\renewcommand{\powertwo}[2]{\begin{tikzpicture}[scale=.33,draw=gray!50,inner sep=0pt,outer sep=0pt]
    \node (b0) {$\emptyset$};
    \node (b1) at ([yshift=2cm]0) {$1234$};
    \node (bx) at ([xshift=-1cm,yshift=1cm]0) {$#1$};
    \node (bx') at ([xshift=1cm,yshift=1cm]0) {$#2$};
    \draw (b0) to (bx) to (b1);
    \draw (b0) to (bx') to (b1);
  \end{tikzpicture}}
\renewcommand{\powerthree}[6]{\begin{tikzpicture}[scale=.5,draw=gray!50,inner sep=0pt,outer sep=0pt]
    \node (c0) {$\emptyset$};
    \node (c1) at ([yshift=26mm]0) {$1234$};
    \node (ca) at ([xshift=-1cm,yshift=8mm]0) {$#1$};
    \node (cb) at ([yshift=8mm]0) {$#2$};
    \node (cc) at ([xshift=1cm,yshift=8mm]0) {$#3$};
    \node (cd) at ([xshift=-1cm,yshift=18mm]0) {$#4$};
    \node (ce) at ([yshift=18mm]0) {$#5$};
    \node (cf) at ([xshift=1cm,yshift=18mm]0) {$#6$};
    \draw (c0) to (ca); \draw (c0) to (cb); \draw (c0) to (cc);
    \draw (c1) to (cd); \draw (c1) to (ce); \draw (c1) to (cf);
    \draw (ca) to (ce); \draw (ca) to (cf);
    \draw (cb) to (cd); \draw (cb) to (cf);
    \draw (cc) to (cd); \draw (cc) to (ce);
  \end{tikzpicture}}  

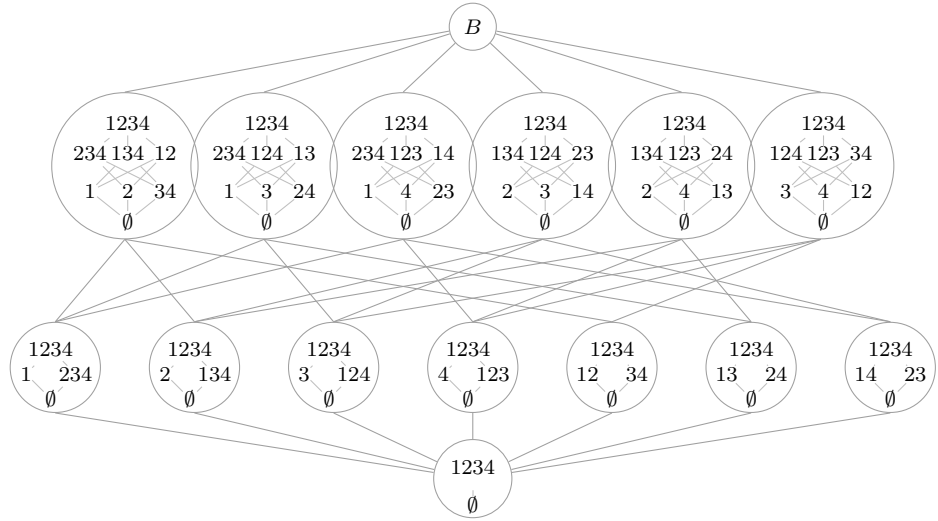
\begin{figure}[h]
\begin{tikzpicture}[font=\tiny,scale=1.85,inner sep=-3pt,draw=gray!75]
    \node[circle,draw] (0) at (0,0) {$\powerone{1234}$};
    \node[circle,draw] (a) at (-3,.8) {$\powertwo{1}{234}$};
    \node[circle,draw] (b) at (-2,.8) {$\powertwo{2}{134}$};
    \node[circle,draw] (c) at (-1,.8) {$\powertwo{3}{124}$};
    \node[circle,draw] (d) at (0,.8) {$\powertwo{4}{123}$};
    \node[circle,draw] (e) at (1,.8) {$\powertwo{12}{34}$};
    \node[circle,draw] (f) at (2,.8) {$\powertwo{13}{24}$};
    \node[circle,draw] (g) at (3,.8) {$\powertwo{14}{23}$};
    \node[circle,draw] (m1) at (-2.5,2.25) {$\powerthree{1}{2}{34}{234}{134}{12}$};
    \node[circle,draw] (m2) at (-1.5,2.25) {$\powerthree{1}{3}{24}{234}{124}{13}$};
    \node[circle,draw] (m3) at (-.5,2.25) {$\powerthree{1}{4}{23}{234}{123}{14}$};
    \node[circle,draw] (m4) at (.5,2.25) {$\powerthree{2}{3}{14}{134}{124}{23}$};
    \node[circle,draw] (m5) at (1.5,2.25) {$\powerthree{2}{4}{13}{134}{123}{24}$};
    \node[circle,draw] (m6) at (2.5,2.25) {$\powerthree{3}{4}{12}{124}{123}{34}$};
    \node[circle,draw] (B) at (0,3.25) {$\quad B\quad$};
    \draw (0) to (a.south); \draw (0) to (b.south); \draw (0) to (c.south); \draw (0) to (d.south); \draw (0) to (e.south); \draw (0) to (f.south); \draw (0) to (g.south); 
    \draw (B) to (m1.north); \draw (B) to (m2.north); \draw (B) to (m3.north); \draw (B) to (m4.north); \draw (B) to (m5.north); \draw (B) to (m6.north);
    \draw (a.north) to (m1.south); \draw (a.north) to (m2.south); \draw (a.north) to (m3.south);
    \draw (b.north) to (m1.south); \draw (b.north) to (m4.south); \draw (b.north) to (m5.south);
    \draw (c.north) to (m2.south); \draw (c.north) to (m4.south); \draw (c.north) to (m6.south);
    \draw (d.north) to (m3.south); \draw (d.north) to (m5.south); \draw (d.north) to (m6.south);
    \draw (e.north) to (m1.south); \draw (e.north) to (m6.south); 
    \draw (f.north) to (m2.south); \draw (f.north) to (m5.south); 
    \draw (g.north) to (m3.south); \draw (g.north) to (m4.south); 
  \end{tikzpicture}
  \caption{The poset of subalgebras of a $16$-element Boolean algebra}
  \label{fig1}
\end{figure}

We describe the directions of $X$ corresponding to the elements $a=\{1\}$ and $b=\{1,2\}$ in Figure~\ref{fig1b} by indicating their values $\down $ or $\up $ on the covers of the basic elements corresponding to these elements. 

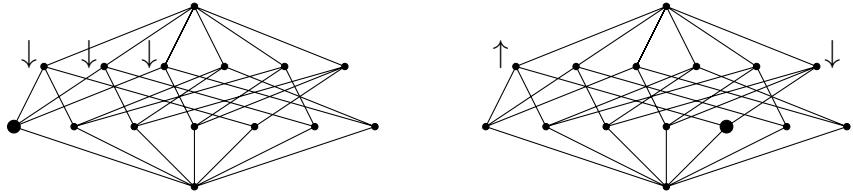
\begin{figure}[h]
\begin{center}
\begin{tikzpicture}[scale=.8]
\draw[fill] (0,0) circle (1.5pt); \draw[fill] (0,3) circle (1.5pt); 
\draw[fill] (-3,1) circle (3pt); \draw[fill] (-2,1) circle (1.5pt); \draw[fill] (-1,1) circle (1.5pt);\draw[fill] (0,1) circle (1.5pt);\draw[fill] (1,1) circle (1.5pt);\draw[fill] (2,1) circle (1.5pt);\draw[fill] (3,1) circle (1.5pt);
\draw[fill] (-2.5,2) circle (1.5pt);\draw[fill] (-1.5,2) circle (1.5pt);\draw[fill] (-.5,2) circle (1.5pt);\draw[fill] (.5,2) circle (1.5pt);\draw[fill] (1.5,2) circle (1.5pt);\draw[fill] (2.5,2) circle (1.5pt);
\draw[thin] (0,0)--(-3,1)--(-2.5,2)--(-2,1)--(0,0)--(-1,1)--(-1.5,2)--(-3,1)--(-.5,2)--(0,1)--(1.5,2)--(-2,1)--(.5,2)--(-1,1)--(2.5,2)--(0,1)--(0,0)--(1,1)--(-2.5,2)--(0,3)--(-1.5,2)--(2,1)--(0,0)--(3,1)--(-.5,2)--(0,3)--(-.5,2)--(0,3)--(.5,2)--(3,1);
\draw[thin] (1.5,2)--(0,3)--(2.5,2)--(1,1);\draw[thin] (2,1)--(1.5,2);
\node at (-2.75,2.2) {$\down $}; \node at (-1.75,2.2) {$\down $}; \node at (-.75,2.2) {$\down $};
\end{tikzpicture}
\quad \quad \quad
\begin{tikzpicture}[scale=.8]
\draw[fill] (0,0) circle (1.5pt); \draw[fill] (0,3) circle (1.5pt); 
\draw[fill] (-3,1) circle (1.5pt); \draw[fill] (-2,1) circle (1.5pt); \draw[fill] (-1,1) circle (1.5pt);\draw[fill] (0,1) circle (1.5pt);\draw[fill] (1,1) circle (3pt);\draw[fill] (2,1) circle (1.5pt);\draw[fill] (3,1) circle (1.5pt);
\draw[fill] (-2.5,2) circle (1.5pt);\draw[fill] (-1.5,2) circle (1.5pt);\draw[fill] (-.5,2) circle (1.5pt);\draw[fill] (.5,2) circle (1.5pt);\draw[fill] (1.5,2) circle (1.5pt);\draw[fill] (2.5,2) circle (1.5pt);
\draw[thin] (0,0)--(-3,1)--(-2.5,2)--(-2,1)--(0,0)--(-1,1)--(-1.5,2)--(-3,1)--(-.5,2)--(0,1)--(1.5,2)--(-2,1)--(.5,2)--(-1,1)--(2.5,2)--(0,1)--(0,0)--(1,1)--(-2.5,2)--(0,3)--(-1.5,2)--(2,1)--(0,0)--(3,1)--(-.5,2)--(0,3)--(-.5,2)--(0,3)--(.5,2)--(3,1);
\draw[thin] (1.5,2)--(0,3)--(2.5,2)--(1,1);\draw[thin] (2,1)--(1.5,2);
\node at (-2.75,2.2) {$\up $}; \node at (2.76,2.2) {$\down $}; 
\end{tikzpicture}
\end{center}
  \caption{Direction for a distinguished atom in a $16$-element Boolean algebra}
  \label{fig1b}
\end{figure}

We note that the upper covers of a basic element will usually be assigned a mixture of values of $\down $ and $\up $, a matter we return to in greater detail when we consider orthoalgebras. 
\end{example}

\section{Orthoalgebras}\label{sec:orthoalgebras}

This section briefly recalls the basics of orthoalgebras and their subalgebras~\cite{Greechie}. 

\begin{definition}
An \emph{orthoalgebra} is a set $A$, together with a partial binary operation $\oplus$ with domain of definition $\perp$, a unary operation $'$, and constants $0,1$, satisfying:
\begin{enumerate}
\item $\oplus$ is commutative and associative in the usual sense for partial operations;
\item $a'$ is the unique element with $a\oplus a'$ defined and equal to $1$;
\item $a\oplus a$ is defined if and only if $a=0$.
\end{enumerate}
An orthoalgebra is \emph{Boolean} when it arises from a Boolean algebra by restricting the join to pairs of orthogonal elements.
\end{definition}

Any orthoalgebra is partially ordered by $a\leq c$ if $a\perp b$ and $a\oplus b=c$ for some~$b$. An orthoalgebra is Boolean if and only if this is the partial ordering of a Boolean algebra. 

\begin{definition}
	For orthoalgebras $A$ and $C$, an {\em orthoalgebra morphism} $f\colon A\to C$ is a function that preserves orthocomplementation and satisfies: if $a\oplus b$ is defined, then so is $f(a)\oplus f(b)$, and $f(a\oplus b)=f(a)\oplus f(b)$. If, in addition, there is an orthoalgebra morphism $g\colon C\to A$ that is the set-theoretic inverse of $f$, then we call $f$ an {\em orthoalgebra isomorphism}.
\end{definition}

\begin{definition}
Let $A$ be an orthoalgebra. A subset $S\subseteq A$ is a \emph{subalgebra} if:
\begin{enumerate}
\item $0,1\in S$;
\item $a\in S\Rightarrow a'\in S$;
\item if $a,b\in S$ and $a\perp b$ then $a\oplus b\in S$.
\end{enumerate}
A subalgebra that is a Boolean orthoalgebra is a \emph{Boolean subalgebra}. A \emph{block} is a maximal Boolean subalgebra. A block is \emph{small} when it has $4$ or fewer elements. Write $\BSub(A)$ for the set of Boolean subalgebras of $A$ partially ordered by inclusion. We call $A$ \emph{proper} if it does not have small blocks, or equivalently, if $\BSub(A)$ does not have basic elements that are maximal. 
\end{definition}

We next consider how to recognize when an orthoalgebra is obtained from a Boolean algebra. 
	\begin{definition}
		For $n\geq 0$ let $x_1,\ldots,x_n$ be a finite sequence of elements of an orthoalgebra $A$. Define $\bigoplus_{i=1}^0x_i=0$. If $\bigoplus_{i=1}^kx_i$ is defined for $0\leq k<n$, and if $\big(\bigoplus_{i=1}^kx_i\big)\oplus x_{k+1}$ is defined, set $\bigoplus_{i=1}^{k+1}x_i=\big(\bigoplus_{i=1}^kx_i\big)\oplus x_{k+1}$.  
	\end{definition}
	Given a permutation $\pi$ of $(1,\ldots,n)$, the commutativity and associativity laws for orthoalgebras assure that $\bigoplus_{i=1}^nx_i$ is defined if and only if $\bigoplus_{i=1}^nx_{\pi(i)}$ is defined, and when defined, these are equal. Thus $\bigoplus F$ in the following definition is well defined.
	
	\begin{definition}
		Let $A$ be an orthoalgebra. We call a finite subset $F$ (say of $n$ elements) of $A$ \emph{jointly orthogonal} if there is an enumeration $x_i$ ($1\leq i\leq n$) of $F$ such that $\bigoplus_{i=1}^nx_i$ is defined, in which case we define $\bigoplus F=\bigoplus_{i=1}^nx_i$. If $0\notin F$ and $\bigoplus F=1$, we call $F$ a \emph{partition of unity}.
	\end{definition}

\begin{proposition}\label{prop:booleansubalgebra}
	An orthoalgebra $A$ is Boolean if and only if every finite subset $S \subseteq A$ is contained in $\{ \bigoplus E \mid E \subseteq F \}$ for some jointly orthogonal set~$F$.
\end{proposition}
\begin{proof}
	First assume $A$ is a Boolean orthoalgebra, say it is the restriction of a Boolean algebra~$B$. Any finite subset $S \subseteq A$ is contained in a finite subalgebra $C$ of $B$ because finitely generated subalgebras of Boolean algebras are finite.
	Let $F$ be the set of atoms of $C$. This is a jointly orthogonal set in $A$. For each $b \in S$ let $E=\{x \in F \mid x \leq b\}$. Now, because $b$ is the join of $E$ in $C$ and $B$, we have $b = \bigoplus E$.
	
	For the converse, suppose that every finite subset $S \subseteq A$ is contained in $\{\bigoplus E \mid E \subseteq F\}$ for some jointly orthogonal set~$F$. First assume that $A$ is finite. In this case, there is a jointly orthogonal family $F$ such that every element of $A$ equals $\bigoplus E$ for some $E \subseteq F$. Clearly $\bigoplus F = 1$, and if $F$ contains $0$, we may remove $0$ from $F$ to obtain a partition of unity of $A$. 
	Since $A=\{\bigoplus E \mid E\subseteq F\}$, it is isomorphic to the orthoalgebra induced by the Boolean algebra $\mathcal{P}(F)$, and hence $A$ is Boolean.
	
	Now consider the case where $A$ is infinite. For each partition of unity $F$ of $A$, let $B_F$ be the subalgebra of $A$ generated by $F$. Explicitly, $B_F = \{ \bigoplus E \mid E \subseteq F\}$, and in particular each $B_F$ is a finite Boolean orthoalgebra. By hypothesis, each finite subset of $A$ is contained in $B_F$ for some partition of unity $F$. If $F_1$ and $F_2$ are partitions of unity, then $B_{F_1}\cup B_{F_2}$ is a finite subset of $A$, hence $B_{F_1}\cup B_{F_2}$ is contained in $B_F$ for some partition of unity $F$. Thus 
	\[
	\{B_F \mid F\mbox{ is a finite partition of unity for }A\}
	\] 
	is an up-directed family of subalgebras of $A$. Furthermore, each finitely generated subalgebra of $A$ is contained in some member of this family. But then the union of this family is all of $A$. Hence $A$ is Boolean.
\end{proof}

When an orthoalgebra $A$ has more than 2 elements, all of its small blocks have 4 elements. In this case small blocks are also known as horizontal summands. By removing a small block from such $A$, we mean removing the two elements of the block that are not $0,1$. Except when $A$ has only small blocks, removing the small blocks from $A$ leaves an orthoalgebra $\tA $ without small blocks, and $A$ can be recovered from $\tA $ by taking the horizontal sum of $\tA $ and an appropriate number of 4-element Boolean algebras. 

We collect in the following remark motivation for why orthoalgebras are a natural choice of ambient structure to reconstruct from Boolean subalgebras.

\begin{remark}
Each element $a$ of an orthoalgebra belongs to the Boolean subalgebra~$x_a$. Thus any orthoalgebra \emph{pastes together} a family of Boolean orthoalgebras. More generally, call a family $\mc{F}$ of Boolean orthoalgebras \emph{compatible}~\cite[1.7]{Czelakowski} if for each $B,C\in\mc{F}$:
\begin{enumerate}
\item $B$ and $C$ have the same $0$ and $1$;
\item If $a\in B\cap C$, then $a'$ in $B$ equals $a'$ in $C$;
\item for $a,b\in B\cap C$, $a\oplus b$ exists in $B$ iff it exists in $C$, and when defined they are equal. 
\end{enumerate}
Any compatible family gives rise to a structure $(A,\oplus,\,',0,1)$ by union. A structure $(A,\oplus,\,',0,1)$ that arises this way is called a \emph{weak orthostructure}, extending~\cite{Czelakowski}. This general setup includes orthoalgebras as well as partial Boolean algebras~\cite{Kochen}. 

A Boolean subalgebra of a weak orthostructure $A$ is a subset $B\subseteq A$ that is closed under $0,1,\,',\oplus$ and forms a Boolean orthoalgebra. One might hope to reconstruct $A$ from its poset $\BSub(A)$ of Boolean subalgebras, but this is impossible: the partially ordered set $\BSub(A)$ in the introduction is not only induced by the orthoalgebra $A$ in the introduction, but it is also isomorphic to $\BSub(D)$ for the weak orthostructure $D$ obtained by taking two $8$-element Boolean algebras that intersect in a 4-element Boolean algebra $x_c$ where $c$ is an atom of one of the $8$-element Boolean algebras and a coatom of the other $8$-element Boolean algebra. This $D$ is not only a weak orthostructure, but is a partial Boolean algebra. This structure $D$ is not an orthoalgebra, and cannot be depicted via a Hasse diagram.
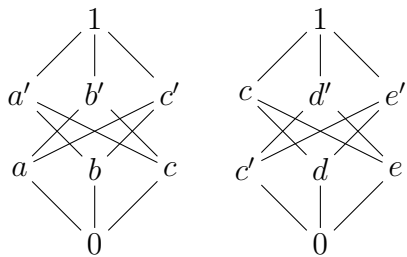
\[\begin{tikzpicture}[inner sep=2pt]
\node (0) at (0,0) {$0$};
\node (1) at (0,3) {$1$};
\node (02) at (3,0) {$0$};
\node (12) at (3,3) {$1$};
\node (a) at (-1,1) {$a$};
\node (b) at (0,1) {$b$};
\node (c) at (1,1) {$c$};
\node (c2) at (2,1) {$c'$};
\node (d) at (3,1) {$d$};
\node (e) at (4,1) {$e$};
\node (a') at (-1,2) {$a'$};
\node (b') at (0,2) {$b'$};
\node (c') at (1,2) {$c'$};
\node (c2') at (2,2) {$c$};
\node (d') at (3,2) {$d'$};
\node (e') at (4,2) {$e'$};
\draw (0) to (a); \draw (0) to (b); \draw (0) to (c); 
\draw (02) to (e); \draw (02) to (d); \draw (02) to (c2); 
\draw (1) to (a'); \draw (1) to (b'); \draw (1) to (c'); 
\draw (12) to (e'); \draw (12) to (d'); \draw (12) to (c2'); 
\draw (a) to (b'); \draw (a) to (c');
\draw (b) to (a'); \draw (b) to (c');
\draw (c) to (a'); \draw (c) to (b'); 
\draw (e) to (d'); \draw (e) to (c2');
\draw (d) to (e'); \draw (d) to (c2');
\draw (c2) to (e'); \draw (c2) to (d'); 
\end{tikzpicture}\]
\end{remark}


\section{Orthodomains and directions}\label{sec:directions}

This section abstracts basic properties of $\BSub(A)$ for orthoalgebras $A$ into a notion of orthodomain. We generalize directions from Boolean domains to directions on orthodomains, and show that an orthoalgebra $A$ can be reconstructed from the directions on its orthodomain $\BSub(A)$. First, an example that exhibits some counterintuitive behavior in $\BSub(A)$. 

\begin{example}\label{ex:frasercube}
The \emph{Fraser cube} is the orthoalgebra $A$ displayed in the diagram. The four vertices of each face are the atoms of a $16$-element Boolean subalgebra of~$A$. 
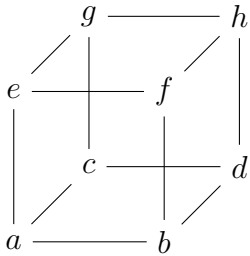
\[\begin{tikzpicture}[scale=2]
\node (a) at (0,0) {$a$};
\node (b) at (1,0) {$b$};
\node (c) at (.5,.5) {$c$};
\node (d) at (1.5,.5) {$d$};
\node (e) at (0,1) {$e$};
\node (f) at (1,1) {$f$};
\node (g) at (.5,1.5) {$g$};
\node (h) at (1.5,1.5) {$h$};
\draw (a) to (b); \draw (b) to (d); \draw (d) to (c); \draw (c) to (a);
\draw (e) to (f); \draw (f) to (h); \draw (h) to (g); \draw (g) to (e);
\draw (a) to (e); \draw (b) to (f); \draw (c) to (g); \draw (d) to (h);
\end{tikzpicture}\]

Consider the element $a\oplus b$. Its orthocomplement in the Boolean algebra corresponding to the bottom face is $c\oplus d$, and its orthocomplement in the Boolean algebra corresponding to the front face is $e\oplus f$. Thus $c\oplus d=e\oplus f$. Similarly, the intersection of the Boolean subalgebras for the top and bottom of the cube consists of $0,a\oplus b$, $b\oplus d$, $c\oplus d$, $a\oplus c$, and $1$. Thus the intersection of two Boolean subalgebras need not be Boolean. This implies that in $\BSub(A)$, two elements need not have a meet, and two elements that have an upper bound need not have a least upper bound, in contrast to the situation for Boolean domains and posets of Boolean subalgebras of orthomodular posets. 
\end{example}

\begin{definition}
Write $\lessdot$ for the covering relation in a partially ordered set: $x\lessdot z$ means $x<z$ and there is no $y$ with $x<y<z$. 
\end{definition}

\begin{definition}\label{orthodomain}
An \emph{orthodomain} is a partially ordered set $X$ with least element $\bot$ such that:
\begin{enumerate}
\item every directed subset of $X$ has a join;
\item $X$ is atomistic and the atoms are compact;
\item each principal ideal $\down x$ is a Boolean domain;
\item if $x,y$ are distinct atoms and $x,y\lessdot w$, then $x\vee y = w$.
\end{enumerate}
\end{definition}




\begin{lemma}
\label{maximal}
Each element of an orthodomain lies beneath a maximal element.
\end{lemma}

\begin{proof}
Let $X$ be an orthodomain and $x\in X$. Zorn's lemma produces a maximal directed set containing $x\in X$. Taking the join of this maximal directed set provides a maximal element of $X$ above $x$.
\end{proof}

We next examine condition (4) more closely.

\begin{definition}
Atoms $x,y$ of an orthodomain are called \emph{near} if they are distinct, their join exists and covers $x$ and $y$. Equivalently, by condition (4): $x$ and $y$ are near if they are distinct and have an upper bound of height~$2$.
\end{definition}

The following property, similar to the exchange property of geometry, will be key.

\begin{proposition}[Exchange property] \label{p:exchange}
If $x,y$ are near atoms of an orthodomain with $x\vee y=w$, then there is exactly one atom $z$ that is distinct from $x,y$ and with $z\lessdot w$. Further, any two of $x,y,z$ are near. 
\end{proposition}

\begin{proof}
By nearness, $x\vee y=w$ exists and covers $x$ and $y$, and by the definition of an orthodomain, $\down w$ is a Boolean domain. Since the top of this Boolean domain covers an atom in it, the Boolean domain $w$ must be isomorphic to the subalgebra lattice of an $8$-element Boolean algebra. Then $\down w$ must have 3 distinct atoms, so there is a third atom $z$ distinct from $x,y$ with $z\lessdot w$. Then $x,y,z\lessdot w$. It follows from the definition of orthodomain that $w$ is the join of any two of $x,y,z$, hence any two of $x,y,z$ are near. 
\end{proof}


\begin{proposition}\label{prop:orthodomain}
  If $A$ is an orthoalgebra, $\BSub(A)$ is an orthodomain where directed joins are given by unions.
\end{proposition}
\begin{proof}
Let $S \subseteq \BSub(A)$ be a directed family. By directedness, $B=\bigcup S$ is closed under $\oplus,\,',0,1$, and is hence is a subalgebra. Also by directedness, Proposition~\ref{prop:booleansubalgebra} shows that $B$ is Boolean. Thus $\BSub(A)$ has directed joins given by union. 

The atoms of $\BSub(A)$ are the Boolean subalgebras $x_a$ where $a\neq 0,1$. Since directed joins are given by unions, it follows that the compact elements of $\BSub(A)$ are exactly the finite Boolean subalgebras, and hence every atom is compact. Any $B\in\BSub(A)$ is the union and hence join of the atoms beneath it, making $\BSub(A)$ atomistic. Finally, for a Boolean orthoalgebra~$B$, its Boolean subalgebras are exactly its subalgebras that are Boolean, so (3) holds in $\BSub(A)$. 

For (4), suppose $x_a,x_b \in \BSub(A)$ be distinct atoms with $x_a,x_b\lessdot w$.
Then $0,a,a',b,b',1$ are all distinct. Since $w$ covers an atom and $\down w$ is a Boolean domain, it is an $8$-element Boolean subalgebra of $A$ containing $x_a,x_b$. So $w=\{0,a,a',b,b',c,c',1\}$ is an $8$-element Boolean subalgebra of $A$ for some $c \in A$. One of $a,a'$ is an atom of $w$, as is one of $b,b'$, and one of $c,c'$. We may assume that $a,b,c$ are atoms. Then $a\oplus b = c'$ in $w$, and so $a\oplus b=c'$ in $A$. Then if $v$ is a Boolean subalgebra of $A$ that contains $x_a,x_b$, we have $a,b\in v$, hence $a\oplus b=c'\in v$. Thus $w=\{0,a,a',b,b',c,c',1\}\subseteq v$, and $x_a\vee x_b=w$. Thus $\BSub(A)$ is an orthodomain. 
\end{proof}

We now begin the task of reconstructing an orthoalgebra $A$ from its orthodomain $\BSub(A)$. The idea is to extend the directions used in the Boolean case to the orthoalgebra setting. The reader should consult Definitions~\ref{log} and \ref{logg}. 

\begin{definition}\label{xes}
Let $A$ be an orthoalgebra, $a\in A$. Define the \emph{direction corresponding to $a$} to be the map $d_a\colon \up x_a\to(\BSub(A))^2$ given by 
\[ d_a(y)  =  (\down_y a\cup\up_y a',\down_y a'\cup\up_y a)\text.\]
\end{definition}

We seek an order-theoretic description in terms of an orthodomain $X$ of the mappings $d_a$. These are again called directions, since when restricted to the setting of Boolean domains, these are the directions given in Definition~\ref{logg}. 

\begin{definition}\label{scary}
A~\emph{direction} for a basic element $x_a$ of an orthodomain $X$ is a map $d \colon \up x_a\to X^2$ such that for each $y,z\in\up x_a$:
\begin{enumerate}
\item $d(y)$ is a principal pair for $x_a$ in the Boolean domain $\down y$;
\item if $y\leq z$ and $d(z)=(v,w)$, then $d(y)=(y\wedge v,y\wedge w)$;
\item if $x_a\lessdot y,z$ and $d(y)=(x_a,y)$, $d(z)=(z,x_a)$, then $y\vee z$ exists and $y,z\lessdot y\vee z$.
\end{enumerate}
Write $\oDir(X)$ for the set of directions for basic elements of~$X$. 
\end{definition}

Condition (3) of Definition~\ref{scary} looks strange, but its effect will become clear in the proof of Proposition~\ref{at most 2}.

Note that if $d$ is a direction for some basic element $x$, then $x$ can be determined from the partial mapping $d$ as the least element of its domain. 

\begin{proposition}\label{elway}
Let $d$ be a direction for a basic element $x$ of an orthodomain $X$.
\begin{enumerate}
\item for any $x\leq y\leq z$, the value of $d(y)$ is determined by $d(z)$;
\item for any $x<y\leq z$, the value of $d(z)$ is determined by $d(y)$.
\end{enumerate}
\end{proposition}

\begin{proof}
(1) This is immediate from the definition of direction. For (2), let $v,w\in\down z$ be such that $(v,w)$ and $(w,v)$ are the two principal pairs for $x$ in $\down z$, so $d(z)=(v,w)$ or $d(z)=(w,v)$. In the first case, $d(y)=(y\wedge v,y\wedge w)$, in the second $d(y)=(y\wedge w,y\wedge v)$. 
We claim that $y\wedge v\neq y\wedge w$, so $d(y)$ determines $d(z)$. To see this, Definition~\ref{pp} gives that $v\wedge w = x$. So if $y\wedge v=y\wedge w$, then $x=y\wedge v\wedge w = y\wedge v = y\wedge w$, but $d(y)=(x,x)$ contradicts $d(y)$ being a principal pair for~$x$ in $\down y$ (Remark~\ref{smallB}).
\end{proof}

We will show that the directions of an orthodomain $\BSub(A)$ form an orthoalgebra $A$, but we will see that not all orthodomains are of the form $\BSub(A)$ for an orthoalgebra $A$. So our task in the orthodomain setting is more complex than in the Boolean domain setting. To make our path efficient, we next prove the following result that is valid for any orthodomain. 

\begin{proposition}\label{at most 2}
A basic element $x$ of an orthodomain $X$ has at most two directions. 
\end{proposition}

\begin{proof}
If $x$ is maximal, $(x,x)$ is the only principal pair for $x$ in $\down x$, so there is only one direction for $x$ in $X$. Suppose that $x$ is not maximal. By Lemma~\ref{maximal} and the definition of a direction, any direction $d$ for $x$ is determined by its value on the maximal elements $w > x$. For any such~$w$, the value $d(w)$ is a principal pair for $x$ in $\down w$, so by Corollary~\ref{two} can take two values. 

Suppose there are three distinct directions $d_1,d_2,d_3$ for $x$. Choose any maximal $w>x$. Then two of $d_1,d_2,d_3$ must agree at $w$, say $d_1$ and $d_2$. Since $d_1\neq d_2$, there is a maximal $v$ with $d_1(v)\neq d_2(v)$. Choose $y,z$ with $x\lessdot y\leq w$ and $x\lessdot z\leq v$. Since $d_1(w)=d_2(w)$ and $y\leq w$, we have $d_1(y)=d_2(y)$, and since $d_1(v)\neq d_2(v)$, we have $d_1(z)\neq d_2(z)$. 

As $x$ is basic and $x\lessdot y$, either $x=\bot$ and $y$ is an atom, or $x$ is an atom and $y$ has height~$2$. 
In either case the principal pairs for $x$ in $\down y$ are $(x,y)$ and $(y,x)$, and similarly the principal pairs for $x$ in $\down z$ are $(x,z)$ and $(z,x)$.
Suppose that $d_1(y)=d_2(y)=(x,y)$.
As $d_1(z)\ne d_2(z)$, one of them is $(x,z)$ and the other is $(z,x)$.
We apply condition (3) of Definition~\ref{scary} and find an upper bound $u=y\vee z$ of $y,z$.
According to Proposition~\ref{swa}, $d_1,d_2$ are uniquely determined on $\down u$ by $d_1(y)=d_2(y)$, a contradiction with $d_1(z)\ne d_2(z)$, $z\in\down u$.
The case $d_1(y)=d_2(y)=(y,x)$ is excluded analogously with the role of $y,z$ interchanged in (3) of Definition~\ref{scary}.
This contradiction shows $x$ has at most two directions. 
\end{proof}

Now we begin putting structure on the set of directions of an orthodomain.

\begin{proposition}\label{complement}
Let $d$ be a direction for a basic element $x$ of an orthodomain $X$. There is a direction $d'$ for $x$ given by $d'(w)=(z,y)$ if $d(w)=(y,z)$. Further, there are directions $0$ and $1$ for the basic element $\bot\in X$, given by
\[ 0(w) = (\bot,w)\quad\quad \mbox{ and }\quad\quad1(w)=(w,\bot)\text. \]
\end{proposition}

\begin{proposition}\label{enough}
For an orthodomain with no basic maximal elements, the following are equivalent:
\begin{enumerate}
\item each basic element has a direction;
\item each basic element has exactly two directions. 
\end{enumerate}
\end{proposition}

\begin{proof} 
The direction (2) $\Rightarrow$ (1) is trivial. For the converse, let $d$ be a direction for~$x$. Then so is $d'$ given by Proposition~\ref{complement}. If $d=d'$, then for a maximal element $w$ above $x$ we have that $d(w)=d'(w)$, so $w$ is basic, contrary to our assumptions. Thus each basic element has at least two directions, so by Proposition~\ref{at most 2} has exactly two directions. 
\end{proof}

\begin{definition}\label{proper}
Call an orthodomain \emph{proper} if it has no maximal elements that are basic. 
Say it \emph{has enough directions} if it is proper and each basic element has a direction. 
\end{definition}
	It directly follows that $A$ is proper if and only if $\BSub(A)$ is proper. Note that the definition of an orthoalgebra with enough directions is somewhat analogous to that of spatial frames, which are defined through the existence of a sufficient supply of points. 

\begin{definition}\label{perp}
For $X$ an orthodomain with enough directions, let $\oplus$ be a partial binary operation on $\oDir(X)$ defined by the following three cases. For each direction $d$ set
\begin{enumerate}
\item $d\oplus 0 = d = 0\oplus d$;
\item $d\oplus d'=1$.
\end{enumerate}
For $d$ a direction for $x$ and $e$ a direction for $y$ with $x,y$ near and $z$ the third atom beneath $x\vee y=w$, then $d\oplus e$ is defined if $d(w)=(x,w)$ and $e(w)=(y,w)$, and in this case 
\begin{enumerate}
\setcounter{enumi}{2}
\item $d\oplus e$ is the direction for $z$ with $(d\oplus e)(w)=(w,z)$.
\end{enumerate}
We will write $d\perp e$ and say $d$ \emph{is orthogonal to} $e$ if $d\oplus e$ is defined. 
\end{definition}

\begin{theorem}\label{hhh}
Let $A$ be a proper orthoalgebra. Then the orthodomain $\BSub(A)$ has enough directions and the map $\xi:A\to\oDir(\BSub(A))$, $a\mapsto d_a$ is an orthoalgebra isomorphism.
\end{theorem}

\begin{proof}
By Proposition~\ref{prop:orthodomain}, $X=\BSub(A)$ is an orthodomain. Let $a\in A$.
%
Let us verify the three conditions of Definition~\ref{scary} for the direction $d_a$ given by Definition~\ref{xes}.
Condition~(1) follows because $d_a(y)$ is a principal pair in $\down y$ for~$x_a$. Condition~(2) follows by construction of~$d_a$. For condition~(3), consider first the case $x_a=\bot$, where $a$ is either $0$ or $1$. If $a=0$ then $d_a(y)=(\bot,y)$ for all $y$, and if $a=1$ then $d_a(y)=(y,\bot)$ for all $y$, so (3) holds vacuously. Suppose $x_a$ is an atom of $X$ with $x_a\lessdot y,z$ and that $d_a(y)=(x_a,y)$ and $d_a(z)=(z,x_a)$. This means that $y,z$ are $8$-element Boolean algebras, that $a$ is an atom in $y$, and $a$ is a coatom in $z$. Let $b_1,b_2$ be the atoms of $y$ distinct from $a$ and $c_1,c_2$ be the atoms of $z$ distinct from $a'$. We depict $y$ below on the left, and $z$ on the right.


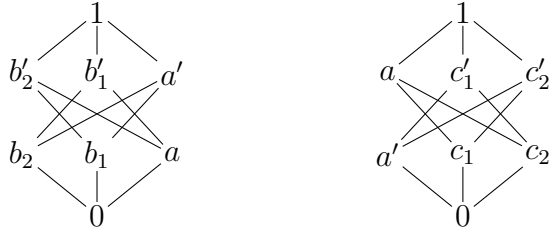
\[\begin{tikzpicture}[inner sep=1pt,yscale=.9]
\node (0) at (0,0) {$0$};
\node (a) at (1,.9) {$a$};
\node (b1) at (0,.9) {$b_1$};
\node (b2) at (-1,.9) {$b_2$};
\node (b2') at (-1,2.1) {$b_2'$};
\node (b1') at (0,2.1) {$b_1'$};
\node (a') at (1,2.1) {$a'$};
\node (1) at (0,3) {$1$};
\draw (0) to (a); \draw (0) to (b1); \draw (0) to (b2);
\draw (1) to (a'); \draw (1) to (b1'); \draw (1) to (b2');
\draw (a) to (b2'); \draw (a) to (b1');
\draw (b1) to (b2'); \draw (b1) to (a');
\draw (b2) to (b1'); \draw (b2) to (a');
\end{tikzpicture}
\qquad\qquad\qquad
\begin{tikzpicture}[inner sep=1pt,yscale=.9]
\node (0) at (0,0) {$0$};
\node (a) at (-1,.9) {$a'$};
\node (b1) at (0,.9) {$c_1$};
\node (b2) at (1,.9) {$c_2$};
\node (b2') at (1,2.1) {$c_2'$};
\node (b1') at (0,2.1) {$c_1'$};
\node (a') at (-1,2.1) {$a$};
\node (1) at (0,3) {$1$};
\draw (0) to (a); \draw (0) to (b1); \draw (0) to (b2);
\draw (1) to (a'); \draw (1) to (b1'); \draw (1) to (b2');
\draw (a) to (b2'); \draw (a) to (b1');
\draw (b1) to (b2'); \draw (b1) to (a');
\draw (b2) to (b1'); \draw (b2) to (a');
\end{tikzpicture}\]

Then $c_1\oplus c_2=a$ and $b_1\oplus b_2=a'$ in $A$. Therefore $b_1,b_2,c_2,c_2$ are the atoms of a $16$-element Boolean subalgebra $u$ of $A$. Clearly $u=y\vee z$ and $y,z\lessdot u$, establishing condition~(3). Thus $d_a$ is a direction. Since this holds for each $a\in A$, the orthodomain $X$ has enough directions. 

Every basic element of $X$ is of the form $x_a$ and has $2$ directions. Since $d_a$ and $d_{a'}$ are directions for $x_a$, the map $\xi$ is surjective. If $\xi(a)=\xi(b)$, then since $d_a$ is a direction given by $a$ and $d_b$ is a direction given by $b$, we must have that $b=a$ or $b=a'$. But $d_a\neq d_{a'}$, so $\xi$ is injective. 

To show that $\xi$ is an isomorphism, it is easily seen that $\xi$ maps $0,1$ of $A$ to the directions $0,1$ of $X$, and that $\xi(a')=\xi(a)'$. It remains to show that $a\perp b$ if and only if $\xi(a)\perp\xi(b)$ and that then $\xi(a\oplus b) = \xi(a)\oplus\xi(b)$. Consider the possibilities to have $a\perp b$. For any $a$ we have $a\perp 0$, $\xi(a)\perp\xi(0)$, and $\xi(a\oplus 0)=\xi(a)=\xi(a)\oplus\xi(0)$. For any $a$ we have $a\perp a'$ and $a\oplus a'=1$. Since $\xi(a')=\xi(a)'$, then $\xi(a)\perp\xi(a')$ and $\xi(a\oplus a')=1=\xi(a)\oplus\xi(a')$. 

The remaining possibility to have $a\perp b$ is when $a,b$ are distinct atoms of an $8$-element Boolean subalgebra $w$ of $A$. In this case, $x_a$ and $x_b$ are basic elements that are near, $x_a\vee x_b=w$, $z=\{0,a\oplus b,(a\oplus b)',1\}$ is the third atom beneath $w$, and $d_a(w)=(x_a,w)$, $d_b(w)=(x_b,w)$. Thus $\xi(a)\perp\xi(b)$, and as $d_a\oplus d_b$ is the direction for $z$ with $(d_a\oplus d_b)(w)=(w,z)$, we have $\xi(a)\oplus\xi(b)=\xi(a\oplus b)$. Conversely, suppose $\xi(a)\perp\xi(b)$ via condition~(3) of Definition~\ref{perp}. Since $d_a$ is a direction for $x_a$ and $d_b$ is a direction for $x_b$, this condition assumes $x_a,x_b$ are near and generate an $8$-element Boolean subalgebra of~$A$. Further, since $d_a(w)=(x_a,w)$ and $d_b(w)=(x_b,w)$, we have that $a,b$ are atoms of $w$, hence $a\perp b$ in $A$. Finally, $d_a\oplus d_b$ is the direction for the third atom $x_{a\oplus b}$ beneath $w$ with $(d_a\oplus d_b)(w)=(w,x_{a\oplus b})$, so $d_a\oplus d_b = d_{a\oplus b}$. 
\end{proof}

\begin{remark}\label{kino}
The previous theorem achieves one of our primary aims: a means to reconstruct an orthoalgebra $A$ from its poset of Boolean subalgebras. It only applies if $A$ is a proper orthoalgebra, but, with one exception, we can still recover $A$ from $\BSub(A)$ without this restriction. The exception is when $\BSub(A)$ has a single element, which occurs when $A$ is a 1-element orthoalgebra and also when $A$ is a 2-element orthoalgebra. In these cases it is clearly impossible to recover $A$ from $\BSub(A)$. 

Suppose then that $A$ has more than two elements. If it does have small blocks, these appear in $\BSub(A)$ as maximal atoms. Provided $A$ has a block that is not small, removing these blocks from $A$ yields an orthoalgebra $\tA $, and $\BSub(\tA )$ is obtained from $\BSub(A)$ from removing maximal atoms. Since we can reconstruct $\tA $ as $\oDir(\BSub(\tA ))$, we can then reconstruct $A$ by adding a number of horizontal summands equal to the number of maximal atoms of $\BSub(A)$. If $A$ consists of only small blocks, it is determined by the cardinality of the set of its maximal atoms. 
\end{remark}

\section{Characterizing orthodomains of the form $\BSub(A)$ }\label{sec:orthodomains}

In this section, we show, for any orthodomain $X$ with enough directions, that $\oDir(X)$ is an orthoalgebra, and characterize those orthodomains that are of the form $\BSub(A)$ for some orthoalgebra~$A$.  

\begin{definition}
For an orthodomain $X$, let $X^*$ be the set of elements of $X$ of height $3$ or less. A \emph{shadow} of $X$ is a nonempty subset $S\subseteq X^*$ satisfying:
\begin{enumerate}
\item $S$ is a downset of $X^*$;
\item $S$ is closed under existing joins in $X^*$.
\end{enumerate}
\end{definition}

Note, the second condition means that if $T\subseteq S$ and there is $w\in X^*$ that is the least upper bound of $T$ in $X$, which will imply that $w$ is also the least upper bound of $T$ in $X^*$, then $w\in S$. 

\begin{proposition} \label{ded}
Let $X$ be an orthodomain, $S$ be a shadow of $X$, $x$ be a basic element of $X$ with $x\in S$, and $d$ be a direction of $X$ for~$x$. Then:
\begin{enumerate}
\item $S$ is an orthodomain;
\item the restriction $d|_S$ of $d$ to $\up x\cap S$ is a direction of $S$.
\end{enumerate}
Hence if $X$ has enough directions and $S$ has no maximal elements which are basic, 
then $S$ has enough directions. 
\end{proposition}

\begin{proof}
Since $X^*$, and hence $S$, has finite height, every directed set has a maximal element and hence a join, and each element is compact. Since $X$ is atomistic and $S$ is a downset, it is atomistic. Since $S$ is a downset of $X$, for each $s\in S$ the ideal $\down s$ is a Boolean domain. Finally, if $x,y$ are atoms of $S$ and $x,y\lessdot w$, then $x\vee y = w$ in $X$, hence $x\vee y = w$ in $S$ as well. Thus $S$ is an orthodomain, establishing part (1). 

To see part~(2) we verify the three conditions of Definition~\ref{scary}. The first two are trivial consequences of restricting. 
For the third, suppose there are $x\lessdot y,z$ with $y,z\in S$ and $d(y)=(x,y)$, $d(z)=(z,x)$. Since $d$ is a direction of $X$, then $y\vee z=w$ exists in $X$ and $y,z\lessdot w$. Since $x$ is basic and $x\lessdot y,z\lessdot w$ then $w$ has height at most 3. So $w\in X^*$ and $w$ is the join of $y,z$ in~$X^*$. Since $S$ is a shadow, it is closed under joins in $X^*$, so $w\in S$ and $w=y\vee z$ in $S$. 
\end{proof}

\begin{definition}
Let $X$ be an orthodomain with enough directions and $S$ be a shadow of $X$ that is proper. Write $\oDir_S(X)$ for the set of directions of $X$ for basic elements $x\in S$, and let $\mu_S \colon \oDir_S(X)\to\oDir(S)$ be given by $\mu_S(d)=d|_S$. 
\end{definition}

\begin{proposition}\label{zzzz}
Let $X$ be an orthodomain with enough directions and $S$ be a shadow of $X$ that is proper. Then:
\begin{enumerate}
\item $\oDir_S(X)$ contains $0,1$ and is closed under $\,'\!$ and $\oplus$;
\item $\mu_S$ is a bijection from $\oDir_S(X)$ to $\oDir(S)$;
\item $\mu_S$ preserves $0,1$ and $\,'$;
\item $d\perp e$ if and only if $\mu_S(d)\perp\mu_S(e)$, and in this case $\mu_S(d\oplus e)=\mu_S(d)\oplus\mu_S(e)$.
\end{enumerate}
\end{proposition}

\begin{proof}
(1) Since $0,1$ are directions for $\bot$ and $\bot\in S$, we have $0,1\in\oDir_S(X)$. If $d$ is a direction for~$x$, then $d'$ is also a direction for~$x$, giving closure under $'$. For closure under $\oplus$, suppose $d,e\in\oDir_S(X)$ with $d$ a direction for $x\in S$, $e$ a direction for $y\in S$, and $d\perp e$. There are several cases for $\perp$. If one of $d,e$ is $0$, then $d\oplus e$ equals $d$ or $e$, and if $e=d'$, then $d\oplus e=1$, so these cases are trivial. In the remaining case $x$ and $y$ are near. Say $x\vee y = w$ with $z$ the third atom beneath $w$. Then $w\in S$ since $S$ is closed under joins in $X^*$, and so $z\in S$ since $S$ is a downset of $X$. Since $d\oplus e$ is a direction for~$z$, we have $d\oplus e\in\oDir_S(X)$. 

(2) For a basic $x \in S$, the two directions for $x$ in $X$ are $d$ and $d'$. These restrict to directions of $S$ for $x$ and their restrictions are orthocomplements. Then as $S$ has no basic maximal elements, these restrictions are distinct and are the only two directions for $x$ in $S$. Part (3) is trivial.

For part (4), suppose $d,e\in\oDir_S(X)$ with $d$ a direction for $x$ and $e$ a direction for~$y$. Note that one of $d,e$ is $0$ iff one of $\mu_S(d),\mu_S(e)$ is $0$, and in this case $\mu_S(d\oplus e)=\mu_S(d)\oplus\mu_S(e)$. Next, $e=d'$ iff $\mu_S(e)=\mu_S(d)'$, and in this case $\mu_S(d\oplus e)=\mu_S(d)\oplus\mu_S(e)$. For the remaining case we have $d\perp e$ if and only if $x,y$ are near and $d(w)=(x,w),e(w)=(y,w)$ where $x\vee y=w$. But this is equivalent to $\mu_S(d)\perp\mu_S(e)$. In this case, $d\oplus e$ is the direction for the third atom $z$ beneath $w$ with $(d\oplus e)(w)=(w,z)$, and thus its restriction is a direction for $z$ taking value $(w,z)$ at $w$, and hence is $\mu_S(d)\oplus\mu_S(e)$. 
\end{proof}

A specific instance of the previous proposition is of particular interest. It is easily seen that $X^*$ is a shadow of $X$ that has no basic maximal elements when $X$ has none. Furthermore, since every basic element of $X$ belongs to $X^*$, we have $\oDir_{X^*}(X)=\oDir(X)$.

\begin{corollary}\label{dsd}
If $X$ is an orthodomain with enough directions, then so is $X^*$, and restriction gives an isomorphism $\oDir(X)\simeq\oDir(X^*)$. 
\end{corollary}

We next set out to prove that $\oDir(X)$ is an orthoalgebra for any orthodomain $X$  with enough directions.

\begin{lemma}\label{associative}
For $X$ an orthodomain with enough directions, the partial binary operation $\oplus$ on $\oDir(X)$ is commutative and associative: when one side of an expression $d\oplus e=e\oplus d$ or $(d\oplus e)\oplus f=d\oplus (e\oplus f)$ is defined, so is the other, and the two are equal. 
\end{lemma}

\begin{proof}
Clearly $\oplus$ is commutative. Making use of this and symmetry, it suffices to show that if $(d\oplus e)\oplus f$ is defined, then $d\oplus(e\oplus f)$ is defined, and the two are equal. For this, we consider a number of cases. 

If any of $d,e,f$ are $0$, then it is easily verified. The only direction orthogonal to $1$ is $0$, so we may also assume that none of $d,e,d\oplus e,f$ is $1$. So there are atoms $x,y,z$ with $d$ a direction for~$x$, $e$ a direction for~$y$, and $f$ a direction for~$z$. Having $d=e'$ gives $d\oplus e=1$, so $x,y$ are distinct, and therefore to have $d\perp e$ we must have that $x,y$ are near. Let $x\vee y=w$ and let $p$ be the third atom beneath $w$. Then $d\oplus e$ is the direction for $p$ with $(d\oplus e)(w)=(w,p)$. 

Since neither $d\oplus e$ or $f$ equals $0$ or $1$, there are two possibilities to have $(d\oplus e) \perp f$. We consider first the case that $f=(d\oplus e)'$. Since $d\oplus e$ is the direction for $p$ with $(d\oplus e)(w)=(w,p)$, this means that $f$ is the direction for $z=p$ with $f(w)=(p,w)$. Since $x,y,p$ are the three pairwise near atoms under $w$, we then have that $e\oplus f$ is defined, and that $e\oplus f$ is the direction for $x$ with $(e\oplus f)(w)=(w,x)$. Thus $e\oplus f = d'$. So $d\oplus(e\oplus f)$ is also defined and both sides of the expression in this case evaluate to $1$. 

\begin{figure}[h]
\begin{tikzpicture}[inner sep=2pt, scale=1.25]
 \node (0) at (0,0) {$\bot$};
 \node (x) at (-2,1) {$x$};
 \node (y) at (-1,1) {$y$};
 \node (p) at (0,1) {$p$};
 \node (z) at (1,1) {$z$};
 \node (q) at (2,1) {$q$};
 \node (w) at (-1,2) {$w$};
 \node (v) at (1,2) {$v$};
 \node (u) at (0,3) {$u$};
 \draw (0) to (x);
 \draw (0) to (y);
 \draw (0) to (p);
 \draw (0) to (q);
 \draw (0) to (z);
 \draw (x) to (w);
 \draw (y) to (w);
 \draw (p) to (w);
 \draw (p) to (v);
 \draw (q) to (v);
 \draw (z) to (v);
 \draw (w) to (u);
 \draw (v) to (u);
\end{tikzpicture}
\quad\quad
\begin{tikzpicture}[inner sep=2pt, scale=1.25]
    \node (0) at (0,0) {$\bot$};
    \node (a) at (-3,.8) {$x$};
    \node (b) at (-2,.8) {$y$};
    \node (c) at (-1,.8) {$z$};
    \node (d) at (0,.8) {$q$};
    \node (e) at (1,.8) {$p$};
    \node (f) at (2,.8) {$\cdot$};
    \node (g) at (3,.8) {$\cdot$};
    \node (m1) at (-2.5,2.25) {$w$};
    \node (m2) at (-1.5,2.25) {$\cdot$};
    \node (m3) at (-.5,2.25) {$\cdot$};
    \node (m4) at (.5,2.25) {$\cdot$};
    \node (m5) at (1.5,2.25) {$\cdot$};
    \node (m6) at (2.5,2.25) {$v$};
    \node (B) at (0,3.25) {$u$};
    \draw (0) to (a.south); \draw (0) to (b.south); \draw (0) to (c.south); \draw (0) to (d.south); \draw (0) to (e.south); \draw (0) to (f.south); \draw (0) to (g.south); 
    \draw (B) to (m1.north); \draw (B) to (m2.north); \draw (B) to (m3.north); \draw (B) to (m4.north); \draw (B) to (m5.north); \draw (B) to (m6.north);
    \draw (a.north) to (m1.south); \draw (a.north) to (m2.south); \draw (a.north) to (m3.south);
    \draw (b.north) to (m1.south); \draw (b.north) to (m4.south); \draw (b.north) to (m5.south);
    \draw (c.north) to (m2.south); \draw (c.north) to (m4.south); \draw (c.north) to (m6.south);
    \draw (d.north) to (m3.south); \draw (d.north) to (m5.south); \draw (d.north) to (m6.south);
    \draw (e.north) to (m1.south); \draw (e.north) to (m6.south); 
    \draw (f.north) to (m2.south); \draw (f.north) to (m5.south); 
    \draw (g.north) to (m3.south); \draw (g.north) to (m4.south); 
  \end{tikzpicture}

\caption{A part of the Hasse diagram of the shadow from the proof of Lemma~\protect\ref{associative} (two possible partial diagrams of the same situation)
\label{part}}
\end{figure}
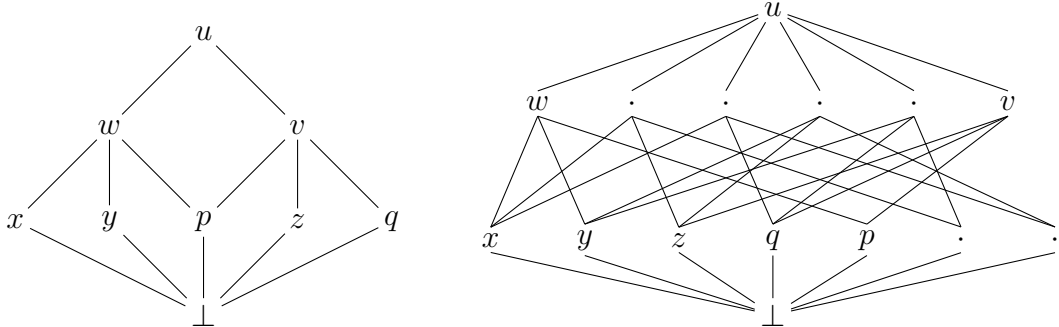

For the final case (see Figure~\ref{part}), it must be that item (3) in Definition~\ref{perp} applies to $(d\oplus e)\perp f$. Since $d\oplus e$ is the direction for $p$ with $(d\oplus e)(w)=(w,p)$ and $f$ is a direction for~$z$, the assumptions of (3) give that $p,z$ are near. Say $p\vee z = v$, and let $q$ be the third atom distinct from $p,z$ under~$v$. Then to have $d\oplus e\perp f$ we have $(d\oplus e)(v)=(p,v)$ and $f(v)=(z,v)$, and the sum $(d\oplus e)\oplus f$ is the direction for $q$ with $\bigl((d\oplus e)\oplus f\bigr)(v)=(v,q)$. 

Since $(d\oplus e)(w)=(w,p)$ and $(d\oplus e)(v)=(p,v)$, we have $w\neq v$. Since the three atoms beneath $w$ are $x,y,p$, the three atoms beneath $v$ are $p,q,z$, and $w,v$ cannot have more than one common atom beneath them since they are distinct, we have that $x,y,p,q,z,w,v$ are distinct. Since $d\oplus e$ is a direction for $p$, by Definition~\ref{scary}
\[\bigl( p\lessdot w,v,\ (d\oplus e)(w)=(w,p) \,\mbox{ and }\,(d\oplus e)(v)=(p,v)\bigr) \ \Rightarrow\  w\vee v  \mbox{ exists and } w,v\lessdot w\vee v\,.\]
Let $u=w\vee v$. Since $p$ is an atom and $p\lessdot w,v\lessdot u$ then $u$ has height 3 so belongs to~$X^*$. Let $S=\down u$ and note that this is a shadow of $X$. Since $S$ is isomorphic to $\Sub(B)$ for a $16$-element Boolean algebra $B$, Theorem~\ref{hhh} gives that $\oDir(S)\simeq B$. Proposition~\ref{zzzz} gives $\oDir_S(X)\simeq\oDir(S)$. Since $d,e,f$ all belong to $\oDir_S(X)$, their associativity under $\oplus$ follows. 
\end{proof}

\begin{theorem}\label{isOA}
If $X$ is an orthodomain with enough directions, then $\oDir(X)$ is a proper orthoalgebra. 
\end{theorem}

\begin{proof}
Lemma~\ref{associative} shows that $\oplus$ is commutative and associative. There are directions $0,1$. For each direction $d$ also $d'$ is a direction, $d\oplus d'$ is defined, and $d\oplus d'=1$. Suppose $e$ is another direction with $d\oplus e$ defined and $d\oplus e=1$. Since $1$ is a direction given by the basic element~$0$, it cannot be that $d\perp e$ because of reason~(3) in Definition~\ref{perp}. If it is defined because of reason~(2), then $e=d'$. If it is defined because of reason~(1), then one of $d,e$ is $0$, and because we have required $d\oplus e=1$, the other must be $1$, hence again $e=d'$. So $d'$ is the unique direction with $d\oplus d'=1$. Finally, suppose that $d$ is a direction with $d\oplus d$ defined. This cannot be defined because of reason~(3) of Definition~\ref{perp}. It cannot be because of reason~(2) since $d\neq d'$. So it must be defined because of reason~(1), giving $d=0$.
	
	To show that $\oDir(X)$ is proper, let $d\neq0,1$ be a direction for the basic element $x\in X$, so $x$ must be an atom. We show that $\{0,d,d',1\}$ is properly contained in some other Boolean subalgebra of $\oDir(X)$. Since $X$ is proper, $x$ cannot be maximal, hence $x<v$ for some $v\in X$. Then $\down v$ is a Boolean domain, hence it must contain some $w$ such that $x\lessdot w$. Let $y$ be some other atom of $X$ such that $y\lessdot w$, so $x$ and $y$ are near. Since $(x,w)$ and $(w,x)$ are principal pairs for $x$ in $\down w$, it follows that either $d(w)=(x,w)$ or $d'(w)=(x,w)$. Without loss of generality, assume the former. Let $e$ be the direction for $y$ with $e(w)=(y,w)$. By Definition \ref{perp}, $d\perp e$ is defined. Since $e$ is a direction for $y$ and $y\neq 0,x$ we have $e\neq 0,1$ and $e$ is distinct from $d,d’$, hence $d$ and $e$ generate a Boolean subalgebra of eight elements, which properly contains $\{0,d,d',1\}$. It follows that any small Boolean subalgebra of $\oDir(X)$ is properly contained in a larger Boolean subalgebra, hence $\oDir(X)$ cannot have small blocks. 
\end{proof}

\begin{remark}\label{r:nonisomorphic}
For an orthodomain $X$ with enough directions, $X \simeq \BSub(\oDir(X))$ does not usually hold. By Corollary~\ref{dsd} $\oDir(X)\simeq\oDir(X^*)$, and we clearly do not have $X\simeq X^*$ for each orthodomain $X$ with enough directions. In fact, $X=\BSub(A)$ provides a counterexample for any orthoalgebra $A$ with no small blocks and a block with more than 4 atoms. 
\end{remark}

\begin{definition} 
A shadow $S \subseteq X^*$ of an orthodomain $X$ is a \emph{Boolean shadow} if either:
\begin{enumerate}
\item $S=\down x$ for some basic $x\in X$;
\item $S$ has enough directions and $\oDir(S)$ is a Boolean orthoalgebra. 
\end{enumerate}
Write $\BShad(X)$ for the partially ordered set of Boolean shadows of $X$ under inclusion.
\end{definition}

\begin{definition}\label{arw}
Let $X$ be an orthodomain with enough directions, and let $B$ be a Boolean subalgebra of $\oDir(X)$. Define:
\begin{align*}
T_B& = \{x \mid x \text{ is basic in $X$ and there is some $d\in B$ with $d$ a direction for~$x$}\}\text,\\
S_B& = \text{the closure of $T_B$ under existing joins in $X^*$}\text.
\end{align*}
\end{definition}

\begin{proposition}\label{dhd}
Let $X$ be an orthodomain with enough directions and let $B$ be a Boolean subalgebra of $\oDir(X)$. Then:
\begin{enumerate}
\item if $B$ has more than $4$ elements, then $S_B$ is proper and $B=\oDir_{S_B}(X)$;
\item $S_B$ is a Boolean shadow of $X$.
\end{enumerate}
\end{proposition}

\begin{proof}
We first prove that $S_B$ is a shadow of $X$.

By definition, $S_B\subseteq X^*$ and is closed under existing joins in $X^*$. It suffices to show that $S_B$ is a downset of $X^*$. Clearly if $w$ is a basic element of $X$ that belongs to $S_B$, then any $x\leq w$ also belongs to $S$. This covers the case that $w$ is of height $0$ or~$1$. Suppose $w\in S_B$ is of height $2$ in~$X$. Then $w$ belonging to $S_B$ means it is the join $w=x\vee y$ of two elements $x,y$ of $T_B$, both of which are atoms of $X$. Since $x,y\in T_B$ there are directions $d,e\in B$ with $d$ a direction for $x$ and $e$ a direction for~$y$. Furthermore, since $x$ and $y$ are near, these may be chosen so that $d\perp e$. If $z\leq w$, then either $z$ is one of $0,x,y,w$, or $z$ is the third atom beneath $x\vee y$. In the last case $d\oplus e$ is a direction for $z$ and $d\oplus e$ belongs to $B$, so $z\in T_B\subseteq S_B$. 

Our final case is when $w\in S_B$ is of height~$3$. Since $w\in S_B$, it is the join of atoms of $T_B$. Since $w\in X$ we have $\down w$ isomorphic to the poset of subalgebras of a $16$-element Boolean algebra as shown in Figure~\ref{fig1}. Our task is to show that all $7$ atoms $x$ in $\down w$ belong to $T_B$ since this then shows that all elements of height $2$ in $\down w$ are in $S_B$. The atoms in $\down w$ can be divided into two groups, the four at left and the three at right. Note that the latter three atoms are not near to each other. There are two possibilities to consider: 
\begin{itemize}
\item[(i)]  $w$ is the join of two atoms $w=x\vee y$ from the right with $x,y\in T_B$;
\item[(ii)] $w$ is the irredundant join of 3 atoms of $T_B$. 
\end{itemize}
  
Using the result above that if $z\in S_B$ is of height 2 then $\down z\subseteq S_B$, and that $S_B$ is closed under joins in $X^*$, the second case can be reduced to the first, so we just consider the first. 
Since $w$ has height~$3$, it follows that $\oDir(\down w)$ must be a 16-element Boolean algebra. Since $\down w$ is a proper shadow of $X$, Proposition \ref{zzzz} yields an isomorphism between $\oDir_{\down w}(X)$ and $\oDir(\down w)$. Let $d,d'$ be the directions for $x$ and $e,e'$ be the directions for~$y$. So $d,d',e,e'\in B\cap\oDir_{\down w}(X)$ and neither of $d,d'$ is orthogonal to either of $e,e'$, since $x$ and $y$ are not near. Since $B$ is a Boolean subalgebra of $\oDir(X)$, then $d,e$ generate a $16$-element subalgebra $Y$ of $B\cap\oDir_{\down w}(X)$, which must be equal to $\oDir_{\down w}(X)$, since the latter is also a $16$-element Boolean algebra. Hence $\oDir_{\down w}(X)\subseteq B$, which expresses that any direction for any atom in $\down w$ is contained in $B$. Hence all atoms of $\down w$ are contained in $T_B$, establishing that $S_B$ is a shadow.

For part~(1), assume $B$ has more than 4 elements. To see that $S_B$ is proper, first note that it is not the case that 0 is maximal in $S_B$. Let $x$ be an atom in $S_B$ and hence in $T_B$. Then there is a direction $d$ in $B$ that is a direction for~$x$. Since $B$ has more than 4 elements, there is a nonzero direction $e$ in $B$  orthogonal and unequal to one of $d$ or $d'$. If $e$ is a direction for~$y$, then $x,y$ are near, and $w=x\vee y\in S_B$. So no atom is maximal in $S_B$, and $S_B$ is proper. 

It remains to show that $B=\oDir_{S_B}(X)$. Let $d$ be a direction in $B$ for a basic element~$x$. Then by definition, $x\in T_B\subseteq S_B$. Thus by definition $d\in\oDir_{S_B}(X)$. Conversely, let $d\in\oDir_{S_B}(X)$ be a direction for the basic element $x$ of~$X$. By definition, $x\in S_B$. But $S_B$ consists of the elements of $X^*$ that are joins of elements of $T_B$, and as $x$ is basic, it must be that $x\in T_B$. Thus there is a direction $e$ in $B$ for~$x$. But there are only two directions for~$x$, namely $d,d'$. So either $e=d$ or $e'=d$, and in either case $d$ is in $B$ since $B$ is closed under orthocomplementation.

For part~(2), it remains to show that the shadow $S_B$ is Boolean. If $B$ has 4 or fewer elements, then $T_B=\down x$ for a basic element $x$, so $S_B=T_B$, and so $S_B$ is Boolean. Suppose $B$ has more than 4 elements. Then by (1) $B=\oDir_{S_B}(X)$. Proposition~\ref{zzzz} gives $\oDir_{S_B}(X)\simeq\oDir(S_B)$. So $\oDir(S_B)$ is Boolean, giving that $S_B$ is a Boolean shadow.
\end{proof}

\begin{proposition}\label{gamp}
For $X$ an orthodomain with enough directions, there is an isomorphism of posets $\Gamma \colon \BSub(\oDir(X))\to\BShad(X)$ given by $\Gamma(B)=S_B$. 
\end{proposition}

\begin{proof}
The map is well-defined by Proposition~\ref{dhd}. If $B_1\subseteq B_2$, then surely $S_{B_1}\subseteq S_{B_2}$, so $\Gamma$ preserves order. Suppose $S_{B_1}\subseteq S_{B_2}$. Since elements of $S_{B_2}$ are joins of elements of $T_{B_2}$, and elements of $T_{B_1}$ are basic and hence join irreducible, this implies that $T_{B_1}\subseteq T_{B_2}$ and this gives that $B_1\subseteq B_2$. So $\Gamma$ is an order embedding. 

To see that it is surjective, let $S$ be a Boolean shadow of $X$. If $S$ is either $\{\bot\}$ or $\{\bot,x\}$ for some atom $x$ of $X$, then $S=\Gamma(B)$ where $B=\{0,1\}$ or $B=\{0,d,d',1\}$ where $d$ is a direction for~$x$ respectively. Suppose that $S$ has enough directions and $\oDir(S)$ is a Boolean orthoalgebra. Let $B=\oDir_S(X)$. By Proposition~\ref{zzzz}, $B$ is a subalgebra of $\oDir(X)$ and the restriction map from $B$ to $\oDir(S)$ is an isomorphism, so $B$ is a Boolean subalgebra of $\oDir(X)$. Then $\Gamma(B)=S_B$ is the shadow generated by $T_B$, and the elements of $T_B$ are those basic elements $x$ of $X$ that have a direction $d\in B=\oDir_S(X)$. By definition, the elements of $\oDir_S(X)$ are those directions that are for some basic $x\in S$. Thus $T_B$ consists of the basic elements in $S$, so $\Gamma(B)=S$. 
\end{proof}

\begin{definition}
Let $X$ be an orthodomain.
We say $X$ is \emph{short} if $X=X^*$. We say $X$ is \emph{tall} if $m=\bigvee S$ exists and $\down m\cap X^*=S$ for each Boolean shadow $S$. 
\end{definition}

\begin{proposition}\label{weep}
Let $A$ be a proper orthoalgebra. Then $X=\BSub(A)$ is a tall orthodomain with enough directions. 
\end{proposition}

\begin{proof}
By Theorem~\ref{hhh}, $X$ is an orthodomain with enough directions, and there is an ortho-algebra isomorphism $\xi \colon A\to\oDir(X)$ where $\xi(a) = d_a$ is the direction for $x_a$ with 
\[d_a(w)=(\down_w a\cup\up_w a',\down_w a'\cup\up_w a)\text.\]
Let $S$ be a Boolean shadow of $X$. If $S=\down x$ for a basic element $x$ it is clear that $x=\bigvee S$ exists and $\down x\cap X^*=S$. Assume that $S$ has enough directions and $\oDir(S)$ is Boolean. By Proposition~\ref{zzzz} 
$\oDir(S)\simeq\oDir_{S}(X)$, hence $\oDir_S(X)$ is Boolean. Let $m=\xi^{-1}[\oDir_S(X)]$. Then $m$ is a Boolean subalgebra of $A$, and consists of all the $a\in A$ with $\xi(a)\in\oDir_S(X)$, hence all $a\in A$ with $d_a\in\oDir_S(X)$, and therefore all $a\in A$ with $x_a\in S$. Since each basic element of $X$ is of the form $x_a$ given by some $a\in A$, we have for a basic element $x\in X$, that $x\in S$ exactly when $x\leq m$. Since $S$ is a downset and $X$ is atomistic $m=\bigvee S$, and since $S$ is closed under existing joins in $X^*$ we have $\down m \cap X^*=S$. Thus $X$ is tall. 
\end{proof}

The following proposition says that the situation described in Remark~\ref{r:nonisomorphic} cannot happen for tall orthodomains.

\begin{proposition}\label{lopp}
If $X$ is a tall orthodomain with enough directions, then $X\simeq\BSub(\oDir(X))$. 
\end{proposition}

\begin{proof}
By Proposition~\ref{gamp}, we have $\BSub(\oDir(X))\simeq\BShad(X)$. Define $\psi\colon\BShad(X)\to X$ by $\psi(S)=\bigvee S$, and $\lambda\colon X\to\BShad(X)$ by $\lambda(m)=\down m\cap X^*$. Since $X$ is tall, $\bigvee S$ exists and $\down (\bigvee S)\cap X^*=S$. For any $w\in X$ we have $\down w\cap X^*$ is a downset of $X^*$ that is closed under existing joins in $X^*$, hence is a shadow. If $w$ is basic in $X$, then by definition $\down w$ is a Boolean shadow. Otherwise $\down w$ is a proper Boolean domain, hence has enough directions. Since $\down w\cap X^*=(\down w)^*$, Corollary~\ref{dsd} gives that $\down w\cap X^*$ is an orthodomain with enough directions and that $\oDir(\down w\cap X^*)$ is isomorphic to $\oDir(\down w)$, and hence is Boolean. In any case, $\down w\cap X^*$ is a Boolean shadow of $X$. So $\psi$ and $\lambda$ are well-defined. Since $\down (\bigvee S)\cap X^*=S$ we have $\lambda\circ\psi=\mathrm{id}$. For $w\in X$, by atomisticity $w=\bigvee(\down w\cap X^*)$, so $\psi\circ\lambda=\mathrm{id}$. Thus $\BShad(X)\simeq X$, so $\BSub(\oDir(X))\simeq X$. 
\end{proof}

\begin{theorem}
The following are equivalent for an 
orthodomain~$X$:
\begin{enumerate}
\item $X$ is tall and has enough directions;
\item $X\simeq\BSub(A)$ for a proper orthoalgebra $A$.
\end{enumerate}
When these conditions hold, $\oDir(X)$ is an orthoalgebra and $X\simeq\BSub(\oDir(X))$. 
\end{theorem}

\begin{proof}
The direction (1) $\Rightarrow$ (2) follows from  Theorem~\ref{isOA} and Proposition~\ref{lopp}. The direction (2) $\Rightarrow$ (1) follows from Theorem~\ref{hhh} and Proposition~\ref{weep}. 
\end{proof} 

\begin{remark}\label{ruut}
The previous theorem only characterizes orthodomains of the form $\BSub(A)$ for a proper orthoalgebra $A$. This can be extended to orthoalgebras $A$ with small blocks as follows. If $A$ has two or fewer elements, then $\BSub(A)$ is a 1-element orthodomain. If $A$ has more than two elements and all its blocks are small, then $\BSub(A)$ is an orthodomain where all elements are basic, and each orthodomain where all elements are basic arises this way as the horizontal sum of 4-element Boolean algebras, one for each atom of the orthodomain. Otherwise not all blocks of $A$ are small. Let $\tA $ be the orthoalgebra obtained by removing small blocks from $A$. Then $\BSub(\tA )$ is a tall orthodomain with enough directions, and $\BSub(A)$ is obtained from this by adding a maximal atom to $\BSub(\tA )$ for each small block of $A$. So the orthodomains isomorphic to $\BSub(A)$ for some orthoalgebra $A$ are exactly those that have one element, have all their elements basic, or are constructed by adding a set of maximal atoms to a tall orthodomain with enough directions. 
\end{remark}

\begin{theorem}
The following are equivalent for an orthodomain~$X$:
\begin{enumerate}
\item $X$ is short and has enough directions;
\item $X\simeq\BSub(A)^*$ for a proper orthoalgebra $A$.
\end{enumerate}
When these conditions hold, $\oDir(X)$ is an orthoalgebra and $X\simeq\BSub(\oDir(X))^*$. 
\end{theorem}

\begin{proof}
The direction (2) $\Rightarrow$ (1) follows from Theorem~\ref{hhh} and Corollary~\ref{dsd}. For the converse, assume (1). Theorem \ref{isOA} assures that $\oDir(X)$ is a proper orthoalgebra. By Proposition~\ref{gamp} there is an isomorphism $\Gamma\colon\BSub(\oDir(X))\to\BShad(X)$ given by $\Gamma(B)=S_B$ where $S_B$ is from Definition~\ref{arw}. Then $\Gamma$ restricts to an isomorphism of posets $\Gamma'\colon\BSub(\oDir(X))^*\to\BShad(X)^*$. We will show that $\BShad(X)^*$ is equal to the poset of principal downsets $\down w$ where $w\in X$, hence is isomorphic to $X$. This will show that $\BSub(\oDir(X))^*$ is isomorphic to $X$, establishing (2) and the further remark. 

Suppose $w\in X$. If $w$ is basic, then by definition $\down w$ is a Boolean shadow that clearly has height at most 1 in $\BShad(X)$. Otherwise $\down w$ is a Boolean domain with enough directions and $\oDir(\down w)$ is a Boolean algebra. Thus $\down w$ is a Boolean shadow. Since $X$ is short, $w$ has height at most 3, so $\down w$ has height at most 3 in $\BShad(X)$, so belongs to $\BShad(X)^*$. 

From the isomorphism $\Gamma'$, the elements of $\BShad(X)^*$ are the $S_B$ where $B$ is a Boolean subalgebra of $\oDir(X)$ with at most $16$ elements. We must show that all such $S_B$ are equal to $\down w$ for some $w\in X$. If $B$ has 4 or fewer elements then $S_B$ is equal to $\down w$ for a basic element $w\in X$. Suppose $B$ has $8$ elements. Let $d_1,d_2,d_3$ be the directions that are the atoms of $B$ and assume $d_i$ is a direction for the basic element $x_i$ of $X$ for $i=1,2,3$. Since $d_1$ is orthogonal to $d_2$ we have that $x_1,x_2$ are near, so have a join $w=x_1\vee x_2$, and this belongs to $S_B$. By simple counting, $S_B$ must be equal to $\down w$. Finally, suppose that $B$ has $16$ elements and $d_1,\ldots,d_4$ are the atoms of $B$ with $d_i$ a direction for $x_i$ for $i=1,\ldots,4$. Since $d_1,d_2$ are orthogonal $x_1,x_2$ are near, so $z=x_1\vee x_2$ exists. Suppose $x$ is the third atom beneath $z$. Then $(d_1\oplus d_2)(z)=(z,x)$. Let $y=x_3\vee x_4$. Since $d_3\oplus d_4=(d_1\oplus d_2)'$ we have that the third atom under $y$ is $x$. Also $(d_3\oplus d_4)(y)=(y,x)$, so $(d_1\oplus d_2)(y)=(x,y)$. Then by condition~(3) of Definition~\ref{scary} $w=y\vee z$ exists and has height 3. Then simple counting gives that $S_B=\down w$. 
\end{proof}

\begin{remark}
The previous theorem extends to small blocks as in Remark~\ref{ruut}. This provides a bijective correspondence between isomorphism classes of orthoalgebras and isomorphism classes of short orthodomains where each basic element has a direction. Here we use the obvious fact that a maximal atom has a direction. 
\end{remark}

\begin{remark}
We have shown that for an orthoalgebra $A$, the poset $\BSub(A)^*$ of elements of height $3$ or less in $\BSub(A)$ is sufficient to reconstruct $A$. We will show one cannot make due with the order structure of the elements of height $2$ or less. Specifically, for an orthodomain~$X$, let $X^\dagger$ be the poset of elements of height $2$ or less in $X$. We will give two non-isomorphic orthoalgebras $A$ and $C$ where $\BSub(A)^\dagger$ and $\BSub(C)^\dagger$ are isomorphic. 

Let $A$ be the $16$-element Boolean algebra shown in Figure~\ref{fig1}. The diagram below shows the Fano plane minus a single line, the circle connecting the middle elements of each side. This figure gives a poset $P$ with bottom $\bot$, seven atoms given by the vertices of this figure, and six elements of height $2$ given by the lines of the figure, with the understanding that a vertex lies beneath a line if it lies on the line. Then $P$ is isomorphic to the elements $\BSub(A)^\dagger$ of height $2$ or less in $\BSub(A)$. 

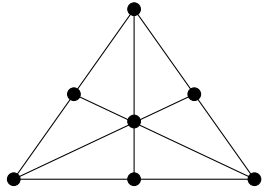
\begin{figure}[h]
\begin{center}
\begin{tikzpicture}[scale = 1.6]
\draw[fill] (0,0) circle(1.5pt); \draw[fill] (1,0) circle(1.5pt); \draw[fill] (2,0) circle(1.5pt); \draw[fill] (1,1.414) circle(1.5pt); \draw[fill] (1,.48) circle(1.5pt); \draw[fill] (.5,.707) circle(1.5pt); \draw[fill] (1.5,.707) circle(1.5pt); 
\draw[thin] (0,0)--(2,0)--(1,1.414)--(0,0)--(1.5,.707); \draw[thin] (1,1.414)--(1,0); \draw[thin] (.5,.707)--(2,0);
\end{tikzpicture}
\end{center}
\caption{The hypergraph view of the subalgebras of a $16$-element Boolean algebra}
\label{fig:hyp}
\end{figure}

However, it follows from the usual Greechie hypergraph representation of orthoalgebras (see \cite{N:handQL} with a correction~\cite{Navara}) that this poset also represents the atom structure of an orthoalgebra $C$. This means that there is an orthoalgebra $C$ whose blocks all have $8$ elements where the atoms of $C$ are the vertices of this figure, and the sets of atoms forming a block of $C$ are exactly the vertices lying on a line in the figure. Then $\BSub(C)$ is isomorphic to $P$, and as every element of it is of height 2 or less, $\BSub(C)^\dagger=\BSub(C)$. Thus $\BSub(A)^\dagger \simeq\BSub(C)^\dagger$, but $A\not\simeq C$. 

It turns out that such pathologies do not exist for orthomodular posets --- an orthomodular poset $A$ is determined up to isomorphism by the elements of height at most 2 in $\BSub(A)$. See Theorem~\ref{OMPdetermined} for details.
\end{remark}

\section{Hypergraphs}\label{sec:hypergraphs}

In this section, we begin the process of making a categorical view of the correspondence between orthoalgebras and their structures of Boolean subalgebras. Here we refine the object level correspondence suggested at the end of the previous section between orthoalgebras $A$ and their posets $\BSub(A)^*$ of Boolean subalgebras having at most $16$ elements. We treat these posets graph-theoretically, to be precise as certain hypergraphs, in a way that seems more intuitive. The next section introduces morphisms between these hypergraphs and relates the resulting categories. 

\begin{definition}
A {\em hypergraph} is a triple $\mathcal{G}=(P,L,T)$ consisting of a set $P$ of {\em points}, a set $L$ of {\em lines}, and a set $T$ of {\em planes}. A line is a set of 3 points, and a plane is a set of 7 points where the restriction of the lines to these 7 points is as shown below. 
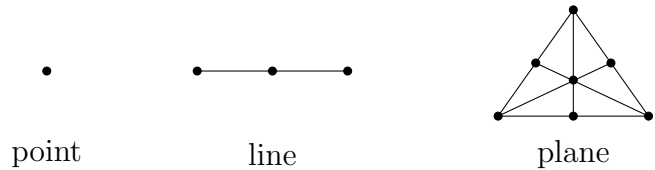
\vspace{2ex}

\begin{center}
\begin{tikzpicture}[scale = 1]
\draw[fill] (-6,.6) circle(1.5pt); \draw[fill] (-4,.6) circle(1.5pt);\draw[fill] (-3,.6) circle(1.5pt);\draw[fill] (-2,.6) circle(1.5pt);
\draw[thin] (-4,.6)--(-2,.6);
\node at (-6,-.5) {point}; \node at (-3,-.5) {line}; \node at (1,-.5) {plane};
\draw[fill] (0,0) circle(1.5pt); \draw[fill] (1,0) circle(1.5pt); \draw[fill] (2,0) circle(1.5pt); \draw[fill] (1,1.414) circle(1.5pt); \draw[fill] (1,.48) circle(1.5pt); \draw[fill] (.5,.707) circle(1.5pt); \draw[fill] (1.5,.707) circle(1.5pt); 
\draw[thin] (0,0)--(2,0)--(1,1.414)--(0,0)--(1.5,.707); \draw[thin] (1,1.414)--(1,0); \draw[thin] (.5,.707)--(2,0);
\end{tikzpicture}
\end{center}
\end{definition}
 
Note that a set of 7 points, where the lines among them are as in a plane, need not be a plane. Having lines among 7 points as shown is a necessary condition to be a plane, but not a sufficient one. If one is drawing a picture of a hypergraph, planes should be placed in circles to indicate that they are indeed planes. Note also that a plane has two types of points. One type, called a \emph{corner point}, lies on 3 lines of the plane; the other type, called an \emph{edge point}, lies on 2 lines of the plane. Each plane has 4 corner points (the middle is a corner) and 3 edge points. 

\begin{definition}
  A subset $S$ of the points of a hypergraph $\mathcal{G}$ is a \emph{subgraph} if it is closed under lines and planes:
  \begin{itemize}
    \item if two points of a line belong to $S$, then so does the third point; 
    \item if two points, that do not lie on a line but do lie in a plane, belong to $S$, then so do all the other points of that plane. 
  \end{itemize}
  The smallest subgraph containing a set of points is called the subgraph it {\em generates}. 
\end{definition}

Observe that a subgraph $S$ of $\mathcal{G}$, together with the sets of lines and planes of $\mathcal{G}$ that contain only elements from $S$, forms a hypergraph. Each point, line, and plane of any hypergraph is a subgraph, but there can be others. 

\begin{definition}
  Each orthoalgebra $A$ defines a hypergraph $\mathcal{G}(A)=(P,L,T)$ by:
  \begin{itemize}
    \item points are the atoms of $\BSub(A)$;
    \item three points form a line if they are the atoms under an element of $\BSub(A)$ of height~$2$, so if they are near;
    \item seven points form a plane if they are the atoms under an element of $\BSub(A)$ of height~$3$. 
  \end{itemize}
  A hypergraph $\mathcal{G}$ is an \emph{orthohypergraph} if it is isomorphic to $\mathcal{G}(A)$ for some orthoalgebra $A$. It is a Boolean hypergraph if it is isomorphic to $\mathcal{G}(B)$ for a Boolean algebra $B$. 
\end{definition}

Notice that the points of $\mathcal{G}(A)$ are all $x_a$, $a\in A\setminus\{0,1\}$.
The Boolean hypergraph of a 4-element Boolean algebra is a single point, that of an $8$-element Boolean algebra is 3 points arranged in a single line, and that of a $16$-element Boolean algebra is 7 points arranged into a plane. The corner points in the plane correspond to Boolean subalgebras of the form $x_a$ for $a$ an atom in the ambient 16-element Boolean algebra. The edge points correspond to Boolean subalgebras of the form $x_a$ where $a$ is an element of height 2 in the ambient Boolean algebra. See also Figure \ref{fig1}, where the four left atoms correspond to corner points, and the three right atoms to edge points.

\begin{theorem}\label{qwqq}
  For an orthoalgebra $A$, the poset $\BSub(A)^*$ of Boolean subalgebras of $A$ of height $3$ or less can be reconstructed from $\mathcal{G}(A)$. Thus, if $A$ is proper, then up to isomorphism $A$ can be reconstructed from its hypergraph $\mathcal{G}(A)$ as the directions of $\BSub(A)^*$. 
\end{theorem}
\begin{proof}
  Construct a poset isomorphic to $\BSub(A)$ as follows. Let $\bot$ be its bottom and let its atoms be the points of $\mathcal{G}(A)$.  Elements of height 2 are the lines, and lie above the atoms they contain. Elements of height 3 are the planes, and lie above the atoms they contain. The poset so constructed is isomorphic to $\BSub(A)^*$. Thus, when $A$ is proper, we can then reconstruct $A$ from its hypergraph $\mathcal{G}(A)$ via the directions of $\BSub(A)^*$. 
\end{proof}

\begin{remark}
Theorem~\ref{qwqq} shows how to characterize orthohypergraphs among hypergraphs. Given a hypergraph $\mathcal{G}$, one constructs a poset $X$ of height at most 3 from it as described in the theorem. Then $\mathcal{G}$ is an orthohypergraph if and only if this poset is an orthodomain with enough directions. Due to the nature of the poset constructed from $\mathcal{G}$, several conditions of the definition of an orthodomain, Definition~\ref{orthodomain}, are automatically satisfied. Finite height implies directed subsets have joins, it is atomistic, and each principle ideal is a Boolean domain. Only the fourth condition needs to be verified, and this says that two if two points lie on a line, then they do not both lie on another line or both belong to a plane that does not contain this line. Determining whether this orthodomain has enough directions is more problematic, as it is in the orthodomain setting. However, the conditions to be a direction are easily translated into the hypergraph setting (see Lemma~\ref{deer} below), and are easier to work with in this way. 
\end{remark}

In the following, we consider the hypergraph $\mathcal{G}(A)$ of an orthoalgebra $A$. By definition, a point $p$ of $\mathcal{G}(A)$ is an atom of $\BSub(A)$, and a line $l$ of $\mathcal{G}(A)$ that contains $p$ is a cover of $p$ in $\BSub(A)$. 
The following observations will be used in the next section.

\begin{proposition}
  Let $A$ be an orthoalgebra and $\mathcal{G}(A)$ be its hypergraph. 
  \begin{enumerate}
    \item If two lines have at least $2$ points in common, then they are equal. 
    \item If $A$ is a Boolean algebra then any two points of $\mathcal{G}(A)$ generate a line or plane. 
  \end{enumerate}
\end{proposition}

\begin{proof}
    (1) Suppose that lines $l,m$ have common points $x_a,x_b$ for some $a,b\in A$ with none of $a,a',b,b'$ equal to $0,1$ or to each other. Furthermore, $l$ and $m$ are $8$-element Boolean subalgebras of $A$ that contain $a,b$. So in $A$, one of $a,a'$ is orthogonal to one of $b,b'$. Suppose $a$ is orthogonal to~$b$. Then the atoms of $l$ are $a,b,(a\oplus b)'$ and the atoms of $m$ are $a,b,(a\oplus b)'$. Thus $l=m$. 

    (2) Suppose $x_a,x_b$ are distinct points, where $a,a',b,b'$ are distinct from $0,1$ and from each other. Let $S$ be the subalgebra of the Boolean algebra $A$ generated by $a,b$. Then $S$ has at least $8$ elements from the properties of $a,a',b,b'$, and $S$ has at most $16$ elements because the free Boolean algebra generated by a 2-element set has $16$ elements. If $S$ has $8$ elements, then it is a line of $\mathcal{G}(A)$ that contains $x_a,x_b$, and if $S$ has $16$ elements, then it is a plane that contains $x_a,x_b$. 
\end{proof}

As is shown in Remark~5.20, the elements of height 2 or less in the poset $\BSub(A)$ do not determine $A$ in the general case when $A$ is an orthoalgebra. Thus points, lines and planes are necessary in some cases to describe the hypergraph of an orthoalgebra. However, if it is known that the orthoalgebra is an orthomodular poset, one can do better, as we will now show. 

\begin{proposition}
	\label{loppo}
	Let $A$ be an orthomodular poset. If the hypergraph $\mathcal{G}(A)$ has points and lines configured as the hypergraph of a 16-element Boolean algebra, then these points and lines are the points and lines of a plane of $\mathcal{G}(A)$. 
\end{proposition}

\begin{proof}
	The assumptions provide that there are elements $a,b,c,d,e,f,g\in A$ different from $0,1$ such that the points $x_a,x_b,x_c,x_d,x_e,x_f,x_g$ of the hypergraph are configured as indicated below. 
	
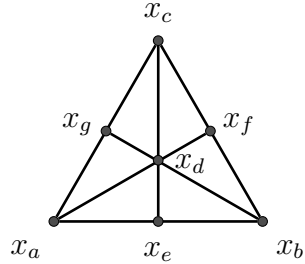
\[\begin{aligned}
	\begin{tikzpicture}[scale=1.85]
	\node[Dot] (a) at (0,0) {};
	\node[Dot] (b) at (1.5,0) {};
	\node[Dot] (e) at (.75,1.3) {};
	\node[Dot] (f) at (.75,.44) {};
	\node[Dot] (af) at (1.125,.65) {};
	\node[Dot] (ab) at (.75,.0) {};
	\node[Dot] (bf) at (.375,.65) {};
	\draw[Line] (a) to (e) to (b) to (a);
	\draw[Line] (a) to (af);
	\draw[Line] (e) to (ab);
	\draw[Line] (b) to (bf);
	\node at (-.2,-.2) {$x_a$};
	\node at (1.7,-.2) {$x_b$};
	\node at (.75,-.22) {$x_e$};
	\node at (.75,1.5) {$x_c$};
	\node at (.98,.415) {$x_d$};
	\node at (.17,.7) {$x_g$};
	\node at (1.33,.7) {$x_f$};
	\end{tikzpicture}
	\end{aligned}\]
	
	Any two points of $x_a,x_b,x_c,x_d$ are connected by a line, or equivalently, are near. Hence, given two of these points, we can form four pairs consisting of a non-trivial member of the first point and one non-trivial members of the second point, and precisely one of these pairs consists of mutually orthogonal elements. For instance, given the points $x_a$ and $x_b$, exactly one of $a\oplus b$, $a'\oplus b$, $a\oplus b'$ and $a'\oplus b'$ is defined.  We first exclude the possibility that both non-trivial elements in a point are orthogonal to some non-trivial members in the other points.
	For instance, we cannot have both $a\perp b$ and $a'\perp c$. By symmetry, this case implies all other cases. So assume $a\perp b$ and $a'\perp c$. We derive a contradiction. 
		
	
	These assumptions mean that $b < a' < c'$. Since these elements form a chain, they lie in a 16-element Boolean subalgebra of $A$ whose atoms are $b$, $a'\wedge b'$, $c'\wedge a$, $c$. Since $b\leq c'$, we have that $b\vee c$ is one of $f,f'$, hence lies in a block with $a,a'$. So one of the following must hold:
	\[ \mbox{(i) } a< b\vee c \qquad \mbox{(ii) } a'< b\vee c \qquad \mbox{(iii) } b\vee c < a\qquad \mbox{(iv) } b\vee c < a'\]
	We are in a Boolean algebra, so if (i) applies, $a=a\wedge(b\vee c) = (a\wedge b)\vee (a\wedge c) = a\wedge c$, giving $a\leq c$, a contradiction. If (ii) applies, $a'=a'\wedge(b\vee c)=(a'\wedge b)\vee(a'\wedge c) = a'\wedge b$, giving $a'\leq b$, a contradiction. If (iii) applies, $b\leq a$, a contradiction, and if (iv) applies, $c\leq a'$, a contradiction. 
	
So for each of $x_a,x_b,x_c,x_d$ there is one non-trivial member of that set that is orthogonal to a non-trivial member of each of the others. Without loss of generality, we may assume that these members are $a,b,c,d$. So $a, b, c, d$ are pairwise orthogonal. We note that pairwise orthogonal elements of an orthomodular poset are jointly orthogonal. Indeed, $a\perp c$ and $b\perp c$ imply $a,b\leq c'$, so $a\oplus b=a\vee b\leq c'$, hence $(a\oplus b)\oplus c$ is defined, and in a similar way, we find that $((a\oplus b)\oplus c)\oplus d)$ is defined. This means that $a,b,c,d$ lie in some Boolean subalgebra of $A$. Note, this is not a property that holds in orthoalgebras! Therefore, in each of the lines in the diagram where $x_a$ appears, $a$ is an atom of the line (recall, this line is a Boolean subalgebra of $A$), and similarly for $b,c,d$. It follows that the third atom of the line containing $x_a,x_b$ is $a'\wedge b'$. Then $a'\wedge b'$ is equal to one of $e,e'$, and we may assume $a'\wedge b'=e$. Since $d<a'\wedge b'$, the third atom of the line containing $x_c,x_d$ cannot be $e$, and therefore must be $e'=a\vee b$. Therefore $a\vee b \vee c\vee d = 1$. It follows that $a,b,c,d$ generate a 16-element Boolean subalgebra $w$ of $A$. Each element of the sets $x_e,x_f,x_g$ can be obtained from $a,b,c,d$, and therefore $x_e,x_f,x_g$ are also contained in $w$. By cardinality considerations, these are all the points beneath $w$, and hence also all the lines beneath $w$. 
\end{proof}

\begin{theorem} \label{OMPdetermined}
	If $A$ is an orthomodular poset, then the hypergraph $\mathcal{G}(A)$ is completely determined by its points and lines. 
\end{theorem}

\begin{proof}
	Proposition~\ref{loppo} shows that any configuration of points and lines in $\mathcal{G}(A)$ that is isomorphic to the set of points and lines of a plane, is the set of all points and lines of a plane of $\mathcal{G}(A)$. So the planes of $\mathcal{G}(A)$ are determined by the points and lines of $\mathcal{G}(A)$. 
\end{proof}

We have given several properties that hold in Boolean hypergraphs and orthohypergraphs. These will be used in the next section when we introduce morphisms. These properties are not intended to be sufficient to characterize these hypergraphs. The matter of having workable conditions to recognize the hypergraphs that are orthohypergraphs or Boolean hypergraphs seem quite interesting, and potentially of considerable use in constructing such orthohypergraphs, and hence their corresponding orthoalgebras. 

\begin{problem}
  Characterize those hypergraphs that are isomorphic to the hypergraphs of Boolean algebras, and those that are isomorphic to the hypergraphs of orthoalgebras. 
\end{problem}

We next consider several examples of hypergraphs of orthoalgebras and compare these with the Greechie diagrams of the structures. We refer the reader to \cite{N:handQL} for a thorough account, and comment only that Greechie diagrams show the atoms of a chain-finite orthoalgebra with the atoms of a block being connected by a line. We begin with the setting where Greechie diagrams are most familiar, representing orthoalgebras whose blocks have at exactly 8 elements. For such an orthoalgebra $A$ its Greechie diagram and hypergraph are the same. The 12-element orthomodular poset $MO_2\times 2$ consists of two 8-element Boolean algebras that intersect in a 4-element Boolean algebra. Its Greechie diagram is the same as its hypergraph, and is shown below. 
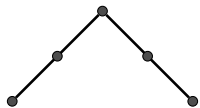
\[\begin{aligned}\begin{tikzpicture}[scale=.6]
\node[Dot] (a) at (0,0) {};
\node[Dot] (b) at (1,1) {};
\node[Dot] (c) at (2,2) {};
\node[Dot] (d) at (3,1) {};
\node[Dot] (e) at (4,0) {};
\draw[Line] (a) to (c) to (e);
\end{tikzpicture}\end{aligned}\]

For orthoalgebras having blocks with 16 or more elements, the situation is more complex. Consider the orthomodular poset $MO_2\times MO_2$. This structure has 36 elements, 8 atoms and 4 blocks of 16 elements each. Its Greechie diagram is given below. This diagram gives good insight into the behavior of the atoms and coatoms, but one must infer the behavior of the 18 elements of height two. In particular, the key feature of this structure, that it has two central elements of height 2 that are therefore in all four blocks, is hidden. With experience, one can tell how the ``missing'' elements of the structure behave, but as the structures become more complex, this becomes increasingly difficult. 
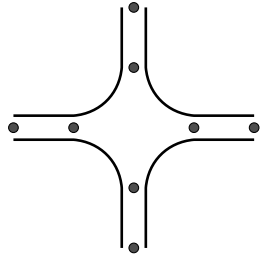
\[\begin{tikzpicture}[scale=.8]
\node[Dot] (a) at (-2,0) {};
\node[Dot] (b) at (-1,0) {};
\node[Dot] (c) at (0,1) {};
\node[Dot] (d) at (0,2) {};
\node[Dot] (e) at (1,0) {};
\node[Dot] (f) at (2,0) {}; 
\node[Dot] (g) at (0,-1) {}; 
\node[Dot] (h) at (0,-2) {}; 
\draw[Line] (-2,-.2) -- (-1,-.2) to [out=-5,in=95] (-.2,-1) -- (-.2,-2);
\draw[Line] (-2,.2) -- (-1,.2) to [out=5,in=-95] (-.2,1) -- (-.2,2);
\draw[Line] (2,-.2) -- (1,-.2) to [out=185,in=85] (.2,-1) -- (.2,-2);
\draw[Line] (2,.2) -- (1,.2) to [out=175,in=-85] (.2,1) -- (.2,2);
\end{tikzpicture}\]
 
In describing the hypergraph of a structure with larger blocks, it is useful to note that there are many ways to draw the seven points and six lines that comprise a plane. Below we give three possibilities, the original as a plane that we have used, and two others that we will employ. They are all equivalent. In these diagrams, we choose to depict the corner points and some portions of lines as larger than others as an aid to readability. The point is that pairs of ``parallel'' sides connecting corner points intersect at an edge point at ``infinity''. 
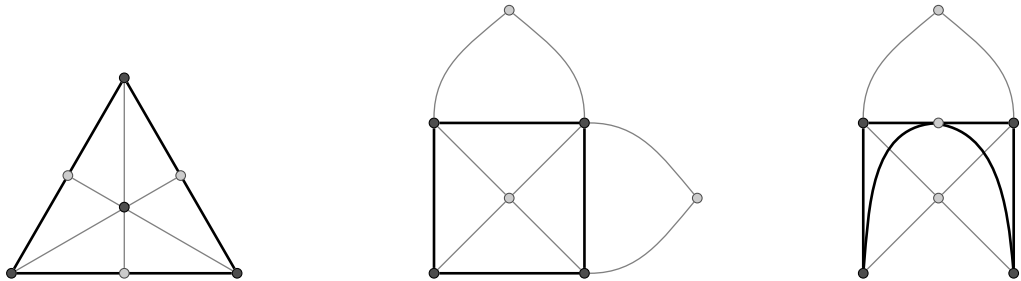
\[\begin{aligned}
\begin{tikzpicture}[scale=2]
\node[Dot] (a) at (0,0) {};
\node[Dot] (b) at (1.5,0) {};
\node[Dot] (e) at (.75,1.3) {};
\node[Dot] (f) at (.75,.44) {};
\node[dot] (af) at (1.125,.65) {};
\node[dot] (ab) at (.75,.0) {};
\node[dot] (bf) at (.375,.65) {};
\draw[Line] (a) to (e) to (b) to (a);
\draw[line] (a) to (af);
\draw[line] (e) to (ab);
\draw[line] (b) to (bf);
\end{tikzpicture}
\quad\quad\quad\quad\quad\quad
\begin{tikzpicture}[scale=2]
\node[Dot] (a) at (0,0) {};
\node[Dot] (b) at (1,0) {};
\node[Dot] (e) at (0,1) {};
\node[Dot] (f) at (1,1) {};
\node[dot] (af) at (.5,.5) {};
\node[dot] (ab) at (1.75,.5) {};
\node[dot] (bf) at (.5,1.75) {};
\draw[Line] (a) to (e) to (f) to (b) to (a);
\draw[line] (a) to (af) to (f);
\draw[line] (e) to (af) to (b);
\draw[line] (ab) to[out=-130,in=0] (b);
\draw[line] (ab) to[out=130,in=0] (f);
\draw[line] (bf) to[out=-140,in=90] (e);
\draw[line] (bf) to[out=-40,in=90] (f);
\end{tikzpicture}
\quad\quad\quad\quad\quad
\begin{tikzpicture}[scale=2]
\node[Dot] (a) at (0,0) {};
\node[Dot] (b) at (1,0) {};
\node[Dot] (e) at (0,1) {};
\node[Dot] (f) at (1,1) {};
\node[dot] (af) at (.5,.5) {};
\node[dot] (ab) at (.5,1) {};
\node[dot] (bf) at (.5,1.75) {};
\draw[Line] (a) to (e) to (f) to (b);
\draw[line] (a) to (af) to (f);
\draw[line] (e) to (af) to (b);
\draw[Line] (a) to [out=85,in=185] (ab) to [out=-5,in=95] (b);
\draw[line] (bf) to[out=-140,in=90] (e);
\draw[line] (bf) to[out=-40,in=90] (f);
\end{tikzpicture}
\end{aligned}\]

Making use of the rightmost form of depicting planes, the hypergraph of $MO_2\times MO_2$ is shown below. The Greechie diagram ``sits inside'' the hypergraph, and the elements missing from the Greechie diagram are included as smaller dots as well. The primary feature of this structure, the central element, is clearly visible. 
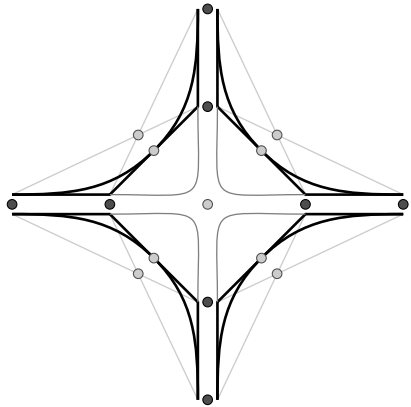
\[\begin{aligned}\begin{tikzpicture}[scale=1.3]
\node[Dot] (a) at (-2,0) {};
\node[Dot] (b) at (-1,0) {};
\node[Dot] (c) at (0,1) {};
\node[Dot] (d) at (0,2) {};
\node[Dot] (e) at (1,0) {};
\node[Dot] (f) at (2,0) {}; 
\node[Dot] (g) at (0,-1) {}; 
\node[Dot] (h) at (0,-2) {}; 
\node[dot] at (.55,.55) {};
\node[dot] at (-.55,.55) {};
\node[dot] at (.55,-.55) {};
\node[dot] at (-.55,-.55) {};
\node[dot] at (0,0) {};
\draw[thinline] (-2,-.1) to (-.1,-1);
\draw[thinline] (-1,-.1) to (-.1,-2);
\draw[thinline] (-2,.1) to (-.1,1);
\draw[thinline] (-1,.1) to (-.1,2);
\draw[thinline] (2,-.1) to (.1,-1);
\draw[thinline] (1,-.1) to (.1,-2);
\draw[thinline] (2,.1) to (.1,1);
\draw[thinline] (1,.1) to (.1,2);
\draw[Line] (-2,-.1) -- (-1,-.1) to (-.1,-1) -- (-.1,-2);
\draw[Line] (-2,.1) -- (-1,.1) to (-.1,1) -- (-.1,2);
\draw[Line] (2,-.1) -- (1,-.1) to (.1,-1) -- (.1,-2);
\draw[Line] (2,.1) -- (1,.1) to (.1,1) -- (.1,2);
\draw[Line] (-2,-.1) to [out=0,in=135] (-.55,-.55) to [out=-45,in=90] (-.1,-2);
\draw[Line] (-2,.1) to [out=0,in=-135] (-.55,.55) to [out=45,in=-90] (-.1,2);
\draw[Line] (2,-.1) to [out=180,in=45] (.55,-.55) to [out=-135,in=90] (.1,-2);
\draw[Line] (2,.1) to [out=180,in=-45] (.55,.55) to [out=135,in=-90] (.1,2);
\node[dot] at (-.71,-.71) {};
\node[dot] at (-.71,.71) {};
\node[dot] at (.71,-.71) {};
\node[dot] at (.71,.71) {};
\draw[line] (-1,.1) to [out=0,in=-90, looseness=2.1] (-.1,1);
\draw[line] (-1,-.1) to [out=0,in=90, looseness=2.1] (-.1,-1);
\draw[line] (1,.1) to [out=180,in=-90, looseness=2.1] (.1,1);
\draw[line] (1,-.1) to [out=180,in=90, looseness=2.1] (.1,-1);  
\end{tikzpicture}\end{aligned}\]

Here we may view the hypergraph as augmenting the Greechie diagram, allowing us to precisely depict and reason about the elements not depicted in the Greechie diagram. For those with experience with Greechie diagrams, this is not necessary for $MO_2\times MO_2$. However, when reasoning with more complicated situations such as the Fraser cube of Example~\ref{ex:frasercube}, this can be useful. For instance, the Fraser cube has the property that two of its blocks have intersection that is non Boolean, a fact that until very recently was incorrectly reported in the literature \cite{Navara}. 

The Fraser cube has 36 elements, with 8 atoms, 6 blocks of 16 elements each. Its Greechie diagram, as the name suggests, is a cube. Here the atoms are vertices of the cube, and the blocks are its faces. Using the middle of the three ways to depict a plane described above, we can draw the hypergraph of the Fraser cube as shown below. 
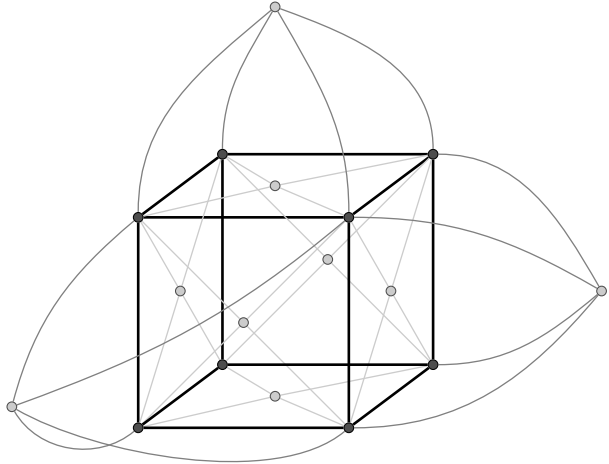
\[\begin{aligned}\begin{tikzpicture}[scale=2.8]
\node[Dot] (a) at (0,0) {};
\node[Dot] (b) at (1,0) {};
\node[Dot] (e) at (0,1) {};
\node[Dot] (f) at (1,1) {};
\node[Dot] (c) at (.4,.3) {};
\node[Dot] (d) at (1.4,.3) {};
\node[Dot] (g) at (.4,1.3) {};
\node[Dot] (h) at (1.4,1.3) {};
\node[dot] (af) at (.5,.5) {};
\node[dot] (ch) at (.9,.8) {};
\node[dot] (df) at (1.2,.65) {};
\node[dot] (ag) at (.2,.65) {};
\node[dot] (ad) at (.65,.15) {};
\node[dot] (eh) at (.65,1.15) {};
\node[dot] (ab) at (2.2,.65) {};
\node[dot] (bf) at (.65,2) {};
\draw[Line] (e) to (f) to (h) to (g) to (e);
\draw[Line] (a) to (e);
\draw[Line] (c) to (g);
\draw[Line] (d) to (h);
\draw[thinline] (a) to (af) to (f);
\draw[thinline] (e) to (af) to (b);
\draw[thinline] (b) to (df) to (h);
\draw[thinline] (f) to (df) to (d);
\draw[thinline] (a) to (ag) to (g);
\draw[thinline] (e) to (ag) to (c);
\draw[thinline] (c) to (ch) to (h);
\draw[thinline] (g) to (ch) to (d);
\draw[thinline] (a) to (ad) to (d);
\draw[thinline] (c) to (ad) to (b);
\draw[thinline] (e) to (eh) to (h);
\draw[thinline] (g) to (eh) to (f);
\draw[Line] (b) to (f);
\draw[Line] (a) to (b) to (d) to (c) to (a);

\node[dot] (ac) at (-.6,.1) {};

\draw[line] (ac) to[out=-60,in=-140] (a);
\draw[line] (ac) to[out=-30,in=-140,looseness=.7] (b);
\draw[line] (ac) to[out=75,in=-140] (e);
\draw[line] (ac) to[out=20,in=-140] (f);

\draw[line] (ab) to[out=-130,in=0] (b);
\draw[line] (ab) to[out=-140,in=0] (d);
\draw[line] (ab) to[out=150,in=0] (f);
\draw[line] (ab) to[out=120,in=0] (h);
\draw[line] (bf) to[out=-140,in=90] (e);
\draw[line] (bf) to[out=-60,in=90] (f);
\draw[line] (bf) to[out=-120,in=90] (g);
\draw[line] (bf) to[out=-20,in=90] (h);
\end{tikzpicture}\end{aligned}\]

The diagram has a point in the middle of each face, as well as 3 additional ``points at infinity'', one for each pencil of parallel edges. The blocks given by the front and back face intersect at the two points at infinity given by sideways and vertical edges, so do not intersect in a Boolean algebra. Here too, the Greechie diagram is sitting inside of the hypergraph, and the additional detail of the hypergraph allows us to depict and reason more clearly. 

\begin{remark}
Throughout our treatment of hypergraphs, we have used terminology reflective of that used in projective geometry. This is indicated in our use of the terms `point', `line', and `plane'. It has also been of benefit that there is a similarity between the planes used in our hypergraphs and the usual Fano plane, which is realized as the lattice of subspaces of a 3-dimensional vector space over the 2-element field $\mathbb{Z}_2$. 

Our results have close analogues in projective geometry. In projective geometry, one takes the lattice $L=\Sub(V)$ of subspaces of a vector space $V$. Lattices arising this way are characterized as certain complemented, algebraic, atomistic, modular lattices. This lattice $L$ is determined by its elements of height at most 2, and these elements are organized into the points and lines of a projective geometry with points being atoms and lines elements of height 2. The vector space $V$ can be reconstructed from $L$ by techniques essentially dating to Euclid. More recent treatments of this classical subject treat categorical aspects of this correspondence as well \cite{Faure}. All this is mirrored in our treatment of an orthoalgebra $A$ via its poset $\BSub(A)$ of Boolean subalgebras. In the case that $A$ is an orthomodular poset, we require only the elements of height 2 in this poset to reconstruct $A$, but for general orthoalgebras must go one level higher. 

Partial explanation of this similarity can be found by noting that each Boolean algebra $A$ is a vector space over $\mathbb{Z}_2$ under the addition $+$ of symmetric difference. Subalgebras of $A$ are the vector subspaces that contain 1 and are additionally closed under meet. Write $\langle 1 \rangle$ for the subspace of $A$ generated by the vector $1$, and consider the interval $I=[\langle 1\rangle,A]$ of the subspace lattice. The subalgebras of height $n$ in $\BSub(A)$ have height $n$ in $I$. Since $I$ is isomorphic to the subspace lattice of $A/\langle 1\rangle$, we may regard $\BSub(A)$ as sitting inside a projective geometry over $\mathbb{Z}_2$. This explains the similarity of the hypergraph of a 16-element Boolean algebra to a Fano plane --- the missing line is a subspace containing $1$ that is not a subalgebra. 
\end{remark}

\section{Morphisms}\label{sec:morphisms}

In this section we consider morphisms between orthoalgebras and orthohypergraphs. There are some basic obstacles to producing a full categorical equivalence, such as the fact that a 4-element Boolean algebra has two automorphisms, while its poset of Boolean subalgebras has only a single element. However, modulo such isolated pathologies, we show that morphisms between orthoalgebras can be captured by morphisms between their associated orthohypergraphs. 

Recall that an orthoalgebra morphism $f\colon A\to C$ preserves the orthocomplementation, and if $a\oplus b$ is defined, so is $f(a)\oplus f(b)$, in which case $f(a\oplus b)=f(a)\oplus f(b)$. Note that the image of $f$ need not be a sub-orthoalgebra of $C$, since elements $f(a)$ and $f(b)$ might be orthogonal in $C$ without $a$ and $b$ being orthogonal in $A$. For example, consider \textsc{mo}$_2$, the horizontal sum of two 4-element Boolean algebras. Call its four atoms $a,a',b,b'$. This embeds into the power set of $\{i,j,k\}$ by $f(a)=\{i\}$, $f(a')=\{j,k\}$, $f(b)=\{j\}$, and $f(b')=\{i,k\}$. On a positive note, we record the following fact. 

\begin{proposition}
  For orthoalgebras $A$ and $C$, a map $f\colon A\to C$ is an orthoalgebra morphism if and only if for each Boolean subalgebra $B$ of $A$, the restriction of $f$ to $B$ is a Boolean algebra homomorphism of $B$ into a Boolean subalgebra of $C$.
  \qed
\end{proposition}
\begin{proof}
Assume that $f$ is an orthoalgebra morphism, and let $a,b\in B$ such that $f(a)\oplus f(b)$ in $C$ is defined. By Proposition \ref{prop:booleansubalgebra}, there is a jointly orthogonal set $F\subseteq B$ such that $a=\bigoplus E_a$ and $b=\bigoplus E_b$ for some $E_a,E_b\subseteq F$. Hence we have $f(a)=\bigoplus f[E_a]$ and $f(b)=\bigoplus f[E_b]$. Assume there is some $e\in E_a\cap E_b$. Since $f(a)\oplus f(b)$ is defined, it follows that $f(e)\oplus f(e)$ is defined, so $f(e)$ must be $0$. Thus $f[E_a]\cap f[E_b]\subseteq\{0\}$. Let $c=\bigoplus (E_b\setminus E_a)$. Then $a\oplus c$ is defined and 
\[ f(a\oplus c)=\bigoplus f[E_a]\oplus\bigoplus f[E_b\setminus E_a]=\bigoplus f[E_a]\oplus\bigoplus f[E_b]=f(a)\oplus f(b),\]
so we conclude that $f(a)\oplus f(b)\in f[B]$. As a consequence, $f[B]$ is a subalgebra of $C$. To see that $f[B]$ is Boolean, let $S\subseteq f[B]$ be finite. Then there is a finite $T\subseteq B$ such that $f[T]=S$. Since $B$ is Boolean, there is some jointly orthogonal set $F\subseteq B$ such that $T\subseteq\{\bigoplus E \mid E\subseteq F\}$. Then $f[F]$ is a jointly orthogonal set in $f[B]$ such that $S\subseteq\{\bigoplus E\mid E\subseteq f[F]\},$ hence $f[B]$ is Boolean. Thus $f$ restricts to an orthoalgebra morphism between the Boolean algebras $B$ and $f[B]$. Let $a\leq b$ in $A$. By definition of the order in an orthoalgebra, we have $a\oplus c=b$ for some $c\in A$, hence $f$ preserves the order. Let $a,b\in B$. Then $f(a)\vee f(b)\leq f(a\vee b)$. Moreover, $a\vee b=e_1\vee e_2\vee e_3$ for $e_1=a\wedge b$, $e_2=a\wedge b'$ and $e_3=a'\wedge b$, which are mutually orthogonal, so
\[f(a\vee b)=f(e_1\oplus e_2\oplus e_3)=f(e_1)\oplus f(e_2)\oplus f(e_3)=f(e_1)\vee f(e_2)\vee f(e_3).\]
Since $f(e_1)\vee f(e_2)\leq f(a)$ and $f(e_1)\vee f(e_3)\leq f(b)$, we obtain $f(e_1)\vee f(e_2)\vee f(e_3)\leq f(a)\vee f(b)$, whence $f(a\vee b)=f(a)\vee f(b)$. Since $f$ preserves the orthocomplementation, it follows now from De Morgan's Law that $f$ preserves meets in $B$, hence its restriction to $B$ is a Boolean algebra homomorphism.

For the converse, let $a,b\in A$ such that $c=a\oplus b$ is defined. Then $B=\{0,a,a',b,b',c,c',1\}$ is a Boolean subalgebra $B$ of $A$, and $f$ restricts to a Boolean algebra homomorphism of $B$ into some Boolean subalgebra of $C$. As a consequence, we have $f(a')=f(a)'$. Since $a\perp b$, we have $b\leq a'$, so $f(b)\leq f(a)'$, i.e., $f(a)\perp f(b)$. Hence \[f(a\oplus b)=f(a\vee b)=f(a)\vee f(b)=f(a)\oplus f(b),\] so $f$ is an orthoalgebra morphism.   
\end{proof}

Suppose that $A$ and $C$ are orthoalgebras. Write $P$ and $Q$ for the sets of points of the hypergraphs $\mathcal{G}(A)$ and $\mathcal{G}(C)$. We define a morphism $\alpha\colon\mathcal{G}(A)\to\mathcal{G}(C)$ to be a partial function $\alpha\colon P\to Q$ satisfying certain properties outlined below. In dealing with partial function, we write $\alpha(p)=\bot$ to indicate that the partial function is not defined at the point $p$, and indicate this diagrammatically by crossing out the vertex of the hypergraph indicating $p$. 

\begin{definition}\label{morphism}
  For $\mathcal{G},\mathcal{H}$ orthohypergraphs with point sets $P$ and $Q$, a morphism $\alpha\colon\mathcal{G}\to\mathcal{H}$ is a partial function $\alpha\colon P\to Q$ such that:
  \begin{itemize}
    \item[(A1)] The partial function $\alpha$ acts on a line $l$ of $\mathcal{G}$ in one of the following ways. 
    
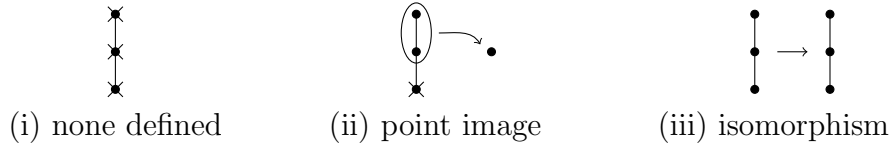
\[
    \begin{tikzpicture}[scale = 1]
    \draw[fill] (0,0) circle(1.5pt); \draw[fill] (0,.5) circle(1.5pt); \draw[fill] (0,1) circle(1.5pt); \draw[thin] (0,0)--(0,1);
    \draw[thin] (-.1,-.1)--(.1,.1); \draw[thin] (.1,-.1)--(-.1,.1); 
    \draw[thin] (-.1,.4)--(.1,.6); \draw[thin] (.1,.4)--(-.1,.6);
    \draw[thin] (-.1,.9)--(.1,1.1); \draw[thin] (.1,.9)--(-.1,1.1);
    \node at (0,-.5) {(i) none defined};

    \draw[fill] (4,0) circle(1.5pt); \draw[fill] (4,.5) circle(1.5pt); \draw[fill] (4,1) circle(1.5pt); \draw[thin] (4,0)--(4,1);
    \draw[thin] (3.9,-.1)--(4.1,.1); \draw[thin] (4.1,-.1)--(3.9,.1); 
    \draw[thin] (4,.75) ellipse(.2 and .4);
    \draw[thin, ->] (4.3,.75) to [out=0,in=120] (4.85,.6); \draw[fill] (5,.5) circle(1.5pt);
    \node at (4.25,-.5) {(ii) point image};

    \draw[fill] (8.5,0) circle(1.5pt); \draw[fill] (8.5,.5) circle(1.5pt); \draw[fill] (8.5,1) circle(1.5pt); \draw[thin] (8.5,0)--(8.5,1);
    \draw[fill] (9.5,0) circle(1.5pt); \draw[fill] (9.5,.5) circle(1.5pt); \draw[fill] (9.5,1) circle(1.5pt); \draw[thin] (9.5,0)--(9.5,1);
    \draw[thin,->] (8.8,.5)--(9.2,.5);
    \node at (8.75,-.5) {(iii) isomorphism};
    \end{tikzpicture} 
    \]
    \item[(A2)] The partial function $\alpha$ acts on a plane $t$ of $\mathcal{G}$ in one of the following ways. 
    
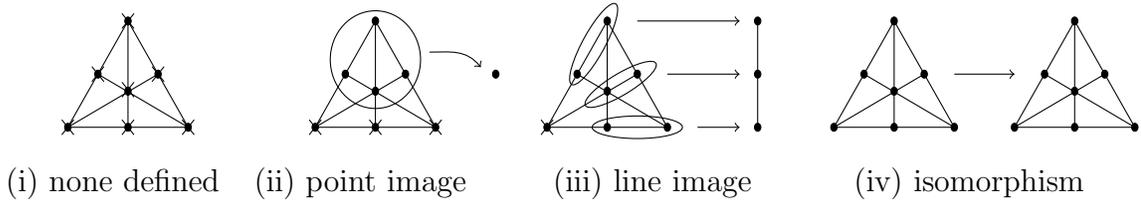
\[
    \begin{tikzpicture}[xscale=.8,yscale=1]
    \draw[fill] (0,0) circle(1.5pt); \draw[fill] (1,0) circle(1.5pt); \draw[fill] (2,0) circle(1.5pt); \draw[fill] (1,1.414) circle(1.5pt); \draw[fill] (1,.48) circle(1.5pt); \draw[fill] (.5,.707) circle(1.5pt); \draw[fill] (1.5,.707) circle(1.5pt); 
    \draw[thin] (0,0)--(2,0)--(1,1.414)--(0,0)--(1.5,.707); \draw[thin] (1,1.414)--(1,0); \draw[thin] (.5,.707)--(2,0);
    \draw[thin] (-.1,-.1)--(.1,.1); \draw[thin] (-.1,.1)--(.1,-.1); \draw[thin] (.9,-.1)--(1.1,.1); \draw[thin] (.9,.1)--(1.1,-.1); \draw[thin] (1.9,-.1)--(2.1,.1); \draw[thin] (1.9,.1)--(2.1,-.1); \draw[thin] (.9,1.314)--(1.1,1.514); \draw[thin] (.9,1.514)--(1.1,1.314);
    \draw[thin] (.9,.4)--(1.1,.6); \draw[thin] (.9,.6)--(1.1,.4); \draw[thin] (.4,.607)--(.6,.807); \draw[thin] (.4,.807)--(.6,.607); \draw[thin] (1.4,.607)--(1.6,.807); \draw[thin] (1.4,.807)--(1.6,.607);
    \node at (.75,-.75) {(i) none defined};
    \end{tikzpicture} 
    \hspace{1ex}
    \begin{tikzpicture}[xscale=.8,yscale=1]
    \draw[fill] (0,0) circle(1.5pt); \draw[fill] (1,0) circle(1.5pt); \draw[fill] (2,0) circle(1.5pt); \draw[fill] (1,1.414) circle(1.5pt); \draw[fill] (1,.48) circle(1.5pt); \draw[fill] (.5,.707) circle(1.5pt); \draw[fill] (1.5,.707) circle(1.5pt); 
    \draw[thin] (0,0)--(2,0)--(1,1.414)--(0,0)--(1.5,.707); \draw[thin] (1,1.414)--(1,0); \draw[thin] (.5,.707)--(2,0);
    \draw[fill] (3,.707) circle(1.5pt);
    \draw[thin] (-.1,-.1)--(.1,.1); \draw[thin] (-.1,.1)--(.1,-.1); \draw[thin] (.9,-.1)--(1.1,.1); \draw[thin] (.9,.1)--(1.1,-.1); \draw[thin] (1.9,-.1)--(2.1,.1); \draw[thin] (1.9,.1)--(2.1,-.1); 
    \draw (1,.9) ellipse (.75 and .65); \draw[thin,->] (1.9,1.0) to [out=0,in=135] (2.75,.8);
    \node at (.75,-.75) {(ii) point image};
    \end{tikzpicture} 
    \hspace{3ex}
    \begin{tikzpicture}[xscale=.8,yscale=1]
    \draw[fill] (0,0) circle(1.5pt); \draw[fill] (1,0) circle(1.5pt); \draw[fill] (2,0) circle(1.5pt); \draw[fill] (1,1.414) circle(1.5pt); \draw[fill] (1,.48) circle(1.5pt); \draw[fill] (.5,.707) circle(1.5pt); \draw[fill] (1.5,.707) circle(1.5pt); 
    \draw[thin] (0,0)--(2,0)--(1,1.414)--(0,0)--(1.5,.707); \draw[thin] (1,1.414)--(1,0); \draw[thin] (.5,.707)--(2,0);
    \draw[fill] (3.5,0) circle(1.5pt); \draw[fill] (3.5,.707) circle(1.5pt); \draw[fill] (3.5,1.414) circle(1.5pt); \draw[thin] (3.5,0)--(3.5,1.414);
    \draw[thin] (-.1,-.1)--(.1,.1); \draw[thin] (-.1,.1)--(.1,-.1);
    \draw[thin] (1.5,0) ellipse (.75 and .15);
    \draw[thin,rotate=25] (1.35,0) ellipse (.65 and .15);
    \draw[thin,rotate=55] (1.35,0) ellipse (.65 and .175);
    \draw[thin,->] (2.5,0)--(3.2,0); \draw[thin,->] (2,.707)--(3.2,.707); \draw[thin,->] (1.5,1.414)--(3.2,1.414);
    \node at (1.75,-.75) {(iii) line image};
    \end{tikzpicture} 
    \hspace{5ex}
    \begin{tikzpicture}[xscale=.8,yscale=1]
    \draw[fill] (0,0) circle(1.5pt); \draw[fill] (1,0) circle(1.5pt); \draw[fill] (2,0) circle(1.5pt); \draw[fill] (1,1.414) circle(1.5pt); \draw[fill] (1,.48) circle(1.5pt); \draw[fill] (.5,.707) circle(1.5pt); \draw[fill] (1.5,.707) circle(1.5pt); 
    \draw[thin] (0,0)--(2,0)--(1,1.414)--(0,0)--(1.5,.707); \draw[thin] (1,1.414)--(1,0); \draw[thin] (.5,.707)--(2,0);
    \node at (2.25,-.75) {(iv) isomorphism};
    \draw[fill] (3,0) circle(1.5pt); \draw[fill] (4,0) circle(1.5pt); \draw[fill] (5,0) circle(1.5pt); \draw[fill] (4,1.414) circle(1.5pt); \draw[fill] (4,.48) circle(1.5pt); \draw[fill] (3.5,.707) circle(1.5pt); \draw[fill] (4.5,.707) circle(1.5pt); 
    \draw[thin] (3,0)--(5,0)--(4,1.414)--(3,0)--(4.5,.707); \draw[thin] (4,1.414)--(4,0); \draw[thin] (3.5,.707)--(5,0);
    \draw[thin,->] (2,.707)--(3,.707);
    \end{tikzpicture} 
    \]
    \item[(A3)] If $l,m$ are lines of $\mathcal{G}$ that intersect in the point $p$, and $\alpha(l),\alpha(m)$ are distinct lines of a plane $t'$ of $\mathcal{H}$ whose intersection is an edge point $\alpha(p)$ of $t'$, then $l,m$ lie in a plane $t$ of $\mathcal{G}$ that is mapped isomorphically to $t'$. 
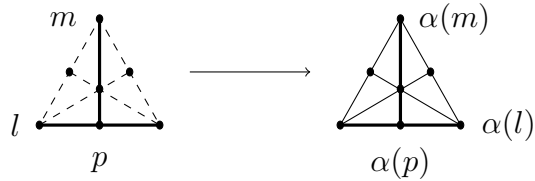
    \[
    \begin{tikzpicture}[xscale=.8,yscale=1]
    \draw[fill] (0,0) circle(1.5pt); \draw[fill] (1,0) circle(1.5pt); \draw[fill] (2,0) circle(1.5pt); \draw[fill] (1,1.414) circle(1.5pt); \draw[fill] (1,.48) circle(1.5pt); \draw[fill] (.5,.707) circle(1.5pt); \draw[fill] (1.5,.707) circle(1.5pt); 
    \draw[thin, dashed] (0,0)--(2,0)--(1,1.414)--(0,0)--(1.5,.707); \draw[very thick] (1,1.414)--(1,0); \draw[thin, dashed] (.5,.707)--(2,0); \draw[very thick] (0,0)--(2,0);
    \node at (-.4,0) {$l$}; \node at (.4,1.414) {$m$};
    \draw[thin,->] (2.5,.707)--(4.5,.707);
    \node at (1,-.5) {$p$};
    \node at (6,-.5) {$\alpha(p)$};
    \draw[fill] (5,0) circle(1.5pt); \draw[fill] (6,0) circle(1.5pt); \draw[fill] (7,0) circle(1.5pt); \draw[fill] (6,1.414) circle(1.5pt); \draw[fill] (6,.48) circle(1.5pt); \draw[fill] (5.5,.707) circle(1.5pt); \draw[fill] (6.5,.707) circle(1.5pt); 
    \draw[thin] (5,0)--(7,0)--(6,1.414)--(5,0)--(6.5,.707); \draw[very thick] (6,1.414)--(6,0); \draw[thin] (5.5,.707)--(7,0); \draw[very thick] (5,0)--(7,0);
    \node at (7.8,0) {$\alpha(l)$}; \node at (6.9,1.414) {$\alpha(m)$};
    \end{tikzpicture} 
    \]
  \end{itemize}
\end{definition}

\begin{proposition}
  Orthohypergraphs and the hypergraph morphisms form a category under the usual composition of partial functions. 
\end{proposition}
\begin{proof}
  The identity map on an orthohypergraph is a hypergraph morphism. Suppose $\mathcal{G},\mathcal{H},\mathcal{J}$ are orthohypergraphs with point sets $P,Q,R$, and $\alpha\colon\mathcal{G}\to\mathcal{H}$ and $\beta\colon\mathcal{H}\to\mathcal{J}$ are hypergraph morphisms. So $\alpha\colon P\to Q$ and $\beta\colon Q\to R$ are partial functions. The composite $\beta\circ\alpha$ is the usual relational product $\beta\circ\alpha = \{(p,r) \mid \exists q\in Q \colon (p,q)\in\alpha, (q,r)\in\beta\}$. So $\gamma=\beta\circ\alpha$ is a partial function from $P$ to $R$. We must verify (A1)--(A3). 

  For (A1), let $l=\{x,y,z\}$ be a line of $\mathcal{G}$. If case (i) of (A1) applies to $\alpha(l)$, then none of $\alpha(x),\alpha(y),\alpha(z)$ are defined, so none of $\gamma(x),\gamma(y),\gamma(z)$ is defined, so (i) of (A1) applies to $\gamma$. In case (ii), $\alpha(l)$ is a point $q$ of $Q$. If $\beta(q)=\bot$, then case (A1.i) applies to $\gamma(l)$, and if $\beta(q)=r$, then (A1.ii) applies to $\gamma(l)$. In case (iii), $\alpha(l)=m$ for some line $m$ of $R$, and whichever of case (i)--(iii) of (A1) applies to $\beta(m)$ also applies to $\gamma(l)$. Thus (A1) holds for $\gamma$.

  For (A2), let $t$ be a plane of $\mathcal{G}$. If case (i) of (A2) applies to $\alpha(t)$, then (A2.i) applies to $\gamma(t)$. In case (ii), $\alpha(t)$ is a point $q$ of $Q$. Then either (A2.i) or (A2.ii) applies to $\gamma(t)$ according to whether $\beta(q)$ is undefined or a point of $R$. In case (iii), $\alpha(t)$ is a line $m$ of $Q$. Then (A2.i), (A2.ii), (A2.iii) applies to $\gamma(t)$ according to whether (A1.i), (A1.ii), (A1.iii) applies to $\beta(m)$. If case (A2.iv) applies to $\alpha(t)$, then $\gamma(t)$ is a plane $t'$ of $Q$. Then case (i)--(iv) of (A2) applies to $\gamma(t)$ according to which of case (i)--(iv) of (A2) applies to $\beta(t')$. Thus (A2) holds for $\gamma$. 

  For (A3), suppose $l,m$ are lines of $P$ that intersect in a point $p$. Suppose that $\gamma(l)=l'',\gamma(m)=m''$ are distinct lines of a plane $t''$ of $\mathcal{J}$ that intersect in an edge point $r$ of $t''$. Since $\gamma(l),\gamma(m)$ are distinct lines of $R$, the lines $\alpha(l)=l'$ and $\alpha(m)=m'$ of $Q$ must be distinct. Since $p$ lies on $l,m$, then $\alpha(p)=q$ is a point on both $l',m'$, and must therefore be their unique intersection point. Then the lines $l',m'$ intersect in a point $q$ and $\beta(l')=l'',\beta(m')=m''$ are distinct lines of the plane $t''$ of $\mathcal{J}$ that intersect in the edge point $r$ of $t''$. Since (A3) applies to $\beta$, the lines $l',m'$ lie in a plane $t'$ of $\mathcal{H}$ that is mapped isomorphically by $\beta$ to $t''$. In particular, $l',m'$ are distinct lines of $t'$, and their intersection point $q$ is an edge point of $t'$. Applying (A3) to $\alpha$ gives a plane $t$ of $\mathcal{G}$ that contains $l,m$ and is mapped by $\alpha$ isomorphically to $t'$. Thus $\gamma$ maps $t$ isomorphically to $t''$, as required. Thus (A3) holds for $\gamma$.
\end{proof}

\begin{definition}
Let $\cat{OA}$ be the category of orthoalgebras and orthoalgebra morphisms, and $\cat{OH}$ be the category of orthohypergraphs and  hypergraph morphisms. 
\end{definition}

Next we extend $A \mapsto \mathcal{G}(A)$ to a functor $\cat{OA} \to \cat{OH}$. Let $A$ and $C$ be orthoalgebras with $P$ and $Q$ the point sets of their hypergraphs. So points of $P$ are 4-element subalgebras $x_a$ of $A$ and points of $Q$ are 4-element subalgebras $x_c$ of~$C$. 

\begin{definition}\label{mii}
  For an orthoalgebra morphism $f\colon A\to C$, define a partial map $\mathcal{G}(f)\colon P\to Q$ by 
  \[\mathcal{G}(f)(x_a)=\begin{cases} x_{f(a)} & \mbox{if  $f(a)\neq 0,1$,}\\
 \bot &\mbox{otherwise}. \end{cases}\]
\end{definition}
Informally, if we regard $\bot$ as an augmented least element of the hypergraph, we have $x_0=x_1=\bot$, hence $\mathcal G(f)(x_a)=x_{f(a)}$ for each $a\in A$. 
\begin{proposition}\label{pkl}
  If $f$ is an orthoalgebra morphism, then $\mathcal{G}(f)$ is a well-defined hypergraph morphism, giving a functor $\mathcal{G}\colon \cat{OA}\to\cat{OH}$.
\end{proposition}
\begin{proof}
  Let $f \colon A\to C$ and $\alpha=\mathcal{G}(f)$. By definition, $\alpha$ is a partial function from the set of points $P$ of the hypergraph $\mathcal{G}(A)$ to the set of points $Q$ of the hypergraph $\mathcal{G}(C)$. To show it is a hypergraph morphism, we must verify (A1)--(A3).

  Suppose $l$ is a line of $P$. Then $l$ is an $8$-element Boolean subalgebra of $A$. The points $p$ on the line $l$ are the $x_a$ where $a$ is an atom of $l$. The image $s=f(l)$ is a Boolean subalgebra of $C$. If $s$ has $1$ or $2$ elements, then $f$ maps each atom of $l$ to $0$ or~$1$, hence each point $p$ of the line $l$ to $\bot$, and (A1.i) applies. If $s$ has $4$ elements, so $s=x_c$ is a point of $Q$, then one atom of $l$ is mapped to $0$ and the other two to $c,c'$. So one point on the line $l$ is mapped by $\alpha$ to $\bot$, and the other two to $s$. So (A1.ii) applies to $\alpha(l)$. If $s$ has $8$ elements, then $f$ is an isomorphism from $l$ to $s$ and (A1.iii) applies. 

  Let $s=f(t)$ be the image of a $16$-element Boolean subalgebra $t$ of~$A$. Then $s$ is a Boolean subalgebra of $C$. If $s$ has $1$ or $2$ elements, then every element of $t$ is mapped to $0,1$, so (A2.i) applies. If $s$ is a $4$-element Boolean algebra with atoms $a,a'$, then $s=x_a$ is a point of $Q$. In this case, exactly two atoms of $t$ are mapped by $f$ to $0$, and the other two atoms are mapped to $a,a'$. Then two corners of $t$ are mapped to $\bot$, and the other two corners to the point~$s$. That the remainder of the situation is as described in (A2.ii) follows from behavior of $\alpha$ on lines as given in (A1) that is already established.  Suppose $s$ has $8$ elements. Then three atoms of $t$ are mapped to the three atoms of $s$, and the fourth atom of $t$ is mapped to~$0$. Thus three corners of $t$ are mapped to the three points on the line $s$ of $A$ and the fourth corner is mapped to~$\bot$. That the remainder of the situation is as in (A2.iii) follows from the behavior of $\alpha$ on lines already established. Finally, if $s$ has $16$ elements, then $f$ is an isomorphism from $t$ to~$s$, and the situation is as in (A2.iv).

  For (A3), let $l,m$ be distinct lines of $P$ that intersect in a point and satisfy the hypotheses of (A3). Suppose the points on $l$ are $x_a, x_b, x_c$, and those of $m$ are $x_c, x_d, x_e$ for some $a,b,c,d,e\in A$ with none equal to $0,1$. One of $a,a'$, one of $b,b'$, and one of $c,c'$ is an atom of $l$; and one of $c,c'$, one of $d,d'$, and one of $e,e'$ is an atom of $m$. We may assume without loss of generality that $a,b$ are atoms of $l$ and that $d,e$ are atoms of $m$. There are now two distinct possibilities: that the same member of $c,c'$ that is an atom of $l$ is an atom of $m$, and that one of $c,c'$ is an atom of $l$ and the other is an atom of $m$. We may assume without loss of generality that $c$ is an atom of $l$. 

  If $c'$ is an atom of $m$, then $a\oplus b = c'$, so $((a\oplus b)\oplus d)\oplus e$ exits and is equal to $1$. Thus $a,b,d,e$ are a partition of unity in $A$, so generate a $16$-element Boolean subalgebra $t$ of $A$. Thus $t$ is a plane of $P$, and since neither $c,c'$ is an atom of $t$, we have that $l,m$ intersect in the edge point $x_c$ of this plane. We now consider the case where $c$ is an atom of both $l,m$. Since $\alpha(x_c)$ is an edge point of the plane $t'$ that contains $\alpha(l)$ and $\alpha(m)$, we have that $f(c)$ is an atom of one of $\alpha(l),\alpha(m)$, and a coatom of the other. But our assumptions of $\alpha(l),\alpha(m)$ give that $f$ maps $l$ isomorphically to $\alpha(l)$, and $f$ maps $m$ isomorphically to $\alpha(m)$. So $f$ cannot map an atom of one of these Boolean algebras to a coatom of its image. 

  This shows that $\alpha=\mathcal{G}(f)$ is a hypergraph morphism. If $g \colon C\to E$ is an orthoalgebra morphism, it is clear that $\mathcal{G}(g\circ f)=\mathcal{G}(g)\circ\mathcal{G}(f)$, and that $\mathcal{G}$ takes the identity morphism of $A$ to the identity of $\mathcal{G}(A)$. Thus $\mathcal{G}$ is a functor. 
\end{proof}

There are several fundamental obstacles preventing an equivalence between the categories $\cat{OA}$ and $\cat{OH}$. On the level of objects, the one-element and 2-element Boolean algebras both have a 1-element poset of Boolean subalgebras, and hence empty hypergraphs. Furthermore, the 4-element Boolean algebra has two automorphisms, while its hypergraph has one point and therefore only one automorphism. The latter difficulty extends to any orthoalgebra having a 4-element block. Moreover, if we were to consider a morphism from a countable free Boolean algebra to itself whose image was a 4-element Boolean algebra, a similar problem would arise, and this would be the case also if we took an orthoalgebra that was a horizonal sum of two such free Boolean algebras. However, modulo such small blocks, we next show that morphisms in $\cat{OA}$ can be treated via morphisms in $\cat{OH}$. 

\begin{definition}
  An orthoalgebra morphism $f \colon A\to C$ is \emph{proper} if 
  \begin{itemize}
  	\item each $a$ 
  	 in $A$ is in a block whose image has more than 4 elements
  \item each orthogonal $a,b\in A$ are contained in a block of $A$ whose image does not equal $\{0,f(a),f(a)',f(b),f(b)',1\}$. 
  \end{itemize}
\end{definition}

 Clearly, if the image of each block of $A$ under $f\colon A\to C$ has more than 4 elements, then $f$ is proper. The reason we prefer the more complex condition of properness to simply saying that the image of each block has more than 4 elements is only in part due to greater generality. The condition of properness has a simpler, and more easily applicable, translation to the hypergraph setting. 

\begin{definition}\label{A4}
  A hypergraph morphism $\alpha\colon P\to Q$ is {\em proper} when:
  \begin{itemize}
    \item[(A4)] Each point $p$ in $P$ is in a line or plane whose image contains a line;
    \item[(A5)] For distinct points $p,q$ of $P$ that lie on a line of $P$ there is a point $s\in P$ so that $p,q,s$ lie in a line or plane of $P$ and $\alpha(s)$ is defined and not equal to $\alpha(p),\alpha(q)$. 
  \end{itemize}
\end{definition}

\begin{proposition}\label{trek}\label{normal}
  Let $A,C$ be orthoalgebras and let $P,Q$ be the point sets of their hypergraphs. An orthoalgebra morphism $f\colon A\to C$ is proper if and only if its hypergraph morphism $\mathcal{G}(f)\colon P\to Q$ is proper.
\end{proposition}
\begin{proof}
  Write $\alpha=\mathcal{G}(f)$.

  ``$\Rightarrow$'' Let $p$ be a point of $P$. Then $p=x_a$ for some $a\neq 0,1$ in $A$. Since $f$ is proper there is a block $B$ of $A$ that contains $a$ and whose image has more than four elements. If $\alpha(p)$ is a point $x_{f(a)}$ of $Q$, then there is some $b\in B$ with $x_b$ a point mapped by $\alpha$ to a point $x_{f(b)}$ of $Q$ different from~$x_{f(a)}$. Since $a,b$ generate a Boolean subalgebra of $B$ with more than $4$ elements, $p,q$ lie on a line or plane of~$P$. The image of this line or plane under $\alpha$ contains distinct points $x_{f(a)}, x_{f(b)}$, hence by (A1)--(A2) contains a line. 

  Suppose $\alpha(p)=\bot$. So we may assume $f(a)=0$. Since the image of $B$ has more than 4 elements, there are $b,c\in B$ with $f(b)=i$ and $f(c)=j$ where $i,j$ are distinct atoms of an $8$-element Boolean subalgebra of the image of $B$. Set $b_1=a'\wedge b\wedge c'$ and $c_1=a'\wedge b'\wedge c$. Then $a,b_1,c_1$ belong to $B$ and $f(a)=0$, $f(b_1)=i$, $f(c_1)=j$. Further, $a,b_1,c_1$ are pairwise orthogonal. So $a,b_1,c_1$ generate a subalgebra of $B$ whose atoms are among $a,b_1,c_1,a'\wedge b_1'\wedge c_1'$. This subalgebra has either $8$ or $16$ elements, so is either a line or plane of $P$ that contains~$p$. Its image under $\alpha$ contains distinct points of $Q$, hence contains a line of~$Q$. 

    To show that $\alpha$ satisfies (A5), suppose $p,q$ are distinct elements of $P$ that lie on a line of $P$. Then there are $a,b\in A$ with $p=x_a$ and $q=x_b$. Since $p,q$ lie on a line we have that one of $a,a'$ is orthogonal to one of $b,b'$, and we assume $a$ is orthogonal to~$b$. By assumption, there is a block $B$ of $A$ that contains $a,b$ and some $c\in B$ with $f(c)\notin S:=\{0,f(a),f(a)',f(b),f(b)',1\}$. We consider several cases, and in each produce an element $c_1\in B$ with $f(c_1)\notin S$ and $a,b,c_1$ generating an at most $16$-element subalgebra of~$B$. We use the fact that three pairwise orthogonal elements of a Boolean algebra, and that a $3$-element chain of $B$, generate an at most $16$-element subalgebra. 

  If $f(c)\wedge f(a)\neq 0,f(a)$, set $c_1=a\wedge c$. Then $c_1<a\leq b'$. Also $f(c_1)\neq 0,f(a)$, and since $0<f(c_1)<f(a)\leq f(b')$, we cannot have $f(c_1)=f(a)',f(b),f(b)',1$. So $f(c_1)\notin S$. If $f(c)\wedge f(b)\neq 0,f(b)$ the situation is symmetric. If $f(c)\wedge f(a)=0$ and $f(c)\wedge f(b)=0$, then set $c_1=a'\wedge b'\wedge c$. Then $f(c_1)=f(c)$ and $a,b,c_1$ are pairwise orthogonal. If $f(c)\wedge f(a)=0$ and $f(c)\wedge f(b)=f(b)$, set $c_1=(a'\wedge c)\vee b$. Then $f(c_1)=f(c)$ and $b\leq c_1\leq a'$. If $f(c)\wedge f(a)=f(a)$ and $f(c)\wedge f(b)=0$ it is symmetric. Finally, if $f(c)\wedge f(a)=f(a)$ and $f(c)\wedge f(b)=f(b)$, set $c_1=a'\wedge b'\wedge c'$. Then $f(c_1)=f(c)'$ and $a,b,c_1$ are pairwise orthogonal. 

  ``$\Leftarrow$'' Suppose $a\neq 0,1$. Then it follows from (A4) that $x_a$ is in a line or plane of $P$ whose image under $\alpha$ contains a line. Thus $a$ is in a Boolean subalgebra of $A$ whose image under $f$ contains an $8$-element subalgebra of $C$. Extend this Boolean subalgebra containing $a$ to a block, and this provides the first condition. For the second, suppose $a,b\in A$ are orthogonal. If one or both of $a,b$ is 0 or $a=b'$, then the second condition for the properness of $f$ follows from the first. If both $a,b\neq 0$ and $a\neq b'$, then $x_a$ and $x_b$ are distinct points of $P$, and the orthogonality of $a,b$ gives that they lie on a line of $P$. By (A5) there is a point $x_c$ of $P$ that lies in a line or plane of $P$ with $\alpha(x_c)$ defined and not equal to $\alpha(x_a),\alpha(x_b)$. Then $a,b,c$ lie in a Boolean subalgebra of $A$ with at most $16$-elements and $f(c)$ is not equal to any of $0,f(a),f(a)',f(b),f(b'),1$. Extend this Boolean subalgebra to a block. 
\end{proof}

\begin{lemma}\label{injectionssurjections}
  Let $A,C$ be proper orthoalgebras.
  If an orthoalgebra morphism $f \colon A \to C$ is injective, then it is proper.
  Moreover, $f$ is injective if and only if $\mathcal{G}(f)$ is injective.
\end{lemma}
\begin{proof}
  If $a \neq 0,1$ in $A$, it is in a block with more than $4$ elements because $A$ has no small blocks.
  Now let $a,b \in A$ be orthogonal.
  Suppose $f$ is injective.
  Then $a,b$ are contained in a block $B$ of $A$ with at least $8$ elements. Hence $f(a),f(b)$ lie in a block that contains $f(B)$ and therefore has at least $8$ elements, so cannot equal $\{0,f(a),f(a)',f(b),f(b)',1\}$. Hence $f$ is proper. Let $x_a,x_b$ be points of the hypergraph corresponding to $A$ and assume that $\mathcal G(f)(x_a)=\mathcal G(f)(x_b)$. Then $\{0,f(a),f(a)',1\}=\{0,f(b),f(b)',1\}$ (also if $f(a)=0,1$), hence $f(a)=f(b)$ or $f(a)=f(b')$. Injectivity of $f$ gives $a=b$ or $a=b'$. In both cases we have $x_a=x_b$. Conversely, assume that $\mathcal{G}(f)$ is injective, and let $a,b\in A$ such that $f(a)=f(b)$. Then \[ \mathcal G(f)(x_a)=x_{f(a)}=x_{f(b)}=\mathcal G(f)(x_b),\]
  	hence $x_a=x_b$ by injectivity of $\mathcal G(f)$, whence $a=b$ or $a=b'$. However, the latter would imply $f(b)=f(a)=f(b')=f(b)'$, which is impossible if $A$ is proper, so we must have $a=b$. We conclude that $f$ is injective.  
%
%
\end{proof}

\begin{remark}
  Neither orthoalgebra morphisms that are proper, nor proper hypergraph morphisms, are closed under composition. 
  For example, consider the $16$-element Boolean algebra $A$ with atoms $a,b,c,d$, define $f \colon A \to A$ by $a\mapsto 0$, $b \mapsto b$, $c \mapsto c$, $d \mapsto d$, and define $g \colon A \to A$ by $a \mapsto a$, $b \mapsto 0$, $c \mapsto c$, $d \mapsto d$. Then $f$ and $g$ are both orthoalgebra morphisms that are proper, but $g \circ f$ has a $4$-element image so is not proper. Since $A$ is a proper orthoalgebra, it also follows that an orthoalgebra morphism between proper orthoalgebras need not be a proper morphism.

  Thus a direct categorical approach using proper morphisms is not possible. 
  By the previous lemma, we can restrict to the category $\cat{OA_i}$ of proper orthoalgebras and injective orthoalgebra homomorphisms. 
  For now, we will stay general, and show that the functor $\mathcal{G} \colon \cat{OA}\to\cat{OH}$ is full and faithful with respect to proper morphisms. 
\end{remark}

\begin{proposition}\label{cvb}
  Orthoalgebra morphisms $f,g \colon A\to C$ that are proper are equal if $\mathcal{G}(f)=\mathcal{G}(g)$.
\end{proposition}

\begin{proof}
  Suppose $P,Q$ are the point sets of $\mathcal{G}(A),\mathcal{G}(C)$ and set $\alpha=\mathcal{G}(f)=\mathcal{G}(g)$. So by Proposition~\ref{normal} $\alpha\colon P\to Q$ is proper, hence by Definition~\ref{A4} it satisfies conditions (A4)-(A5). We have that $\mathcal{G}(g)(x_a)=\mathcal{G}(f)(x_a)=\mathcal{G}(f)(x_{a'})$ for each $a\in A$. In particular, $\mathcal{G}(f)[A]=\mathcal{G}(g)[A]$, and also $f[A]=g[A]$. To show that $f=g$, we take an arbitrary $a\in A$ and show that $f(a)=g(a)$. This is obvious if $a=0,1$, so we assume $a\neq 0,1$. Since $f$ is proper, there is a block $B$ that contains $a$ whose image under $f$ has more than $4$ elements. In showing that $f(a)=g(a)$ we consider two cases. 

  Suppose $f(a)$ is either $0,1$, and therefore that $g(a)$ is either $0,1$. We will show that $f(a)=0$ implies $g(a)=0$. The argument for $f(a)=1$ follows from this since $f(a')=0$ implies $g(a')=0$, hence $g(a)=1$. Let $e\in B$ with $f(e)\neq 0,1$. Set $b=e\wedge a'$ and $c=e'\wedge a'$. Then $a,b,c$ are pairwise orthogonal in $B$ and their join is 1. We have $a\neq 0,1$ and since $f(b)=f(e)\neq 0,1$ and $f(c)=f(e')\neq 0,1$, we have $b,c\neq 0,1$. Thus $a,b,c$ are atoms of an $8$-element Boolean subalgebra of $B$. We cannot have $g(a)=1$ since that gives $g(a')=0$, hence $g(b)=0$, contrary to $g(b)$ being either $f(b),f(b')$ and therefore not equal to $0,1$. 

  Next, suppose $f(a)\neq 0,1$. Since the image of $B$ under $f$ has more than $4$ elements, there is $e\in B$ with $f(e)\notin x_{f(a)}$. We claim there is an $8$-element subalgebra $S$ of $B$ that contains $a$ that is mapped isomorphically to an $8$-element Boolean subalgebra $T$ of~$C$. Since $a,a'$ are an atom and coatom of $S$ and $g(a)$ is equal to either $f(a)$ or $f(a)'$, it must be that $g(a)=f(a)$. To produce this $S$, we consider several cases. First, if $f(e)<f(a)$, then set $b=e\wedge a$. Then $f(b)=f(e)$, so $0<b<a$, and the subalgebra $S$ generated by $a$ and $b$ is an $8$-element subalgebra. Also the image of $S$ is an $8$-element subalgebra $T$ generated by $f(a), f(e)$, where $0<f(e)<f(a)$. A similar argument holds in the case of any comparability among $f(a),f(a)'$ and $f(e),f(e)'$. Suppose there is no such comparability, so a meet of one of $f(a),f(a)'$ with one of $f(e),f(e)'$ does not belong to $\{0,f(a),f(a)',f(e),f(e)',1\}$. Set $b=e\wedge a'$ and $c=e'\wedge a'$. Then $a,b,c$ are pairwise orthogonal and their join is 1, so they generate an $8$-element subalgebra $S$ of $B$. Since none of $f(a),f(b),f(c)$ is 0 or 1, the image of $T$ of $S$ is then an $8$-element Boolean subalgebra of~$C$. 
\end{proof}


Having shown that the functor $\mathcal{G}$ is faithful on proper morphisms, we turn attention to showing that it is full. Our earlier notion of directions will be key. 

\begin{definition}
  Suppose $A$ is a proper orthoalgebra, and let $\mathcal{G}=\mathcal{G}(A)$ is its hypergraph. Write $\oDir(\mathcal{G})$ for the orthoalgebra of directions of the orthodomain $\BSub(A)^*$ that is the set of points, lines, and planes of $\mathcal{G}$ with the obvious order. 
\end{definition}

Let's review some basics of directions when using the terms points, lines, and planes of $\mathcal{G}$ to refer to elements of height at most 3 in the orthodomain $\BSub(A)^*$. Basic elements are $\bot$ and the points $p$. A direction $d$ for a basic element assigns to each cover of that basic element either ${\uparrow}$ or $\downarrow$ (see also the remark below Corollary \ref{Cor:booleanobjects}). The direction $d$ is determined by its assignment on any given cover. The orthocomplementary direction $d'$ assigns exactly the opposite choice of $\uparrow$ or $\downarrow$ to each cover. The basic element $\bot$ has two directions, 0 and 1. The direction 0 assigns $\downarrow$ to each point, and the direction 1 assigns $\uparrow$ to each point. 

\begin{lemma}\label{deer}
  Let $d$ be a direction for a point $p$ of an orthohypergraph $\mathcal{G}$.
  \begin{enumerate}
    \item If $p$ is a corner point of a plane $t$, then $d$ takes the same value of $\,{\uparrow}\!$ or ${\downarrow}$ on all three lines of $t$ containing $p$. 
    \item If $p$ is a edge point of a plane $t$, then $d$ takes opposite values of $\,\uparrow\!$ and $\downarrow$ on the two lines of $t$ containing $p$. 
    \item If $l,m$ are two lines containing $p$ and $d$ takes opposite values of $\,\uparrow\!$ and $\downarrow$ on $l,m$, then there is a plane $t$ with $p$ as an edge point and $l,m$ the two lines of $t$ containing $p$. 
  \end{enumerate}  
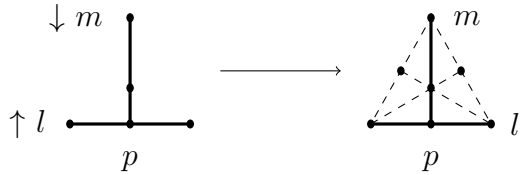
    \[
    \begin{tikzpicture}[xscale=.8,yscale=1]
    \draw[fill] (0,0) circle(1.5pt); \draw[fill] (1,0) circle(1.5pt); \draw[fill] (2,0) circle(1.5pt); \draw[fill] (1,1.414) circle(1.5pt); \draw[fill] (1,.48) circle(1.5pt); 
    \draw[very thick] (1,1.414)--(1,0); \draw[very thick] (0,0)--(2,0);
    \node at (-.7,0) {${\uparrow}\,\,l$}; \node at (.1,1.414) {${\downarrow}\,\,m$};
    \draw[thin,->] (2.5,.707)--(4.5,.707);
    \node at (1,-.5) {$p$};
    \node at (6,-.5) {$p$};
    \draw[fill] (5,0) circle(1.5pt); \draw[fill] (6,0) circle(1.5pt); \draw[fill] (7,0) circle(1.5pt); \draw[fill] (6,1.414) circle(1.5pt); \draw[fill] (6,.48) circle(1.5pt); \draw[fill] (5.5,.707) circle(1.5pt); \draw[fill] (6.5,.707) circle(1.5pt); 
    \draw[thin,dashed] (5,0)--(7,0)--(6,1.414)--(5,0)--(6.5,.707); \draw[very thick] (6,1.414)--(6,0); \draw[thin,dashed] (5.5,.707)--(7,0); \draw[very thick] (5,0)--(7,0);
    \node at (7.4,0) {$l$}; \node at (6.6,1.414) {$m$}; 
    \end{tikzpicture} \]
\end{lemma}
\begin{proof}
  (1) and (2) can be shown by calculating the directions for a $16$-element Boolean algebra $B$ and depicting this on the plane that is its hypergraph. Alternately, for such $B$, the corner points of its plane are the subalgebras $x_a$ where $a$ is an atom of $B$, and the edge points are the subalgebras $x_b$ where $b,b'$ are elements of height $2$ in $B$. A direction for the point essentially chooses one element from the point and provides $\uparrow$ or $\downarrow$ as a value for a line containing that point depending one whether the chosen element is an atom or coatom of the $8$-element subalgebra corresponding to that line. An atom $a$ of $B$ lies in three $8$-element subalgebras of $B$ and is an atom of each. An element $b$ of height $2$ in $B$ lies in two $8$-element subalgebras, and is an atom in one, and a coatom in the other (see for instance Figure \ref{fig1} where the element `1', corresponding to an atom of $B$, occurs in the first three subalgebras of the second row as an atom. Likewise, the element `12', corresponding to an element of height two in $B$,  occurs in the first and the last subalgebras of the second row as a coatom and an atom, respectively). (3) Definition~\ref{scary} of a direction provides that in the indicated circumstance $l\vee m=t$ exists and covers $l,m$. Then $t$ is a plane containing $l,m$, and the remainder follows from (2). 
\end{proof}

Now comes the key notion. In the rest of this section we assume that $A,C$ are proper orthoalgebras. Write $\mathcal{G}$ and $\mathcal{H}$ for their hypergraphs, with point sets $P$ and $Q$.

\begin{definition}\label{map}
  For $\alpha\colon P\to Q$ a proper hypergraph morphism, define $f_\alpha\colon \oDir(\mathcal{G})\to\oDir(\mathcal{H})$ as follows. 
  Let $f_\alpha$ map the directions $0,1$ to $0,1$. If $d$ is a direction for $p\in P$, let $l$ be a line through $p$ such that $\alpha(l)$ goes through $\alpha(p)$, and set 
  \[f_\alpha(d) = \mbox{ the direction for $\alpha(p)$ that takes value $d(l)$ at $\alpha(l)$}\text.\]
\end{definition}

The following results show that this is indeed well-defined.

\begin{lemma}\label{loi}
  If $\alpha \colon P\to Q$ is a proper hypergraph morphism, then for each $p\in P$ there is a line $l$ containing $p$ with $\alpha(l)$ covering $\alpha(p)$. 
\end{lemma}
\begin{proof}
  Condition (A4) of a proper hypergraph morphism says that $p$ is in a line or plane whose image contains a line. Suppose $p$ is in a line $l$ whose image contains a line. Then $\alpha(p)$ is a point and $\alpha(l)$ is a line containing that point, hence covering $\alpha(p)$. Suppose $p$ is in a plane $t$ whose image contains a line. In condition (A2), only cases (iii) or (iv) may apply to $\alpha(t)$. If $\alpha(p)=\bot$, then $p$ is in a line in this plane whose image is a point and therefore covers $\alpha(p)$, and otherwise $p$ is in a line in this plane whose image is a line and therefore covers $\alpha(p)$. 
\end{proof}

For the following result, recall that each basic element $p$ has exactly two directions, and that if one direction for the element is $d$, then the other $d'$ is formed by reversing the directions of all values $d(l)$ for covers of $l$ of the basic element. So if $l,m$ cover $p$, then one direction for $p$ takes the same value at $l,m$ precisely when both directions for $p$ take the same value at $l,m$. 

\begin{lemma}
  Suppose $\alpha\colon P\to Q$ is a proper hypergraph morphism and $p\in P$ belongs to lines $l,m$ where $l'=\alpha(l)$ and $m'=\alpha(m)$ cover $\alpha(p)$. If $d$ is a direction of $\mathcal{G}$ for $p$ and $e$ is a direction of $\mathcal{H}$ for $\alpha(p)$, then $d(l)=e(l')$ implies that $d(m)=e(m')$. 
\end{lemma}
\begin{proof}
  Suppose first that $\alpha(p)=\bot$. Since the directions of $\mathcal{H}$ for $\bot$ are 0, 1 and each takes the same value on all covers of $\bot$, that is, on all points of $Q$, we must show that a direction $d$ for $p$ takes the same value on the lines $l,m$ that contain $p$. If not, then by Lemma~\ref{deer}(3), there is a plane $t$ containing $l,m$ and having their intersection $p$ as an edge point. Consider condition (A2) together with the assumption $\alpha(p)=\bot$. The only possibilities that have an edge point of the plane undefined are (i) and (ii), and both have at least one of the two lines $l,m$ containing $p$ mapped to $\bot$, and hence not covering $\alpha(p)$. 

  Suppose that $\alpha(p)=q$ is a point of $Q$. So for $l'$ and $m'$ to cover $\alpha(p)$ we have that $l',m'$ are lines of $Q$ that contain $q$. Our result will follow if we show that $d(l)= d(m)$ iff $e(l')=e(m')$. If $d(l)$ and $d(m)$ take opposite values, then by Lemma~\ref{deer}(3) we have that $l, m$ lie in a plane $t$ with their intersection $p$ an edge point of this plane. Considering the possibilities for $\alpha(t)$ given by (A2), only case (iv) can apply. So $\alpha$ maps $t$ isomorphically to a plane $t'$, hence with $l',m'$ distinct lines of $t'$ with their intersection point $q=\alpha(p)$ an edge point of $t'$. Then by part 2 of Lemma~\ref{deer}, we have that $e$ takes opposite values at $l',m'$. If $e(l')$ and $e(m')$ take opposite values, then by part 3 of Lemma~\ref{deer} we have that $l',m'$ are distinct lines of a plane $t'$ with their intersection $q$ an edge point of this plane. By condition (A3), $l$ and $m$ lie in a plane $t$ of $\mathcal{G}$ that is mapped isomorphically by $\alpha$ to $t'$, and hence with the intersection $p$ of $l,m$ being an edge point of $t$. So $d(l)$ and $d(m)$ take opposite values by Lemma~\ref{deer}(2). 
\end{proof}

Together, Lemmas~\ref{loi} and~\ref{deer} show that Definition~\ref{map} is a valid definition of a mapping $f_\alpha \colon \oDir(\mathcal{G})\to\oDir(\mathcal{H})$. We now proceed to establish properties of this map. 

\begin{proposition}
  If $\alpha\colon P\to Q$ is a proper hypergraph morphism, then $f_\alpha\colon \oDir(\mathcal{G})\to\oDir(\mathcal{H})$ is an orthoalgebra morphism. 
\end{proposition}
\begin{proof}
  By Definition~\ref{map}, $f_\alpha$ preserves 0 and 1. If $d$ is a direction for a basic element $p$ and $l$ is a cover of $p$ with $\alpha(l)$ a cover of $\alpha(p)$, then we have that $d$ and its orthocomplement $d'$ take opposite values at $l$. Then by Definition~\ref{map}, $f_\alpha(d)$ and $f_{\alpha}(d')$ are directions for $\alpha(p)$ that take opposite values at the cover $\alpha(l)$ of $\alpha(p)$. Therefore $f_\alpha(d)$ and $f_\alpha(d')$ are orthocomplements. So $f_\alpha$ also preserves orthocomplementation. 

  It remains to show that $f_\alpha$ preserves orthogonal joins. For the rest of the proof, assume that $d$ is a direction for the basic element $p$, that $e$ is a direction for the basic element $q$, and that $d$ is orthogonal to $e$. We must show that $f_\alpha(d)$ is orthogonal to $f_\alpha(e)$ and that $f_\alpha(d\oplus e)=f_\alpha(d)\oplus f_\alpha(e)$. This requires distinguishing several cases. 
  \begin{description}
    \item[Case 1]
      At least one of $p,q$ is $\bot$. Since the directions for $\bot$ are 0 and 1, since $d$ is orthogonal to $e$, it follows that at least one of $d$ and $e$ is 0. Since $f_\alpha$ preserves 0, it then follows that $f_\alpha(d)$ is orthogonal to $f_\alpha(e)$ and that $f_\alpha(d\oplus e)=f_\alpha(d)\oplus f_\alpha(e)$. 
  \end{description}
  In the remainder, assume neither $p$ nor $q$ equals $\bot$, and therefore both are points of $P$. 
  \begin{description}
    \item[Case 2] $p=q$. Since $d,e$ are directions for the same point $p$ and are orthogonal, it follows that they are orthocomplements. Since $f_\alpha$ preserves orthocomplements, it follows that $f_\alpha(d)$ is orthogonal to $f_\alpha(e)$ and that $f_\alpha(d\oplus e)=f_\alpha(1)=1=f_\alpha(d)\oplus f_\alpha(e)$. 
\end{description}
    So we now have the situation where $p$ and $q$ are distinct points of $P$. Since $d$ is orthogonal to $e$, it follows that $p$ and $q$ lie on a line $l$. Let $r$ be the third point on $l$. Then since $d$ is orthogonal to $e$ we have that $d(l)$ and $e(l)$ have the value ${\downarrow}$, and that $d\oplus e$ is the direction for $r$ with $(d\oplus e)(l)$ having value ${\uparrow}$. The cases that follow will all assume this setup. 
\begin{description}
    \item[Case 3] $\alpha(l)$ properly contains $\alpha(p),\alpha(q)$. From (A1), this implies that $\alpha(l)=l'$ is a line, and therefore $\alpha(p)=p'$, $\alpha(q)=q'$, and $\alpha(r)=r'$ are distinct points comprising the line. Definition~\ref{map} then gives that $f_\alpha(d)$ is the direction for $p'$ with $f_\alpha(d)(l')={\downarrow}$, that $f_\alpha(e)$ is the direction for $q'$ with $f_\alpha(e)(l')={\downarrow}$, and that $f_\alpha(d\oplus e)$ is the direction for $r'$ with $f_\alpha(d\oplus e)(l')={\uparrow}$. Then $f_\alpha(d)$ and $f_\alpha(e)$ are orthogonal and $f_\alpha(d)\oplus f_\alpha(e)$ is the direction for $r'$ taking value ${\uparrow}$ at $l'$. Since $f_\alpha(d\oplus e)$ and $f_\alpha(d)\oplus f_\alpha(e)$ are directions for $r'$ taking the same value at $l'$, they are equal. 
\end{description}
    Suppose $\alpha(l)$ does not properly contain $\alpha(p),\alpha(q)$. Up to symmetry, there are several possibilities: (i) $\alpha(p)=\alpha(q)=\bot$, (ii) $\alpha(p)=\bot$, $\alpha(q)=q'$, and (iii) $\alpha(p)=\alpha(q)=p'$. In any of these cases, since $\alpha$ is proper, by (A5) there is a plane $t$ of $P$ with $\alpha(t)$ properly containing $\alpha(p),\alpha(q)$. Hence there is a point $s\in t$ with $\alpha(s)$ not in $l'$. Our remaining cases include this setup. 
\begin{description}
    \item[Case 4] $\alpha(p)=\alpha(q)=\bot$. By (A1) also $\alpha(r)=\bot$. Since there is a point $s$ in $t$ with $\alpha(s)\neq\bot$, the situation must be as indicated as in (A2.ii) with $p,q,r$ forming the bottom of the plane. We then let $u$ be the top of the plane, and note that $\alpha(u)\neq\bot$. There are two different possibilities, that $p,q$ are both corners of $t$, and that one of $p,q$ is an edge point of $t$, say $p$ is an edge point. These lead to the two situations depicted below. 
    
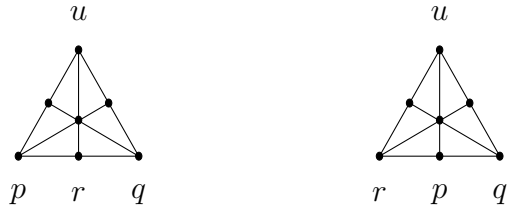
\[\begin{tikzpicture}[xscale=.8,yscale=1]
    \draw[fill] (0,0) circle(1.5pt); \draw[fill] (1,0) circle(1.5pt); \draw[fill] (2,0) circle(1.5pt); \draw[fill] (1,1.414) circle(1.5pt); \draw[fill] (1,.48) circle(1.5pt); \draw[fill] (.5,.707) circle(1.5pt); \draw[fill] (1.5,.707) circle(1.5pt); 
    \draw[thin] (0,0)--(2,0)--(1,1.414)--(0,0)--(1.5,.707); 
    \draw[thin] (1,1.414)--(1,0); \draw[thin] (.5,.707)--(2,0);
    \node at (0,-.5) {$p$}; \node at (1,-.5) {$r$}; \node at (2,-.5) {$q$}; \node at (1,1.9) {$u$};
    \end{tikzpicture}
    \hspace{15ex}
    \begin{tikzpicture}[xscale=.8,yscale=1]
    \draw[fill] (0,0) circle(1.5pt); \draw[fill] (1,0) circle(1.5pt); \draw[fill] (2,0) circle(1.5pt); \draw[fill] (1,1.414) circle(1.5pt); \draw[fill] (1,.48) circle(1.5pt); \draw[fill] (.5,.707) circle(1.5pt); \draw[fill] (1.5,.707) circle(1.5pt); 
    \draw[thin] (0,0)--(2,0)--(1,1.414)--(0,0)--(1.5,.707); 
    \draw[thin] (1,1.414)--(1,0); \draw[thin] (.5,.707)--(2,0);
    \node at (0,-.5) {$r$}; \node at (1,-.5) {$p$}; \node at (2,-.5) {$q$}; \node at (1,1.9) {$u$};
    \end{tikzpicture}\]
    Note that $l$ is the bottom line of each plane. In each case, let $i$ be the line containing $p,u$, let $j$ be the line containing $q,u$, and let $k$ be the line containing $r,u$. Since $\alpha(u)\neq\bot$, in each case $\alpha(i)=i'$, $\alpha(j)=j'$ and $\alpha(k)=k'$ cover $\bot$. In a plane, a direction for a corner point takes the same value on all three lines containing the point, and a direction for an edge point takes opposite values on the two lines containing the point. Since $d(l)={\downarrow}$, $e(l)={\downarrow}$, $(d\oplus e)(l)={\uparrow}$, this yields the following. 
    \begin{align*}
      \mbox{Situation at left: }&\,\,\,d(i)={\downarrow},\,\,\, e(j)={\downarrow},\,\, \mbox{ and }\,\,(d\oplus e)(k)={\downarrow}\\ 
      \mbox{Situation at right: }&\,\,\,d(i)={\uparrow},\,\,\, e(j)={\downarrow},\,\, \mbox{ and }\,\,(d\oplus e)(k)={\uparrow}
    \end{align*}
    Definition~\ref{map} provides that $f_\alpha(d)(i')$ takes the same value of $\uparrow$ or $\downarrow$ as $d(i)$. But since $\alpha(p)$, $\alpha(q)$, and $\alpha(r)$ equal $\bot$, each of $f_\alpha(d)$, $f_\alpha(e)$, and $f_\alpha(d\oplus e)$ is either 0 or 1. Together with the above information, this provides the following. 
    \begin{align*}
      \mbox{Situation at left: }&\,\,\,f_\alpha(d)=0,\,\,\, f_\alpha(e)=0,\,\, \mbox{ and }\,\,f_\alpha(d\oplus e)=0\\ 
      \mbox{Situation at right: }&\,\,\,f_\alpha(d)=1,\,\,\, f_\alpha(e)=0,\,\, \mbox{ and }\,\,f_\alpha(d\oplus e)=1
    \end{align*}
    In each case, $f_\alpha(d)$ and $f_\alpha(e)$ are orthogonal, and $f_\alpha(d\oplus e)=f_\alpha(d)\oplus f_\alpha(e)$. 

  \item[Case 5] 
    $\alpha(p)=\bot$ and $\alpha(q)=q'$. From (A1), $\alpha(r)$ is also equal to the point $q'$. The existence of the point $s$ in $t$ with $\alpha(s)$ different from $\alpha(p),\alpha(q)$ implies (A2.iii) applies. There are then two possibilities depending on whether $q$ or $r$ is an edge point of $t$. 
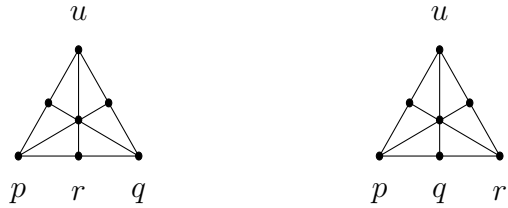
    \[\begin{tikzpicture}[xscale=.8,yscale=1]
    \draw[fill] (0,0) circle(1.5pt); \draw[fill] (1,0) circle(1.5pt); \draw[fill] (2,0) circle(1.5pt); \draw[fill] (1,1.414) circle(1.5pt); \draw[fill] (1,.48) circle(1.5pt); \draw[fill] (.5,.707) circle(1.5pt); \draw[fill] (1.5,.707) circle(1.5pt); 
    \draw[thin] (0,0)--(2,0)--(1,1.414)--(0,0)--(1.5,.707); 
    \draw[thin] (1,1.414)--(1,0); \draw[thin] (.5,.707)--(2,0);
    \node at (0,-.5) {$p$}; \node at (1,-.5) {$r$}; \node at (2,-.5) {$q$}; \node at (1,1.9) {$u$};
    \end{tikzpicture}
    \hspace{15ex}
    \begin{tikzpicture}[xscale=.8,yscale=1]
    \draw[fill] (0,0) circle(1.5pt); \draw[fill] (1,0) circle(1.5pt); \draw[fill] (2,0) circle(1.5pt); \draw[fill] (1,1.414) circle(1.5pt); \draw[fill] (1,.48) circle(1.5pt); \draw[fill] (.5,.707) circle(1.5pt); \draw[fill] (1.5,.707) circle(1.5pt); 
    \draw[thin] (0,0)--(2,0)--(1,1.414)--(0,0)--(1.5,.707); 
    \draw[thin] (1,1.414)--(1,0); \draw[thin] (.5,.707)--(2,0);
    \node at (0,-.5) {$p$}; \node at (1,-.5) {$q$}; \node at (2,-.5) {$r$}; \node at (1,1.9) {$u$};
    \end{tikzpicture}\]
    Write $u$ for the top of the plane and $i,j,k$ for the lines containing $p,u$, and $q,u$, and $r,u$ respectively as in the previous case. In each situation we have that $d(l)={\downarrow}$, so $d(i)={\downarrow}$, and hence the direction $f_\alpha(d)$ for $\bot$ takes value ${\downarrow}$ at the point $\alpha(i')$, and therefore $f_\alpha(d)=0$. So in each situation we have that $f_\alpha(d)$ is orthogonal to $f_\alpha(e)$ and $f_\alpha(d)\oplus f_\alpha(e)=f_\alpha(e)$. It remains to show that $f_\alpha(d\oplus e)=f_\alpha(e)$. Using the fact that $e(l)={\downarrow}$ and $(d\oplus e)(l)={\uparrow}$, as well as our description of how arrows work at corner and edge points of a plane, we have the following. 
    \begin{align*}
      \mbox{Situation at left: }& e(j) = {\downarrow}\,\,\, \mbox{ and }\,\,\, (d\oplus e)(k) = {\downarrow}\\
      \mbox{Situation at right: }& e(j) = {\uparrow}\,\,\, \mbox{ and }\,\,\, (d\oplus e)(k) = {\uparrow}
    \end{align*}
   Since (A2.iii) applies, it follows that $\alpha(j)=j'$ and $\alpha(k)=k'$ are the same line containing $q'$.
    In the first situation, $f_\alpha(e)(j')={\downarrow}$ and $f_\alpha(d\oplus e)(k')={\downarrow}$, whereas in the second situation $f_\alpha(j')={\uparrow}$ and $f_\alpha(d\oplus e)(k')={\uparrow}$. Since $j'=k'$, in either situation $f_\alpha(e)=f_\alpha(d\oplus e)$ as required. 

  \item[Case 6]
    $\alpha(p)=\alpha(q)= p'$. Then (A1.ii) applies to $\alpha(l)$, so $\alpha(r)=\bot$. Since there is a point $s$ in $t$ with $\alpha(s)$ different from $p'$ and $\bot$, case (A2.iii) applies with $r$ the corner point of $t$ mapped to $\bot$. Then one of $p,q$ is a corner point and the other an edge point. By symmetry we need only consider the situation where $p$ is an edge point. 
    
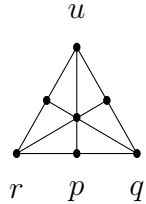
\[\begin{tikzpicture}[xscale=.8,yscale=1]
    \draw[fill] (0,0) circle(1.5pt); \draw[fill] (1,0) circle(1.5pt); \draw[fill] (2,0) circle(1.5pt); \draw[fill] (1,1.414) circle(1.5pt); \draw[fill] (1,.48) circle(1.5pt); \draw[fill] (.5,.707) circle(1.5pt); \draw[fill] (1.5,.707) circle(1.5pt); 
    \draw[thin] (0,0)--(2,0)--(1,1.414)--(0,0)--(1.5,.707); 
    \draw[thin] (1,1.414)--(1,0); \draw[thin] (.5,.707)--(2,0);
    \node at (0,-.5) {$r$}; \node at (1,-.5) {$p$}; \node at (2,-.5) {$q$}; \node at (1,1.9) {$u$};
    \end{tikzpicture}\]
    Write $u$ for the top of the plane and $i,j,k$ for the lines containing $p,u$, and $q,u$, and $r,u$ respectively as in the previous case. Note that (A2.iii) implies that $\alpha(u)=u'$ is a point distinct from $p'$. Our considerations for the way arrows behave in a plane in conjunction with $d(l)={\downarrow}$, $e(l)={\downarrow}$ and $(d\oplus e)(l)={\uparrow}$ imply the following. 
    \[d(i)={\uparrow},\,\,\, e(j)={\downarrow},\,\,\, \mbox{ and }\,\,\, (d\oplus e)(k)={\uparrow}\]
    Since $\alpha(p)=\alpha(q)=p'$, then $f_\alpha(d)$ and $f_\alpha(e)$ are directions for $p'$. Also, $\alpha(i)=i'$ and $\alpha(j)=j'$ both contain the distinct points $p',u'$, so $i'=j'$. It follows from the above that $f_\alpha(d)$ and $f_\alpha(e)$ take opposite values ${\uparrow}$ and ${\downarrow}$ respectively at the line $i'=j'$. Since they are directions for the same point $p'$, they are orthocomplements. Thus $f_\alpha(d)\oplus f_\alpha(e)=1$. We have that $d\oplus e$ is a direction for $r$ and $\alpha(r)=\bot$. Also, $\alpha(k)=k'$ is the point $u'$. Since $d\oplus e$ takes value ${\uparrow}$ at $k$, then $f_\alpha(d\oplus e)$ takes value ${\uparrow}$ at $u'$. Thus $f_\alpha(d\oplus e)=1$. This completes the proof of this case, and of the proposition. 
    \qedhere
  \end{description}
\end{proof}

Recall from Theorem~\ref{hhh} that for a proper orthoalgebra $A$, there is an orthoalgebra isomorphism from $A$ to $\oDir(\BSub(A))$ taking an element $a\in A$ to the direction $d_a$ given by Definition~\ref{xes}. Suppose $\mathcal{G}$ is the hypergraph of $A$. Then as noted in Theorem~\ref{qwqq}, and as used throughout this section, the orthodomain $\BSub(A)$ can be reconstructed from $\mathcal{G}$.
 
\begin{theorem}\label{drew}
  Suppose $A$ and $C$ are orthoalgebras with hypergraphs $\mathcal{G}$ and $\mathcal{H}$ whose point sets are $P$ and $Q$. Suppose $\alpha\colon P\to Q$ is a proper hypergraph morphism. Then the orthoalgebra morphism $f_\alpha\colon \oDir(\mathcal{G})\to\oDir(\mathcal{H})$ induces an orthoalgebra morphism $g_\alpha\colon A\to C$ where 
  \[g_\alpha(a) = c\quad \mbox{ if }\quad f_\alpha(d_a) = d_c\]
  Further, the induced hypergraph morphism $\mathcal{G}(g_\alpha)$ is equal to $\alpha$, and therefore $g_\alpha$ is proper. 
\end{theorem}
\begin{proof}
  Since $g_\alpha$ is given by the composite of the orthoalgebra morphism $f_\alpha$ with the orthoalgebra isomorphisms from $A$ to $\oDir(\mathcal{G})$ and from $\oDir(\mathcal{H})$ to $C$, it is an orthoalgebra morphism. To see that $\mathcal{G}(g_\alpha)=\alpha$, suppose $x_a$ is point of~$P$. Let $g_\alpha(a)=c$. Then $x_c$ is either $\bot$ or a point of $Q$. By Definition~\ref{mii}, we have that $\mathcal{G}(g_\alpha)(x_a)=x_c$. Since $d_a$ is a direction for $x_a$, Definition~\ref{map} gives that $f_\alpha(d_a)$ is a direction for $\alpha(x_a)$. But $g_\alpha(a)=c$ implies by the definition of $g_\alpha$ that $f_\alpha(d_a)=d_c$. But $d_c$ is a direction for $x_c$. Thus $f_\alpha(d_a)$ is a direction for $\alpha(x_a)$ and a direction for $x_c$, hence $\alpha(x_a)=x_c$. Thus $\mathcal{G}(g_\alpha)=\alpha$. That this implies that $g_\alpha$ is proper is given by Proposition~\ref{trek}. 
\end{proof}

A functor $F\colon \mathcal{C}\to\mathcal{D}$ is an equivalence of categories when it is full, faithful, and essentially surjective. Fulness and faithfulness mean that for any objects $A$ and $C$ of $\mathcal{C}$, there is a bijection $F\colon \mathcal{C}(A,C)\to\mathcal{D}(F(A),F(C))$ of homsets. Essential surjectivity means that each object $D\in\mathcal{D}$ is isomorphic to $F(C)$ for some object $C$ of $\mathcal{C}$. We have seen various obstructions to providing an equivalence between the categories of orthoalgebras and orthohypergraphs. These include the fact that a one and two-element orthoalgebra have the same hypergraph, and various difficulties involving morphisms between orthoalgebras where the image of a block might be small. Essentially, the difficulty arises from the fact that a 4-element Boolean algebra has 2 automorphisms while its 1-element hypergraph only has one. This is the only difficulty.

\begin{definition}
  Write $\cat{OA_p}(A,C)$ for the collection of orthoalgebra morphisms that are proper from one orthoalgebra $A$ to another $C$,
  and $\cat{OH_p}(\mathcal{G},\mathcal{H})$ for the collection of proper hypergraph morphisms from one hypergraph $\mathcal{G}$ to another $\mathcal{H}$.
  Write $\cat{OA_i}$ for the category of proper orthoalgebras and injective orthoalgebra morphisms. 
  Write $\cat{OH_i}$ for the category of orthohypergraphs in which every point lies on a line and injective orthohypergraph morphisms.
\end{definition}

\begin{theorem}
  The functor $\mathcal{G} \colon \cat{OA} \to \cat{OH}$ has the following properties:
  \begin{itemize}
    \item it is essentially surjective on objects;
    \item it is injective on objects with the exception of 1- and 2-element orthoalgebras;
    \item it is full on proper morphisms: $\mathcal{G} \colon \cat{OA_p}(A,C) \to \cat{OH_p}(\mathcal{G}(A),\mathcal{G}(C))$ is surjective;
    \item it is faithful on proper morphisms: $\mathcal{G} \colon \cat{OA_p}(A,C) \to \cat{OH_p}(\mathcal{G}(A),\mathcal{G}(C))$ is injective.
  \end{itemize}
\end{theorem}
\begin{proof}
  That $\mathcal{G} \colon \cat{OA} \to \cat{OH}$ is a functor is Proposition~\ref{pkl}. By definition of orthohypergraphs, $\mathcal{G}$ is essentially surjective. A proper orthoalgebra $A$ is isomorphic to the orthoalgebra of directions of $\BSub(A)^*$, and hence determined by its hypergraph $\mathcal{G}(A)$. As described in Remarks~\ref{kino} and~\ref{ruut}, non-trivial proper orthoalgebras are also determined up to isomorphism by their posets of Boolean subalgebras of height at most 3, and hence by their hypergraphs. So $\mathcal{G}$ is essentially injective on non-trivial orthoalgebras.  Proposition~\ref{cvb} proves $\mathcal{G} \colon \cat{OA_p}(A,C)\to\cat{OH_p}(\mathcal{G}(A),\mathcal{G}(C))$ injective, and Theorem~\ref{drew} proves   it surjective. 
\end{proof}

\begin{corollary}    
  The functor $\mathcal{G}$ restricts to an equivalence $\cat{OA_i} \simeq \cat{OH_i}$. 
\end{corollary}
\begin{proof}
  Observe that an orthoalgebra $A$ has no small blocks if and only if every point of $\mathcal{G}(A)$ lies on a line, and combine the previous theorem with Lemma~\ref{injectionssurjections}.
\end{proof}

\section{Concluding remarks}\label{sec:conclusion}

We have introduced a new method to describe orthoalgebras. Several previous methods have existed for about 50 years \cite{Kalmbach}. These include pasted families of Boolean algebras, which describe an orthostructure by specifying its maximal Boolean subalgebras (blocks) and their intersections; orthogonality relations, which give the elements of the orthostructure directly and a relation of orthogonality; and Greechie diagrams used for chain-finite orthostructures, where the structure is described via a hypergraph whose points are atoms of the structure. 

This new description is based on the poset of Boolean subalgebras of an orthoalgebra. We emphasize that it is the abstract structure of this poset that is required, not a knowledge of the actual Boolean subalgebras and their containments; it is enough to know `how the parts fit together', we do not need to know the parts themselves. Any non-trivial orthoalgebra can be reconstructed from its poset of Boolean subalgebras via a technique called directions. This can be further refined to use only of the portion of this poset of Boolean subalgebras of height 3 or less, which can be described as a type of hypergraph. For orthoalgebras, this hypergraph requires points, lines, and planes, while for orthomodular posets points and lines of 3 points each are sufficient. The idea behind this reconstruction of an orthoalgebra from a hypergraph is as follows: each point of the hypergraph yields two elements of the orthoalgebra to be reconstructed, and the two directions for each point say whether each element sits as an atom or coatom in each $8$-element Boolean algebra that contains it. This suffices to determine the orthoalgebra up to isomorphism. 

In contrast to other methods of representing orthostructures, a categorical correspondence in the setting of hypergraphs is relatively elegant. Morphisms between hypergraphs are certain partial mappings between their points that satisfy basic conditions describing how homomorphisms work on Boolean algebras with at most $16$ elements, as well as one simple, but more specific axiom. This provides a functor $\mathcal{G}$ from the category $\cat{OA}$ of orthoalgebras and their morphisms to the category $\cat{OH}$ of orthohypergraphs and their morphisms. On objects this is essentially surjective, and even injective when excepting trivial orthoalgebras. Basic obstacles involving automorphisms of 4-element Boolean algebras prevent an equivalence between any obvious modification of these categories. However, a stronger result obtains when restricting to proper morphisms, that for orthoalgebras bypass the difficulty of the image of a block being a 4-element Boolean algebra. On these proper morphisms, $\mathcal{G}$ is full and faithful. Unfortunately proper morphisms do not compose and therefore do not form a category in their own right. Informally, modulo some minor exceptions for trivial orthoalgebras and morphisms where some blocks have small images, the categories $\cat{OA}$ and $\cat{OH}$ are `nearly' equivalent. 

A description of sorts is given for the posets, and therefore the hypergraphs, that arise as posets of Boolean subalgebras of an orthoalgebra. Several of the conditions required are relatively simple using in an essential way the characterization of the poset of Boolean subalgebras of a Boolean algebra in terms of partition lattices \cite{Gratzer}. However, a higher order condition involving the existence of a sufficient supply of directions for the poset is also required. This leads to the following question. We believe a positive solution to this would be of substantial benefit in moving this direction of research forward to allow use of hypergraph techniques to problems in orthostructures previously addressed only via techniques similar to Greechie diagrams \cite{Navara,N:handQL}. 

\begin{problem}
  Characterize the hypergraphs that arise as orthohypergraphs of Boolean algebras, orthomodular lattices, orthomodular posets, and of orthoalgebras. 
\end{problem}

Aside from its basic interest and potential applicability, the results here are directly related to several lines of research. They are a direct continuation of work begun by Sachs \cite{Sachs} and continued by Gr\"atzer \textit{et.\ al.}~\cite{Gratzer} on the connection between Boolean algebras and their lattices of subalgebras. Indeed, our key notion of directions requires the analysis originally given by Sachs. Results here are new even in the Boolean context. They provide a more direct reconstruction of a Boolean algebra from its poset of subalgebras via directions rather than by the colimit approach of \cite{Gratzer}; they introduce hypergraph techniques that simplify descriptions; and they give a categorical treatment that involves morphisms. 

The results here are also directly related to the topos approach to quantum mechanics of Isham \textit{et.\ al.}~\cite{Isham}. In this line of investigation, the poset of Boolean subalgebras of the orthomodular lattice of closed subspaces of a Hilbert space is the central ingredient used to construct various sheaves. In \cite{HardingNavara} it was shown that even in the setting of orthomodular lattices, this poset determines the orginal orthomodular lattice. Various studies have continued this investigation to the matter of connecting the poset of abelian subalgebras of a von Neumann algebra or $C^*$-algebra to the given von Neumann or $C^*$-algebra \cite{Doring,Hamhalter,Lindenhovius,HeunenReyes}. There is a fundamental obstacle in this line of investigation given that there exist non-isomorphic von Neumann algebras with isomorphic Jordan structure. Also a primary barrier is the fact that only for special classes of $C^*$-algebras is the poset of its abelian C*-subalgebras atomic. Bearing this in mind, we ask the following broadly phrased question. 

\begin{problem}
Develop the connection between $C^*$-algebras and their posets of abelian subalgebras using hypergraph techniques. 
\end{problem}

\end{document}